\documentclass[a4paper,11pt]{article}

\usepackage{amssymb}
\usepackage{amsfonts,dsfont}
\usepackage{amsmath}
\usepackage{amsthm}
\usepackage{cite}
 \usepackage{enumitem}
\numberwithin{equation}{section}

\usepackage{xcolor}

\newcommand{\R}{\mathbb{R}}

\newcommand{\N}{\mathbb{N}}

\newcommand{\cuad}{{\sqcap\kern-.68em\sqcup}}

\newcommand{\norm}[1]{\|#1\|}

\newtheorem{definition}{Definition}[section]
\newtheorem{theorem}[definition]{Theorem}
\newtheorem{proposition}[definition]{Proposition}

\newtheorem{lemma}[definition]{Lemma}
\newtheorem{corollary}[definition]{Corollary}
\newtheorem{remark}{Remark}[section]

\renewcommand{\phi}{\varphi}

\newcommand{\cD}{{\mathcal D}}

\newcommand{\cH}{{\mathcal H}}
\newcommand{\cI}{{\mathcal I}}
\newcommand{\cJ}{{\mathcal J}}
\newcommand{\cK}{{\mathcal K}}

\newcommand{\cM}{{\mathcal M}}

\newcommand{\cO}{{\mathcal O}}
\newcommand{\cP}{{\mathcal P}}
\newcommand{\cQ}{{\mathcal Q}}

\newcommand{\cS}{{\mathcal S}}
\newcommand{\cT}{{\mathcal T}}

\newcommand{\bX}{{\mathbb X}}
\newcommand{\bL}{{\mathbb L}}

\newcommand{\dist}{{\rm dist}}

\newcommand{\loc}{{\rm loc}}

 \allowdisplaybreaks

\begin{document}

\begin{center}
{\bf    Maximal  hypersurfaces with prescribed   light-like cones
\\[2mm]   in  Lorentz-Minkowski space 
  }

 \bigskip
 \bigskip

  {\small  Huyuan Chen\footnote{chenhuyuan@yeah.net, chenhuyuan@simis.cn} \qquad   Ying Wang\footnote{yingwang00@126.com} \qquad  Feng Zhou\footnote{fzhou@math.ecnu.edu.cn } 
 \bigskip

  $ ^1$   Center for Mathematics and Interdisciplinary Sciences, Fudan University, \\ 
Shanghai 200433, China\\[2pt]
Shanghai Institute for Mathematics and Interdisciplinary Sciences,\\ 
 Shanghai 200433, China\\[10pt] 
   $ ^2$ School of Information Management and Mathematics, Jiangxi University \\
of  Finance and Economics,  Nanchang, Jiangxi 330032,  PR China\\[10pt]

 $^3$ {\small Center for PDEs, School of Mathematical Sciences, East China Normal University,\\
Shanghai Key Laboratory of PMMP, Shanghai 200062, PR China }  \\[16pt]

   }

\end{center}

 \begin{abstract}
  \small
 
 The purpose in this paper is to study the maximal  hypersurfaces with multiple  light-cones
  in  Lorentz-Minkowski space by considering the weak  solutions to the mean curvature equation with multiple Dirac masses in 
$N$-dimensional Lorentz–Minkowski space
$$-\nabla\cdot\Big(\frac{\nabla u }{\sqrt{1-|\nabla u |^2}} \Big)=\sum^{m_0}_{j=1}\alpha_j\delta_{p_j}\quad {\rm in}\ \cD'(\R^N)$$
 for  $N\geq 2$ and $m_0\geq 2$. 
  Such solutions are constructed via an approximation procedure, using regular solutions with smooth sources that converge weakly to the Dirac measures. 
  
   When $N \geq 3$, a light-cone singular solution with decaying at infinity can also be viewed as a critical point of the associated energy functional.    However, this variational characterization fails for $N = 2$, as the energy functional diverges in this case.

   For $N \geq 3$, we conduct a comprehensive analysis of equations involving positive Dirac mass sources and resolve two open questions raised in \cite{BAP}: (i) whether the variational solution coincides with a weak solution, and (ii) how to strengthen the regularity assumptions to ensure the solution is classical. Furthermore, when both positive and negative Dirac masses are present, we establish a sharper sufficient condition for $C^2$ regularity.
   
   Finally, we extend the construction to include maximal hypersurfaces with infinitely many light-cones.

\end{abstract}

 {\footnotesize  \tableofcontents} \vspace{1mm}
 {\small
  \vspace{1mm}

}

\setcounter{equation}{0}
\section{Introduction}

Denote  $\bL^{N+1}$ the Minkowski space, which is $\R^{N+1}$ equipped with the Lorentzian  metric $\displaystyle ds^2=\sum_{i=1}^N dx_i^2-dx_{N+1}^2$ and the inner product by $\langle\cdot,\cdot\rangle$. The light cone at $\xi_0=(x_0,t_0)\in \bL^{N+1}$ is defined by 
 \begin{equation}\label{cone-sin}
 C_{\xi_0}=\Big\{\xi\in \bL^{n+1}:\ \langle \xi-\xi_0,\xi-\xi_0\rangle=0\Big\}. 
   \end{equation}
Let $\cS$ be an $N$-dimensional hypersurface in $\bL^{N+1}$, always represented as the graph of a function $u\in C^{0,1}(\Omega)$, where $\Omega$ is a domain of $\R^N$. The hypersurface $\cS$ is called \smallskip

{\it weakly spacelike }\quad  if $|Du|\leq 1$ a.e. in $\Omega$;\smallskip

 {\it spacelike } \quad  if $|u(x)-u(y)|<|x-y|$ whenever  $x,y\in\Omega,\ x\not=y$ and the line segment $\overline{xy}\subset \Omega$;\smallskip
 
 {\it strictly spacelike } \quad if $\cS$ is spacelike, $u\in C^1(\Omega)$ and $|Du|<1$ in $\Omega$.  
\medskip

Maximal hypersurfaces occupy a fertile intersection of elliptic partial differential equations--despite their Lorentzian origins--geometric analysis, and mathematical relativity--the Born-infeld model. They provide a setting in which many techniques from minimal surface theory remain applicable, yet they exhibit striking differences: the core constraint $|\nabla u|<1$, the presence of a hyperbolic Gauss map, and their significance in the context of initial data sets for the Einstein equations. These features make maximal hypersurfaces a fundamental object of study in both differential geometry and general relativity.

A central problem is the construction of maximal hypersurfaces in Lorentz--Minkowski space, either by studying the area functional $\int_{\Omega} \sqrt{1-|\nabla u|^2}\,dx$ or by solving the associated Euler--Lagrange equation, namely the type of the mean curvature equation
\begin{align}\label{eq 1.0-0}
\cM_0u(x):= -\nabla\cdot\Big(\frac{\nabla u(x)}{\sqrt{1-|\nabla u(x)|^2}} \Big)=0\quad {\rm for}\ x\in   \R^N.   
\end{align}
 Calabi \cite{C68} and Cheng-Yau \cite{CY76} provided a fundamental result that 
 $$\text{ \it the only entire maximal hypersurfaces in  $\bL^{N+1}$ are spacelike hyperplanes.} $$
 Later on, Bartnik and Simon \cite{BS82} established basic results on the boundary-value problem   
\begin{align}\label{eq 1.0-1}
\cM_0 u=H\quad {\rm in}\ \,\Omega, \qquad u=\psi\quad {\rm on}\ \, \partial \Omega, 
\end{align}
 and provides necessary and sufficient conditions of $H, \psi$  for   the existence of regular strictly spacelike solution.  Moreover,  the  principle method is to consider the critical point of the energy functional, which may generates the hyper plane with slope 1 is possible under the suitable assumptions. 
   Bartnik et.al. \cite{B84,B84-1,B88} established qualitative properties of solutions to the mean curvature equations through analysis of the associated energy functional.    The  mean curvature equations have been of considerable interest in last few years.       Bonheure-Iacopetti  \cite{BI23} studied gradient estimates for related Poisson problem.  Further properties  could see \cite{SSY75,HY21,RWX24,H20,H16,AD26,W25} and reference therein.
 
 Maximal hypersurfaces in Minkowski space exhibiting cone-like singularities have attracted considerable attention over the past several decades.   Kobayashi \cite{K84} classified isolated singularities in dimension two as cone-like.
 Kiessling in \cite{K12} tried to consider the cone-like singular solution of 
\begin{align}\label{eq 1.0-2}
\cM_0u=4\pi \sum _{j=1}^{m_0}\alpha_j\delta_{p_j}\quad {\rm in}\ \, \cD'(\R^3),\qquad \lim_{|x|\to+\infty}u(x)=0
\end{align}
via the variational method by employing a Taylor expansion decomposition technique. 
In 2016, Bonheure-d'Avenia-Pomponio  \cite{BAP} 
studied  the Born-Infeld-type electrostatic equation
\begin{equation}\label{eq 1.1-fund3}
 \left\{
\begin{array}{lll}
\displaystyle \  \cM_0 u = \sum^{m_0}_{j=1}\alpha_j\delta_{p_j}  \quad
  {\rm in}\  \, \cD'(\R^N), 
\\[5mm]
 \phantom{    }
 \displaystyle \lim_{ |x|\to+\infty}u(x)=0.        
 \end{array}
 \right.
 \end{equation}
where $N\geq 3$ and   $\delta_{p_j}$ is the Dirac mass concentrated on $p_j\in\R^N$. We refer to \cite{BS23,D13,T82}  for more studies on the entire solutions of the mean curvature equations.  Recently, the Dirichlet problems
with singular Lorentzian mean curvature in bounded regular domain has been studied in \cite{BI24} and also see the references  \cite{BA19,BA20,BI23,BI19} for the study of  the Born-Infeld-type electrostatic equation. 
 From \cite[Theorem 1.3]{BAP}, they proved that Eq.(\ref{eq 1.1-fund3}) has a unique minimum point of the energy functional $\cJ_N$, where 
  \begin{equation}\label{eq1-ind-1}
\cJ_N(w)= \int_{\R^N}\big(1- \sqrt{1-|\nabla w|^2}\big)\, dx -\sum_{j=1}^{m_0}\alpha_j w(p_j)  \quad {for\ }\, w\in \bX_\infty(\R^N)
 \end{equation} 
and
$$\bX_\infty (\R^N)=\{v\in C^{0,1}(\R^N): |\nabla v|\leq 1\ \, a.e.\text{ in}\ \R^N , \ \  \int_{\R^N}\big(1- \sqrt{1-|\nabla v|^2}\big)dx<+\infty \big\}. $$
They demonstrated that the solution is classical in $\R^N\setminus \cP_{m_0}$ under either of the following two conditions: either (i) the singular set $\cP_{m_0}$ consists of points that are mutually well-separated, i.e. far away each other or (ii) the coefficients of the underlying equation are sufficiently close to zero. These assumptions serve to exclude the presence of lightlike segments connecting any pair of singular points. A lightlike segment with endpoints $x_0$ and $y_0$ is defined as $\bL_{x_0,y_0}=\big\{tx_0+(1-t)y_0, t\in [0,1]\big\}$, and along such a segment, the solution satisfies $u(x)=u(y)+|x-y|$ for any $x,y\in \bL_{x_0,y_0}$, thereby exhibiting singular behavior on $\bL_{x_0,y_0}$. As shown in \cite{BS82}, when the problem is posed in a bounded domain $\Omega\subset \R^N$, any lightlike segment connecting two singular points can be extended to an entire straight line that traverses $\Omega$ without intersecting the boundary $\partial\Omega$. This geometric property constitutes a key ingredient in establishing improved interior regularity of solutions.

Thanks to the existence singular sets,  the authors in \cite{BAP} also proposed a conjecture that 
$$
\text{\it whether the  maximizer of $\cJ_N$ is a weak solution of the related Euler-Lagrangian Eq.(\ref{eq 1.1-fund3}). } $$
Similar conjectures could see \cite{BI24} for bounded domains. 
Furthermore, several fundamental questions remain open regarding Eq.\eqref{eq 1.1-fund3}: \smallskip 

{\it 
  Does every maximizer of $\mathcal{J}_N$ exhibit regularity in $\mathbb{R}^N \setminus \cP_{m_0}$?  
  
 Can such solutions be approximated—in an appropriate functional sense—by solutions corresponding to smooth approximations of the Dirac masses concentrated at $\cP_{m_0}$?
 }

\smallskip

The aim of this paper is to prove the existence of maximal hypersurfaces with multiple light-cone singularities—points where $|\nabla u(x)|=1$ at a prescribed finite set in the entire space, by solving the mean curvature equation directly.

To this end, let introduce the basic notations. Denote  $\cP_{m_0}$  the set of  the light-cone   vertices   with $1\leq m_0\in \N$ 
 \begin{equation}\label{sing set-1}
 \cP_{m_0}=\Big\{p_j\in\R^N\!: j=1,\cdots, m_0,\  \, p_{j_1}\not=p_{j_2}\ {\rm for}\ j_1\not=j_2\ {\rm if}\ m_0\geq 2 \Big\}
  \end{equation}
  and   the light cone singularity of the hypersurface as following: a graph function $u\in C^2(\R^N\setminus \cP_{m_0})\cap C^{0,1}(\R^N)$ is said to be light-cone singular
at  $\cP_{m_0}$    if 
$$|\nabla u (x)|<1\ \ {\rm in}\ \, \R^N \setminus \cP_{m_0},\qquad\  |\nabla u (x)|\to1 \ \ {\rm as}\ \, |x-p|\to0^+\ \ {\rm for\ any}\ p\in \cP_{m_0}.   $$
Now we involve  the N-dimensional  mean curvature operator  (MC opoerator)
$$\cM_0u(x)= -\nabla\cdot\Big(\frac{\nabla u(x)}{\sqrt{1-|\nabla u(x)|^2}} \Big)\quad \text{for\ $u\in C^2$ at $x\in\R^N$ and  $|\nabla u(x)|< 1.$}   $$ 
Note that 
$\cM_0$ is strictly elliptic operator at the domain $\big\{x\in\R^N\!:\, |\nabla u(x)|<1\big\}$ and degenerates 
 at $\{x\in\R^N\!:\, |\nabla u(x)|=1\}$. It is the mean curvature operator in Lorentz–Minkowski space for a spacelike hypersurface given by a graph $(x,u(x))$ in $\R^{N+1}$.

\medskip


Our first purpose in this article is to investigate the light-cone singular solutions  of  Eq.\eqref{eq 1.1-fund3} 
with $N\geq 3$ involving multiple positive Dirac masses. 
  
{\it Here a function $u$ is said to be a weak solution of (\ref{eq 1.1-fund3}) if 
$u\in C^{0,1}_{\loc}(\R^N)\cap C^2_{\loc}(\R^N\setminus \cP_{m_0})$ such that 
$\frac{|\nabla u|}{\sqrt{1-|\nabla u|^2}}\in L^1_{\loc}(\R^N)$, $\displaystyle \lim_{ |x|\to+\infty}u(x)=0$ and 
 \begin{equation}\label{fun 1 weak form-N}
\int_{\R^N} \frac{\nabla u(x) \cdot \nabla \phi(x)}{\sqrt{1-|\nabla u(x)|^2}} dx=\sum^{m_0}_{j=1} \alpha_j\phi(p_j)
\quad{\rm for\ any}\ \, \phi\in C^{0,1}_c(\R^N).  
\end{equation}}
For $\alpha>0$ and $N\geq 3$, denote
 \begin{equation}\label{fun 1}
 \Phi_{N, \alpha}(x)
 =   c_{N,\alpha} \int_{|x|}^{\infty}  \big( s^{2(N-1)}+c_{N,\alpha}^2 \big)^{-\frac12}\, ds  
    \quad {\rm for}\ \, x\not=0, 
\end{equation}
where $c_{N,\alpha}=\frac{\alpha}{|\partial B_1(0)|}$.
When $m_0=1$, direct computation shows that    (\ref{eq 1.1-fund3}) has a unique solution $\Phi_{N, \alpha_1}(\cdot-p_1)$. 
 and 
 $$ 
\Phi_{N, \alpha_1}(x-p_1) \sim c_{N,\alpha} |x|^{2-N} \quad {\rm as}\ \, |x|\to +\infty.  
 $$
  For $m_0\geq2$, we have following light-cone singularities. 
   
 \begin{theorem}\label{teo 1-fund3}
 Let   $N\geq 3$, $\cP_{m_0}$ be given in (\ref{sing set-1}) with    $m_0\geq2$ and 
 $$   \alpha_j>0 \quad{\rm and}\quad   \alpha_0=\sum^{m_0}_{j=1} \alpha_j, $$
 then  Eq.(\ref{eq 1.1-fund3}) has a unique weak solution $ u_{N,\alpha_0}\in C^2(\R^N\setminus \cP_{m_0})\cap C^{0,1}(\R^N)$  satisfying that 
 $\cP_{m_0}$ is the set of  light-cone singularities of $ u_{N,\alpha_0}$ and 
 \begin{equation}\label{be hv-1}
     u_{N,  \alpha_0}(x)  = c_N \alpha_0   |x|^{2-N}+O( |x|^{1-N})\quad {\rm as}\ \, |x|\to+\infty
\end{equation}
and
$$  u_{N, \alpha _0}(x)\geq  \Phi_{N, \alpha_j}(x-p_j), \ \, j=1,\cdots,m_0, \qquad  \max_{x\in\R^N}  u_{N,\alpha_0}(x)\leq \Phi_{N,  \alpha_0}(0), $$
where 
$c_N=\frac{\Gamma(\frac{N}{2})}{2\pi^{\frac N2}}=\frac1{|\partial B_1(0)|}$.

 \smallskip

Furthermore,   
$(a)$ there exist $\lambda_j\in\R$ with $j=1,\cdots, m_0$ such that 
$$\lim_{|x-p_j|\to0^+} u_{N,\alpha_0}(x)=\lambda_j$$
and
$$ |\lambda_j-\lambda_{j'}|< |  p_j-p_{j'}| \quad{\rm for}\ j\not=j'. $$
 
 
 $(b)$ The function $ u_{N,\alpha_0}$ verifies the equation 
 \begin{equation}\label{eq 1.1-fund-cal3}
 \left\{
\begin{array}{lll}
\displaystyle \quad \cM_0 u =0  \quad
  {\rm in}\  \, \R^N\setminus \cP_{m_0}, 
\\[3mm]
 \phantom{    }
 \displaystyle \lim_{ |x|\to+\infty}u(x)=0.      
 \end{array}
 \right.
 \end{equation}
 
 $(c)$  $ u_{N,\alpha_0}$  
 is   the minimizer of the energy functional  
$$
\cJ_{N}(w)= \int_{\R^N}\big(1- \sqrt{1-|\nabla w|^2}\big)\, dx - \sum^{m_0}_{j=1}\alpha_j w(p_j) \quad {for\ }\, w\in \bX_\infty (\R^N).   
$$

\end{theorem}
\medskip


For given positive Dirac masses, Theorem \ref{teo 1-fund3} provides a complete characterization of the solution to Equation (\ref{eq 1.1-fund3}). In particular, it affirms the conjecture by extending the admissible test function space to $C^{0,1}_c(\mathbb{R}^N)$—the largest natural space for weak solutions involving Dirac measures. Moreover, our theorem imposes no restrictions on either the locations of the Dirac points or the magnitudes of their coefficients. Finally, the asymptotic behavior at infinity (\ref{be hv-1}) is established by invoking the results of \cite{HY21}, where the authors classified all possible asymptotic behaviors of maximal hypersurfaces in exterior domains.

 \smallskip
  
Next,  we consider  the   light-cone singular solutions  of  elliptic equations involving multiple positive Dirac masses 
\begin{equation}\label{eq 1.1-fund}
 \left\{
\begin{array}{lll}
\displaystyle \  \cM_0 u = \sum^{m_0}_{j=1}\alpha_j\delta_{p_j}  \quad
  {\rm in}\  \, \cD'( \R^2), 
\\[5mm]
 \phantom{    }
 \displaystyle \lim_{ |x|\to+\infty}u(x)=-\infty,        
 \end{array}
 \right.
 \end{equation}
where $\alpha_j>0$.

The study of maximal hypersurfaces in two-dimensional spacetime proceeds fundamentally differently from higher-dimensional cases: complex-analytic techniques become applicable, and the first model-Born’s field equations-was already formulated by Pryce \cite{P35} in 1935. This early framework yields explicit solutions featuring a singular lightlike segment.  Subsequently, the authors \cite{FL05} combined tools from complex analysis, Riemann surface theory, and algebraic geometry to construct families of maximal hypersurfaces with finitely many isolated singularities. More recently, Umehara and Yamada \cite{UY19} constructed examples admitting an entire singular lightlike line. Additional foundational contributions include works by   \cite{K84,FS08,UY06,AUY19,UY19}. 

Our aim  is  to provide a complete classification of maximal hypersurfaces in two dimensions possessing a finite, positive number of singularities, via a approximation method.

{\it Here a function $u$ is said to be a weak solution of (\ref{eq 1.1-fund}) if 
$u\in C^{0,1}_{\loc}(\R^2)\cap C^2_{\loc}(\R^2\setminus \cP_{m_0})$ such that 
$\frac{|\nabla u|}{\sqrt{1-|\nabla u|^2}}\in L^1_{\loc}(\R^2)$, $\displaystyle \lim_{ |x|\to+\infty}u(x)=-\infty$ and 
$$\int_{\R^2} \frac{\nabla u(x)\cdot \nabla \phi(x)}{\sqrt{1-|\nabla u(x)|^2}} dx=\sum^{m_0}_{j=1} \alpha_j\phi(p_j)
\quad{\rm for\ any }\ \, \phi\in C^{0,1}_c(\R^2).  $$}

It is well-known that  the single Dirac mass can be  obtained directly by the ODE method:
when  $m_0=1$, $\alpha\not=0$  and $\cP_{1}=\{0\}$,
  problem (\ref{eq 1.1-fund}) has a unique solution 
\begin{equation}\label{sol-fund-20}
\Phi_{2,\alpha}(x)=-  \frac{\alpha}{2\pi} 
\Big( \ln \Big(  r+ \sqrt{ \big(\frac{\alpha}{2\pi}\big)^2+r^2}\, \Big)-\ln  \big(\frac{\alpha}{2\pi}\big) \Big)   \quad {\rm for}\ \, r=|x|>0.  
 \end{equation}
 Due to the quasilinear nature of the operator $\cM_0$, the fundamental solution of (\ref{eq 1.1-fund}) involving multiple Dirac masses cannot be obtained either by the ODE method or by superposing individual fundamental solutions corresponding to single Dirac masses.

 \begin{theorem}\label{teo 1-fund}
 
   Let $\cP_{m_0}$ be given in (\ref{sing set-1}) with $m_0\geq2$ and 
 $$    \alpha_j>0,\qquad \alpha_0=\sum^{m_0}_{j=1}\alpha_j.  $$
 Then Eq.(\ref{eq 1.1-fund}) has a weak solution $ u_{2,\alpha_0}\in C^2(\R^2\setminus \cP_{m_0})\cap C^{0,1}(\R^2)$  satisfying that 
 $\cP_{m_0}$ is the set of  light-cone singularities of $ u_{2,\alpha_0}$ and 
 $$   u_{2,\alpha_0}(x)   =-\frac{\alpha_0}{2\pi} (\ln |x|)+c+o(1)\quad  {\rm as}\ \, |x|\to+\infty$$ 
 for some $c\in\R$. 
 
 The solution $ u_{2,\alpha_0}$ is unique under the constraint at infinity that 
 $$u(x)=- \frac{\alpha_0}{2\pi} (\ln |x|)+c+o(1)\quad  {\rm as}\ \, |x|\to+\infty $$
 for a given $c\in\R$. 
  \smallskip

Furthermore,   
$(a)$ there exist $\lambda_j\in\R$ with $j=1,\cdots, m_0$ such that 
 \begin{equation}\label{fund-hei-2}
 \lim_{|x-p_j|\to0^+} u_{2,\alpha_0}(x)=\lambda_j
  \end{equation}
and
$$ |\lambda_j-\lambda_{j'}|< |  p_j-p_{j'}| \quad{\rm for}\ j\not=j'. $$


  $(b)$ The function $ u_{2,\alpha_0}$ is a classical solution of  the equation 
 \begin{equation}\label{eq 1.1-fund-cal}
 \left\{
\begin{array}{lll}
\displaystyle \quad \cM_0 u =0  \quad
  {\rm in}\  \, \R^2\setminus \cP_{m_0}, 
\\[3mm]
 \phantom{    }
 \displaystyle \lim_{ |x|\to+\infty}u(x)=-\infty.      
 \end{array}
 \right.
 \end{equation}
 
\end{theorem}

 
 



In Theorem \ref{teo 1-fund}, which involves the multiple Dirac mass model, solutions can be constructed when the coefficients $\{\alpha_j\}_j$ associated with the Dirac masses at points $\cP_{m_0}$ are prescribed. The coefficient governing the leading-order behavior at infinity is then determined by the combined effect of these Dirac masses, despite $\cM_0$ being a quasilinear elliptic differential operator. Similarly, due to the additivity inherent in the quasilinear operator, the heights $\{\lambda_j\}_j$ depend on the coefficients $\{\alpha_j\}_j$. {\it However, establishing an explicit and precise relationship between the heights $\{\lambda_j\}_j$ and the coefficients $\{\alpha_j\}_j$ remains challenging. Furthermore, it is still an open question whether the heights $\{\lambda_j\}_j$ of the conical singularities can be independently prescribed.} \smallskip

 Finally,  we proceed to construct hyper-surfaces containing singularities with downward and upward openings.  
To this end, we consider the weak solution of mean curvature equation involving the positive and negative Dirac masses
\begin{equation}\label{eq 1.1-fund+-}
    \left\{
\begin{array}{lll}
\displaystyle\quad  \cM_0 u = \sum^{m_1}_{j=1}\alpha_j\delta_{p_j}- \sum^{m_0}_{j=m_1+1}\beta_j\delta_{p_j}   \quad
 {\rm in}\  \, \cD'( \R^N), 
\\[3mm]
 \phantom{    }
 \displaystyle \lim_{ |x|\to+\infty}u(x)=0.      
 \end{array}
 \right.
 \end{equation}
where $\alpha_j,\beta_j>0$, $\{p_j\}_j \subset \R^N$, $p_j\not=p_{j'}$ for $j\not=j'$  and integers $m_1,m_2\geq 1$. In this section, we use the following notations: 
$$\cP_{m_1,+}=\{p_j,\, j=1,\cdots,m_1\},\quad \cP_{m_2,-}=\{p_j,\, j= m_1+1,\cdots, m_1+m_2\}$$
$$\cP_{m_0}:=\big(\cP_{m_1,+}\cup \cP_{m_2,-} \big)\subset B_{\frac12R_0}(0)\quad \text{ with $m_0=m_1+m_2$,\ $R_0\geq 1$}, $$ 
   and  
 $$   \alpha_0=\sum^{m_1}_{j=1}\alpha_j,\quad  \beta_0=\sum^{m_0}_{j=m_1+1}\beta_j.  $$
{\it Here a function $u$ is said to be a weak solution of (\ref{eq 1.1-fund+-}) if 
$u\in C^{0,1}_{\loc}(\R^N)\cap C^2_{\loc}(\R^N\setminus \cP_{m_0})$ such that 
$\frac{|\nabla u|}{\sqrt{1-|\nabla u|^2}}\in L^1_{\loc}(\R^N)$, $\displaystyle\lim_{ |x|\to+\infty}u(x)=0$  and 
$$\int_{\R^N} \frac{\nabla u(x)\cdot \nabla \phi(x)}{\sqrt{1-|\nabla u(x)|^2}} dx=\sum^{m_1}_{j=1} \alpha_j\phi(p_j)-\sum^{m_0}_{j=m_1+1} \beta_j\phi(p_j)
\quad{\rm for\ any }\ \, \phi\in C^{0,1}_c(\R^N).  $$}

The  results on light-cone singular solutions are stated as follows. 

\begin{theorem}\label{teo 1-fund+-}
 Assume that  $N\geq 3$, 
 \begin{equation}\label{con-classic-1}
 \Phi_{N,\alpha_0}(0)+\Phi_{N,\beta_0}(0) <l_0, 
    \end{equation}
    where
  $$l_0=\dist(\cP_{m_1,+},\cP_{m_2,-}):=\min\{|p_j-p_i|, j=1,\cdots, m_1, i=m_1+1,\cdots,m_0 \}.  $$
 
Then  Eq.(\ref{eq 1.1-fund+-}) has a weak solution $u_{N}\in C^2\big(\R^N\setminus (\cP_{m_1,+}\cup \cP_{m_2,-})\big)\cap C^{0,1}(\R^N)$  satisfying that 
 $\cP_{m_0}$ is the set of  light-cone singularities of $u_{N}$ and 
   \begin{equation}\label{fund-bah+-2}
    u_{N}(x) =    
  c_N ( \alpha_0-\beta_0) |x|^{2-N}+O(|x|^{1-N})\quad  {\rm as}\ \, |x|\to+\infty \quad {\rm for} \ \, N\geq 3.  
  \end{equation}

Furthermore,   
$(a)$ there exist $\lambda_j\in\R$ with $j=1,\cdots, m_0$ such that 
 \begin{equation}\label{fund-hei-2+-}
 \lim_{|x-p_j|\to0^+}u_{N}(x)=\lambda_j 
  \end{equation}
and
$$ |\lambda_j-\lambda_{j'}|< |  p_j-p_{j'}| \quad{\rm for}\ j\not=j'. $$


  $(b)$ The function $u_{N}$ is a classical solution of  the equation 
 \begin{equation}\label{eq 1.1-fund-cal+-}
  \cM_0 u =0  \quad
  {\rm in}\  \, \R^N\setminus  \cP_{m_0 }.
  \end{equation}
  
   \end{theorem}

In our analysis, we approximation the weak solutions of Eq.\eqref{eq 1.1-fund3}  for $N\geq 3$ by the classical solutions $u_{n}$ of 
\begin{equation}\label{eq 1.1-R--}
  \left\{
\begin{array}{lll}
\displaystyle \cM_0 u= g_n  \quad
  {\rm in}\  \, \cD'\big(\R^N\big), 
\\[3mm]
 \phantom{    }
 \displaystyle \lim_{|x|\to+\infty} u(x)=0   
 \end{array}
 \right.
 \end{equation}
as $n\to+\infty$, while the solution $u_{n}$ is approximated by the regular solutions $u_{n,R}$ of 
 \begin{equation}\label{eq 1.1-R--n}
  \left\{
\begin{array}{lll}
\displaystyle \cM_0 u= g_n   \quad   {\rm in}\  \, B_R(0), 
\\[2mm]
 \phantom{  --   }
 \displaystyle  u=0 \quad {\rm on}\ \partial B_R(0),        
 \end{array}
 \right.
 \end{equation}
where $\{g_n\}_n$ is a sequence of functions that converge to $\displaystyle\sum_{j=1}^{m_0} \alpha_j  \delta_{p_j}$
as $n\to+\infty$. One of the main difficulties  in this approximation is the convergence of
$\frac{|\nabla u_{n}|}{\sqrt{1-|\nabla u_{n}|^2}}$  in $L^1_{\loc}(\R^N)$.   Another difficulty is to get a uniform bound, which could provide the decaying at infinity.
To overcome this,  we employ the method of rearrangement, which allows for a comparison with single isolated singular solutions, then we make use of the classification of the isolated singularities by the Schwartz theorem, which plays the most important role in the dealing with Dirac masses. We then take the limit of $u_{n}$ as $n \to +\infty$ to derive a weak solution $u_{N,\alpha_0}$ defined on $\mathbb{R}^N$. 
Also, we show that $u_{N,\alpha_0}$ is the minimizer of energy functional
$$
\cJ_{N}(w)= \int_{\R^N}\big( 1-\sqrt{1-|\nabla w|^2}\big)\, dx -\sum_{j=1}^{m_0}\alpha_j w(p_j)  \quad {for\ }\, w\in \bX_\infty(\R^N).  
 $$
Furthermore, under the condition of decay at infinity, we can show that $ u_{N, {a}_0}$ is the unique critical point of $\mathcal{J}_{N}$.
  As we have established that $u_{N,\alpha_0}$ is a weak solution, we are in a position to address the conjecture posed in \cite{BAP}. 

 
\vskip2mm
However, when $N=2$,  we can't pass to the limit of $u_{n}$ of (\ref{eq 1.1-R--n}) as $n\to+\infty$ directly, because it  blows up 
wholly in $\R^2$. In fact,  the weak solution $\Phi_{2,\alpha}$ of  problem (\ref{eq 1.1-fund}) with single Dirac mass could be obtained with form (\ref{sol-fund-20}) by ODE method. It is no longer a critical point of $\cJ_{2}$,  defined by (\ref{eq1-ind-1}), thanks to
  \begin{align*}
 \int_{B_R}\Big(1- \sqrt{1-|\nabla \Phi_{2,\alpha}|^2}\Big)\, dx& \geq 
  \frac12 \int_{B_R}  |\nabla \Phi_{2,\alpha}|^2  dx 
  \\[0.5mm]& =\frac12  \big(\frac{\alpha}{2\pi} \big)^2 \int_{B_R}  \frac1{(\frac{\alpha}{2\pi})^2+|x|^2}  dx   \to+\infty\quad {\rm as}\ \, R\to+\infty,
 \end{align*}
 since 
\begin{align*}
 |\nabla \Phi_{2,\alpha} (x)|= \frac{\alpha}{2\pi}\frac{1}{\sqrt{(\frac{\alpha}{2\pi})^2+|x|^2}}.  
\end{align*}
For this reason, the variational method fails.   
The weak solution of problem (\ref{eq 1.1-fund}) is obtained by normalization via an adjustment of the maximum, specifically by setting
$$\tilde u_{n,R}=u_{n,R}-\max_{z\in B_R(0)} u_{n,R}(z),$$
which ensures that the maximum value is zero.
 The sequence $\tilde u_{n,R}$ keeps locally uniformly bounded and the same maximum point   as $R\to+\infty$, taking a subsequence if necessary. 
 
 \smallskip

The remainder of this paper is organized as follows. In Section 2, we recall the basic properties of mean curvature operators, build the Symmetric Decreasing Rearrangement, show the basic regularity theory for the Poisson problem, and prove the classification of isolated singularities. Section 3 is devoted to constructing light-cone solutions of (\ref{eq 1.1-R--}), which are approximated by classical solutions to (\ref{eq 1.1-R--n}). Sections 4  present the analysis of solutions to (\ref{eq 1.1-fund}) in dimension 2 and to (\ref{eq 1.1-fund3}) in dimensions $N \geq 3$ and show the existence solution of Eq.(\ref{eq 1.1-fund+-}). 
Finally, we construct hypersurfaces with infinitely many light-cones by 
considering the problem
 $$ \cM_0 u =  \sum^{\infty}_{j=1} \alpha_j  \delta_{p_j}    \quad
  {\rm in}\  \, \cD'(\R^N), 
$$
  under the hypothesis that $N\geq3$,  $\alpha_j>0$ and $\displaystyle \sum^{\infty}_{j=1}\alpha_j<+\infty$.  

 \medskip

\setcounter{equation}{0}
\section{Preliminary}

   Let $R_0\geq 1$ be  such that 
$$\cP_{m_0}:=\{p_j: j=1,\cdots,m_0\}\subset B_{\frac12 R_0}(0). $$
 Let $C^{0,1}_0 (B_R(0))=\{\zeta\in C^{0,1} (B_R(0)):\, \zeta=0\ \ {\rm on}\ \, \partial B_R(0) \}$ for $R<+\infty$ and  when $R=+\infty$,    $C^{0,1}_c (\R^N)=\{\zeta\in C^{0,1} (\R^N):\, \zeta\ \ \text{ has compact support}\}$. {\bf  For simplicity, we still use the notations }
 $$   C^{0,1}_c (\R^N)=C^{0,1}_0 (B_R(0))\quad \text{ when $R=+\infty$.}$$

 \subsection{Properties of the MC operator}

 We first introduce the classical comparison principle. 

 \begin{lemma}\label{lm cp-cla}\cite[Lemma 1.2]{BS82}
 Let $\Omega$ be a  bounded $C^2$ domain in $\R^N$ with $N\geq 2$, functions $u,v\in C^2(\bar \Omega)$ be such that $|\nabla u|, |\nabla v|<1$ in $\Omega$ and 
 $$\cM_0u\geq \cM_0v \quad {\rm in}\ \, \Omega, \qquad u=\psi_1,\ \  v=\psi_2\ \ {\rm on}\ \, \partial \Omega, $$
 then 
 $$u\geq v+\inf_{x\in\partial\Omega}( \psi_1-\psi_2)\quad\text{ in $\Omega$}.$$ 
 \end{lemma}

This principle could be extended to weak source in following setting. 
 $$\cI_{w,H}(u)=\int_{\Omega} \big(\sqrt{1-|\nabla u|^2}+H(x)u(x)\big)dx,\quad \forall\, u\in \bX_w(\Omega) $$
 with 
$$\bX_w(\Omega):=\Big\{v\in C^{0,1}(\bar \Omega):\,  v= w\ {\rm on}\ \partial \Omega,\ \, |\nabla v|\leq 1\ \ {\rm a.e.\ in\ }  \Omega\Big\},   $$
where $\Omega$ be a  bounded $C^2$ domain in $\R^N$ with $N\geq 2$.

 \begin{lemma}\label{lm cp}
 Let  $H_1,H_2$ be two bounded Radon measures such that 
 $$\int_{\Omega} H_1 \xi dx \geq \int_{\Omega} H_2 \xi dx\quad \text{ for all }\, \xi \in C(\bar \Omega), \ \xi \geq 0.$$ 
Let $w_1,w_2\in C^{0,1}(\bar \Omega)$  satisfy  $w_1\geq w_2$, 
and   $u_i\in C^{0,1}(\bar \Omega)$ be the critical points of $\cI_{w_i,H_i}$ with $i=1,2$, 
 then 
 $$u_1\geq u_2 \quad\text{ in $\Omega$}.$$ 
 \end{lemma}
 \noindent{\bf Proof. } Let $C=\sup_{\partial \Omega} (w_1-w_2)$ and $\tilde u_1=u_1+C+\epsilon$ with $\epsilon>0$. 
 Set $\Omega_+=\{x\in\Omega:\, u_2>\tilde u_1 \}$.  
 
 If $\Omega_+$ is non-empty, the function $(u_2-\tilde u_1)_+:=\max\{0, u_2-\tilde u_1\}\in C^{0,1}(\Omega)$ vanishes on $\partial \Omega$. Define 
 $$\cI_i(u)=\int_{\Omega_+} \big(\sqrt{1-|\nabla u|^2}+uH_i\big)dx,\quad i=1,2. $$
 Then $\tilde u_1$ maximizes $\cI_1$ with respect to $\tilde u_1\big|_{\partial \Omega_+}$ and $u_2$ maximizes $\cI_2$ with the same boundary values by the definition of $\Omega_+$. By the uniqueness \cite[Proposition 1.1]{BS82},  there holds 
 $\cI_2(\tilde u_1)<\cI_2(  u_2)$ and then by the fact that $H_1\geq H_2$ and  $ \tilde u_1<u_2$ in $\Omega_+$ 
 \begin{align*}
 \int_{\Omega_+}  \sqrt{1-|\nabla \tilde u_1|^2}dx&<\int_{\Omega_+} \big(\sqrt{1-|\nabla u_2|^2}- (\tilde u_1-u_2)H_2\big)dx
 \\[1mm]&= \int_{\Omega_+} \big(\sqrt{1-|\nabla u_2|^2}- (\tilde u_1-u_2)H_1- (\tilde u_1-u_2)(H_2-H_1)\big)dx
 \\[1mm]&\leq \int_{\Omega_+} \big(\sqrt{1-|\nabla u_2|^2}- (\tilde u_1-u_2)H_1 \big) dx,
 \end{align*}
then we obtain   that 
  \begin{align*}
 \int_{\Omega_+} \big(\sqrt{1-|\nabla \tilde u_1|^2}dx+\tilde u_1 H_1\big)dx <\int_{\Omega_+} \big(\sqrt{1-|\nabla u_2|^2}+  u_2 H_1 \big)dx,
 \end{align*}
 which contradits the maximality of $\cI_1(\tilde u_1)$.  Therefore, $\Omega_+=\emptyset$ and by the arbitrary of $\epsilon>0$, 
 we have that $u_1\geq u_2$ in $\Omega$. \hfill$\Box$  
 \medskip


The Hopf's Lemma is stated as following. 

 \begin{lemma}\label{lm hopf}
 Let $\Omega$ be a  bounded $C^2$ domain in $\R^N$ with $N\geq 2$,  function $u\in C^2(\Omega)\cap C^1(\bar\Omega)$ be such that $|\nabla u|\leq \theta<1$ in $\Omega$ and 
 $$\cM_0u=0 \quad {\rm in}\ \, \Omega.  $$
 If $x_0\in \partial \Omega$ such that $u(x_0)>u(x)$   ($u(x_0)<u(x)$ resp.) for all $x\in\Omega$ , then $\nabla u(x)\cdot \nu>0$ ($\nabla u(x)\cdot \nu<0$ resp.), where $\nu$ is the normal vector pointing outside of $\Omega$.

 \end{lemma}
\noindent{\bf Proof. } Since $|\nabla u| <1$ in $\Omega$, then $\cM_0$ is uniformly elliptic with respect to $u$.  Note that 
\begin{align*}
\cM_0 w(x)&=  -  \frac{\Delta w(x)}{ (1-|\nabla w(x)|^2)^{\frac12}} - \sum^{N}_{i,j=1} \frac{D_iw(x)D_jw(x) D_{ij}w(x)}{ (1-|\nabla w(x)|^2)^{\frac32}}
\\[1mm]&=\sum^{N}_{i,j=1}  a_{ij} D_{ij}w(x) \quad {\rm for}\ \, x\in\Omega,    
\end{align*}
 then 
 $$a_{ij}= \frac{(1-|p|^2)\delta_{ij} +p_ip_j }{ (1-|p|^2)^{\frac32}},$$
 which is independent of $z$ and continuous differentiable respect to the $p=D w(x) (=\nabla w(x))$ variable. 

Then it is a uniformly  elliptic operator if  $|\nabla u|\leq \theta<1$ in $\Omega$,
and our statement follows by the Hopf's Lemma \cite[Lemma 3.4]{GT83}. \hfill$\Box$

 \begin{corollary}\label{cor 2.1}
 Let $u\in C^2(\Omega)\cap C^1(\bar \Omega)$ satisfy $|\nabla u|<1$ and 
 $$\cM_0u=0\ \ \text{in  $\Omega$},$$
  then $u$ has no local maximum point or no local maximum point  in $\Omega$.   
 \end{corollary}
\noindent{\bf Proof. } If there exists a local  maximum point, a local  minimal point   $x_0\in \Omega$, then $\nabla u(x_0)=0$ 
and there is a $C^2$ domain  $\cO_0\subset \{u(x)<u(x_0)\}$ such that $x_0\in \partial \cO_0$. By Hopf's Lemma 
there holds 
$$D_\nu u>0,\quad \nu\text{ is a normal vector at $x_0$ pointing outside of $\cO_0$},  $$
which contradicts the fact that $\nabla u(x_0)=0$. \hfill$\Box$\medskip

 \subsection{Dirichlet problems }

 We first recall the notion of Symmetric Decreasing Rearrangement.  For a function $w: \Omega \to \mathbb{R}^+$,  its symmetric decreasing rearrangement 
 $w^*: \Omega^* \to \mathbb{R}^+$ is a radially symmetric, decreasing function that has the same distribution function as $w$, 
 where $\Omega^*$ is the ball centered at the origin with the same volume as $\Omega$. 
 The level sets of $w^*$ are balls whose volume equals the volume of the corresponding level sets of $w$. The rearrangement preserves $L^p$ norms: 
 $$\|w\|_{L^p(\Omega)} = \|w^*\|_{L^p(\Omega^*)}\quad\text{ for $p\in[ 1,+\infty]$.}$$
 Moreover, we have that
 $$w^*(0)=\max_{x\in \Omega} w(x). $$
 
Note that  P\'{o}lya-Szeg\H{o} inequality 
 $$\norm{\nabla w}_{L^1(\Omega)}\ge \norm{\nabla w^*}_{L^1(B_r(0))}\quad{\rm for}\ \ r>0\ {\rm s.t.}\ |B_r|=|\Omega|. $$
Generally,  let $\cH: \mathbb{R}^+ \to \mathbb{R}^+$ be a non-decreasing, convex function with $\cH(0) = 0$. Let   $w \in W^{1,1}_0(\Omega)$ with $w \ge 0$, then 
  $$\int_{\Omega^*} \cH(|\nabla w^*|) dx\leq  \int_{\Omega} \cH(|\nabla w|) dx. $$

 \begin{lemma}\label{lm 2.2-rea}
 Let $f\in C^1(B_R(0))$ be non-negative and non-trivial with $R>0$ and  $u_f$ be the positive solution of 
  \begin{equation}\label{eq 2.2}
 \left\{
\begin{array}{lll}
\displaystyle  \cM_0u =f  \quad
  &{\rm in}\  \ B_R(0), 
\\[2mm]
 \phantom{\quad \,  }
 \displaystyle u=0 \quad & {\rm on}\  \, \partial B_R(0).    
 \end{array}
 \right.
 \end{equation}
Then the rearrangement $u^*_f$ verifies that 
  \begin{equation}\label{ineq reag}
  u^*_f\leq u_{f^*}\quad {\rm in}\ B_R(0),
   \end{equation}
where $u_{f^*}$ is the radial symmetric  solution of 
  \begin{equation}\label{eq 2.2*}
  \left\{
\begin{array}{lll}
\displaystyle  \cM_0u= f^*  \quad
  &{\rm in}\  \ B_R(0), 
\\[2mm]
 \phantom{\quad \,  }
 \displaystyle u=0 \quad & {\rm on}\  \, \partial B_R(0).    
 \end{array}
 \right.
 \end{equation}
 \end{lemma}
\noindent{\bf Proof. } Here we can apply the method in \cite{T76} to show the bounds, where the author proved 
the same results for the Laplacian case. 
Use the notations: $\{u(x)>t\}=\{x\in B_R(0):\, u(x)>t\}$ for $t>0$. By integrate the equation (\ref{eq 2.2}) over $\{u(x)>t\}$, we derive that 
\begin{align*}
\int_{\{u(x)>t\}} f(x)dx&=\int_{\{u(x)>t\}} \cM_0u(x) dx
\\[1mm]&=\int_{\{u(x)=t\}} \frac{\nabla u}{\sqrt{1-|\nabla u|^2}} \cdot \frac{\nabla u}{|\nabla u|}dH_{N-1}(x)
\\[1mm]&=\int_{\{u(x)=t\}} \frac{|\nabla u|}{\sqrt{1-|\nabla u|^2}} dH_{N-1}(x)
\\[1mm]&\geq \int_{\{u(x)=t\}}  |\nabla u| dH_{N-1}(x), 
\end{align*}
i.e. 
$$\int_{\{u(x)=t\}}  |\nabla u| dH_{N-1}(x)\leq \int_{\{u(x)>t\}} f(x)dx\quad {\rm for\ a.e.}\ t>0. $$
The remainder proof is nothing with the form of the equation and it follows the proof  of  \cite[Theorem I]{T76}
directly to obtain  the inequality (\ref{ineq reag}). 
\hfill$\Box$\medskip

 Next  we recall the previous  results on Poisson problems involving the mean curvature operator
\begin{equation}\label{eq 2.1-fg}
 \left\{
\begin{array}{lll}
\displaystyle   \cM_0u = f\quad
  &{\rm in}\  \, \Omega, 
\\[2mm]
 \phantom{- \ \,  }
 \displaystyle u=g\quad &{\rm on}\ \, \partial \Omega, 
 \end{array}
 \right.
 \end{equation}
 where $\Omega$ is a  bounded, connected $C^2$ domain of $\R^N$ with $N\geq 2$.  
 Here  a function $u\in \bX_g(\Omega)$ is a  solution of  (\ref{eq 2.1-fg}) if  there holds  
   \begin{align*}
 \int_{\Omega }   \frac{\nabla u\cdot \nabla  \phi}{\sqrt{1-|\nabla u|^2}}\,  dx=   \int_{\Omega } f\phi dx, \quad \text{ for $\phi\in C^1_0(\Omega)$},  
    \end{align*}
where
$$\bX_g(\Omega)=\Big\{w\in C^{0,1}(\bar \Omega):\,w=g\ {\rm on}\ \partial\Omega,\quad  |\nabla w|\leq 1\ {\rm a.e.\ in}\ \Omega,\quad  \frac{|\nabla w|}{\sqrt{1-|\nabla w|^2}}  \in L^1 (\Omega)\Big\}. $$

\begin{lemma}\label{lm 2.1-Existence}\cite[Corollary 4.3]{BS82}
Let 
$f\in L^\infty(\Omega)$ and $g\in C^{0,1}(\bar \Omega)$  satisfy 
$$\sup_{x,y\in \partial \Omega,x\not=y} \frac{g(x)-g(y)}{|x-y|}<1,  $$
 then problem (\ref{eq 2.1-fg})
 has a unique weak solution.  
\end{lemma}

\begin{lemma}\label{lm 2.1-reg}\cite[Theorem 3.6]{BS82}
Assume that  $\Omega$ is a bounded,  $C^{2,\alpha}$ domain in $\R^N$ with $N\geq 2$,   
$f\in C^{0,\alpha}(\Omega),\ g\in C^{2,\alpha}(\bar \Omega)$  with $\alpha\in(0,1)$ satisfy that 
$$|f|\leq \Lambda_0\ \  \text{in $\Omega$}, \qquad  |\nabla g|\leq 1-\theta_0 \ \  \text{in $\bar\Omega$} $$
for some $\Lambda_0>0,\ \theta_0\in(0,1)$. 
Then problem (\ref{eq 2.1-fg}) has strictly spacelike solution $u_{f,g}\in C^{2,\alpha}(\bar \Omega)$.  

 Furthermore, there exists $\theta=\theta(\Lambda_0,\theta_0,\Omega,g)\in(0,1)$ such that  
  $|\nabla u_{f,g}|\leq 1-\theta$ in $\Omega$.

\end{lemma}

 We recall also the interior gradient estimate and higher regularity result as follows. 
\begin{lemma}\label{lm 2.1-reg-int}  
Let  $B_r(0)$ be a ball in $\R^N$ with $N\geq 2$, $r>0$.   Let $u\in C^2(B_r(0))$ satisfy
that 
$$  \cM_0 u=0\quad {\rm in}\ \, B_r(0)\quad{\rm and}\quad |\nabla u|\leq  \theta\ \ \text{in $B_{\frac r2}(0) $}  $$
for some $\theta\in(0,1)$. 
Then  there exist $C=C(N,r,\theta)>0$  and $\gamma\in(0,1)$ independent of $u$ such that   
$$   \|u\|_{C^{2,\gamma}(B_{\frac r4}(0))}\leq C. $$

\end{lemma}
 \noindent{\bf Proof. } 
As shown previously, 
  \begin{align*}
\cM_0 w(x) =-  \sum^{N}_{i,j=1}  a_{ij} D_{ij}w(x) \quad {\rm for}\ \, x\in\Omega,    
\end{align*}
 with 
 $$a_{ij}= \frac{(1-|p|^2)\delta_{ij} +p_ip_j }{ (1-|p|^2)^{\frac32}},$$
which is uniformly elliptic in $B_{\frac r2}(0)$ by the gradient bound. 
Precisely, we have that 
$$\lambda|\zeta|^2 \leq a_{ij}\zeta_i\zeta_j\leq \Lambda|\zeta|^2\quad {\rm for}\ \zeta\in\R^N, $$
where 
$$\lambda=(1-\theta)^{-\frac12}\quad{\rm and}\quad  \Lambda=(1-\theta)^{-\frac32}. $$
Since $u-t$ with $t\in\R$ verifies  $\cM_0 (u-t)=0$ in $B_r(0)$, so we can assume 
$u(0)=0$. In this case, $\|u\|_{L^\infty(B_{\frac r2}(0))}<\frac r2$ by the fact $|\nabla u|<1$,

It follows by \cite[Theorem 8.24,\ Theorem 8.32]{GT83} that for some $\gamma\in(0,1), \ C>0$ independent of $u$, 
$$ \|\nabla u\|_{C^\gamma(B_{\frac r2}(0))}\leq C, $$
which implies that $a_{ij}\in C^{\gamma}(B_{\frac r2}(0))$, 
Now we apply \cite[Theorem 6.2]{GT83} to obtain the bound
  $$
   \|u\|_{C^{2,\gamma}B_{\frac r4}(0)}\leq C'\|u\|_{L^\infty(B_{\frac r2}(0))}\leq C. 
  $$
The proof ends.\hfill$\Box$\medskip

The following classification of the behaviors at infinity of maximal hypersurfaces in exterior domain plays an important 
role in our analysis of the ones with light-cones in the whole space. 
\begin{theorem}\label{theo 2.1-inf}
\cite[Theorem 1.1]{HY21} 
Let 
\begin{align}\label{eq 2.1-ext}
\cM_0u=0\quad {\rm in}\ \,  \R^N\setminus A
 \end{align} 
 and for $R>r_0$
 $${\rm Res}[u]=\int_{\partial B_R} \frac{\nabla u (x)}{\sqrt{1-|\nabla u(x)|^2}}\cdot \frac{x}{R}\,d\omega_{1}(x),  $$
 where $A$ is a compact set in $\R^N$ and  $A\subset B_{r_0}(0)$ for some $r_0>0$. 
Then there exist $c\in\R$ and $\vec{a}\in B_1(0)$ such that when $N=2$
\begin{align}
 u(x)=&\vec{a}\cdot x +\frac1{2\pi}(1-|\vec{a}|){\rm Res}[u]\ln\sqrt{|x|^2-(\vec{a}\cdot x )^2}  +c\nonumber
\\[1mm] &\qquad  +{\rm Res}[u]\,  |\vec{a}| \, \frac{|x|(a\cdot x)}{|x|^2-(a\cdot x)^2}\cdot \frac{\ln |x|}{|x|}+o(|x|^{-1})\quad {\rm as}\ \, |x|\to+\infty \label{bb1-1} 
 \end{align} 
  and when $N\geq 3$ 
\begin{align}\label{bb1-2} 
u(x)=\vec{a}\cdot x+c - \frac1{|\partial B_1(0)|} (1-|\vec{a}|) {\rm Res}[ u]\big(\sqrt{|x|^2-(\vec{a}\cdot x )^2} \big)^{2-N} +O(|x|^{1-N})\quad {\rm as}\ \, |x|\to+\infty.
 \end{align}

\end{theorem}

We suse also the following

\begin{proposition}\label{pr grad-ext-2.1}\cite[Theorem 5.3]{HY21} Let $u$ be an classical  solution of 
(\ref{theo 2.1-inf}) in an exterior domain $\R^N\setminus A$. For any open set $U\supset A$, there exists $\theta\in(0,1)$ such that
$$|\nabla u|\leq \theta\quad {\rm in}\ \, \R^N\setminus U. $$
Moreover, 
$$\lim_{|x|\to+\infty} \nabla u(x)=\vec{a}$$
for some $\vec{a}\in B_1(0)$. 

\end{proposition}

  \subsection{Isolated singularities}

Let $u$ be a  classical solution of 
\begin{equation}\label{eq 1.1-R}
 \left\{
\begin{array}{lll}
\cM_0 u= 0  \quad
  {\rm in}\  \, B_R(0)  \setminus\cP_{m_0}, 
\\[3mm]
 \phantom{    }
 \displaystyle  u(x)=0 \quad {\rm on}\ \partial B_R(0)       
 \end{array}
 \right.
 \end{equation}
 or 
\begin{equation}\label{eq 1.1-infty}
 \cM_0 u = 0  \quad
  {\rm in}\  \, \R^N\setminus\cP_{m_0}.  
  \end{equation}

  \begin{proposition}\label{cr 2.1-iso}
Let $N\geq 2$, $R>R_0$ and  $u$ be a nonnegative classical solution of (\ref{eq 1.1-R}) or (\ref{eq 1.1-infty}) satisfying 
\begin{equation}\label{eq 1.1-req}
\frac{|\nabla u|}{\sqrt{1-|\nabla u|^2}}\in L^1(B_R(0)),\ \,  \big( \in L^1_{\loc}(\R^N)\ \text{when $R=+\infty$}\big).  \end{equation} 
 Then $u$  is a weak solution of
\begin{equation}\label{eq 1.1-R-w}
 \left\{
\begin{array}{lll}
\displaystyle \cM_0 u= \sum_{j=1}^{m_0} k_{p_j}  \delta_{p_j}   \quad
  {\rm in}\  \, \cD'\big(B_R(0)\big), 
\\[3mm]
 \phantom{    }
 \displaystyle  u(x)=0 \quad {\rm on}\ \partial B_R(0)       
 \end{array}
 \right.
 \end{equation}
 or 
 \begin{equation}\label{eq 1.1-infty-w}
  \cM_0 u = \sum_{j=1}^{m_0} k_{p_j}  \delta_{p_j} \quad
  {\rm in}\  \,  \cD'\big(\R^N\big)
  \end{equation}
for some $k_{p_j}\in\R$ with $j=1,\cdots,m_0$.

\end{proposition}

Let $u\in C^{0,1}(B_R(0)),\ \big(\,C^{0,1}(\R^N)\, {\rm resp.}\big)$  satisfy
$$|\nabla u|< 1 \ \ {\rm a.e.},\quad  \frac{|\nabla u|}{\sqrt{1-|\nabla u|^2}}\in L^1(B_R(0)), \quad\Big( \frac{|\nabla u|}{\sqrt{1-|\nabla u|^2}}\in L^1_{loc}(\R^N)\ {\rm resp.} \Big),  $$
where  $R\in(R_0,+\infty)$.  

\vskip2mm
Denote
 \begin{equation}\label{operator 1-R}
\cT_{u,R}(\xi):=
 \left\{
\begin{array}{lll}
\displaystyle \int_{B_R(0)}  \frac{\nabla u\cdot\nabla \xi}{\sqrt{1-|\nabla u|^2}} \,dx \quad &{\rm for}\ \, \forall\,\xi \in  C^{0,1}_0 (B_R(0))\quad {\rm if}\ \, R\in(0,+\infty),\\[6mm]
 \phantom{    }
\displaystyle  \int_{\R^N}  \frac{\nabla u\cdot\nabla \xi}{\sqrt{1-|\nabla u|^2}} \,dx\quad &{\rm for}\ \, \forall\, \xi\in C^{0,1}_c (\R^N)\qquad {\rm if}\ \, R=+\infty, 
 \end{array}
 \right.
 \end{equation}

 Observe that by assumption (\ref{eq 1.1-req}), for any $\xi\in C^{0,1}_0 (B_R(0))$,
$$\Big|\int_{B_R(0)}  \frac{\nabla u\cdot\nabla \xi}{\sqrt{1-|\nabla u|^2}} \,dx\Big| <+\infty,$$
then 
 $\cT_{u,R}$ 
 is a    bounded functionals of $C^{0,1}_0 (B_R(0))$.
Assume more that  for any $\xi\in C^{0,1}_0 (B_R(0))$ with
the compact support in $B_R(0)\setminus\cP_{m_0}$, then
\begin{equation}\label{supp-1}
\cT_{u,R}(\xi)=0.
\end{equation}

This means that the support of $\cT_{u,R}$ is an isolated set $\cP_{m_0}$, a set of finite points, by Theorem XXXV in \cite{S} (see also Theorem 6.25 in \cite{R}), it implies that
\begin{equation}\label{S}
\cT_{u,R}=\sum_{j=1}^{m_0}\Big(\sum_{|a|=0}^{N_j} k_{p_j, a} D^{a}\delta_{p_j}\Big),
\end{equation}
for $N_j \geq 1$, $a=(a_1, \cdots, a_N)$, which is a multiple index with $a_i\in\N$, 
where
 $\displaystyle |a|= \sum_{i=1}^N a_i$,   $D^0\delta_{p_j}=\delta_{p_j}$
and
$$\langle D^{a}\delta_{p_j},\xi\rangle= \frac{\partial^{|a|} \xi(0)}{\partial^{a}  x}.$$

Then we have that
\begin{equation}\label{3.3}
\cT_{u,R}(\xi)= \int_{B_R(0)}  \frac{\nabla u\cdot\nabla \xi}{\sqrt{1-|\nabla u|^2}} \,dx =\sum_{j=1}^{m_0}\Big(\sum_{|a|=0}^{N_j} k_{p_j, a} D^{a}\xi(p_j)\Big),\quad\ \ \forall\, \xi\in C^\infty_0 (B_R).
\end{equation}

\begin{lemma}\label{pr 2.1-iso}
Under the assumption of Proposition \ref{cr 2.1-iso}, let $\cT_{u,R}$ be given in (\ref{operator 1-R}) with $u$ being the solution form Proposition \ref{cr 2.1-iso}.
Then
\begin{equation}\label{3.0}
 k_{p_j,a}=0\quad{\rm for\ any}\quad |a|\ge 1.
\end{equation}
\end{lemma}
\noindent {\bf Proof. } Without loss of generality,   we only need to consider one singular point $p_j$ and set $p_j=0$, $k_{p_j,a}=k_a$. 
 
For any multiple index $a=(a_1,\cdots,a_N)$, let $\zeta_a$ be a    $C^\infty$ function such that
\begin{equation}\label{3.2}
{\rm supp}(\zeta_a)\subset \overline{B_1(0)}\quad{\rm and}\quad \zeta_a(x)=k_{a} \prod_{i=1}^N x_i^{a_i} \quad {\rm for}\  \ x\in B_1(0).
\end{equation}
Now we use the test function
$\xi_\epsilon(x):=\zeta_a(\epsilon^{-1}x)$ for $ x\in\R^N$
in  (\ref{3.3}),
we have that
$$\sum_{|a|\le q} k_{a} D^{a}\xi_\epsilon(0)=\frac{k_{a}^2}{\epsilon^{|a|}} \prod^{N}_{i=1}a_i! ,$$
where  $ a_i!=a_i\cdot (a_i-1)\cdots1>0$ and  $0!=1$.

\vskip2mm
Let $r>0$, we obtain that
\begin{align*}
\Big|\int_{B_R(0)}  \frac{\nabla u\cdot\nabla \xi_\epsilon}{\sqrt{1-|\nabla u|^2}} dx\Big|  &=  \frac1{\epsilon} \Big|\int_{B_R(0)}  \frac{\nabla u(x)\cdot\nabla \xi_a(\epsilon^{-1}x)}{\sqrt{1-|\nabla u|^2}}\,dx\Big|   \\
    &\le   \frac1{\epsilon}  \left[\int_{B_R(0)\setminus B_r(0)} \frac{\nabla u(x)\cdot\nabla \xi_a(\epsilon^{-1}x)}{\sqrt{1-|\nabla u|^2}} |\, dx
    +\int_{B_r(0)} \frac{|\nabla u(x)|  |\nabla \xi_a(\epsilon^{-1}x)|}{\sqrt{1-|\nabla u|^2}} \, dx\right].
\end{align*}
For fixed $r>0$, we see that
$$| \nabla \zeta_a(\epsilon^{-1} x)|\to 0\quad{\rm as}\quad \epsilon\to0\quad {\rm uniformly\ in}\quad  B_R(0)\setminus B_r(0),$$
then
\begin{eqnarray*}
\int_{B_R(0)\setminus B_r(0)} \frac{\nabla u(x)\cdot\nabla \xi_a(\epsilon^{-1}x)}{\sqrt{1-|\nabla u|^2}}   dx  \to  0 \quad{\rm as}\quad \epsilon\to0.
\end{eqnarray*}
Furthermore,
\begin{align*}
\int_{B_r(0)}  \frac{|\nabla u(x)|  |\nabla \xi_a(\epsilon^{-1}x)|}{\sqrt{1-|\nabla u|^2}} \, dx &\le  \norm{ \xi_a}_{C^1(\R^N)}  \int_{B_r(0)} \frac{|\nabla u(x)| }{\sqrt{1-|\nabla u|^2}}\,dx  \to  0\ \ {\rm as}\  r\to0.
\end{align*}
Then we have that
\begin{equation}\label{3.4}
\Big|\int_{B_R(0)}  \frac{\nabla u\cdot\nabla \xi_\epsilon}{\sqrt{1-|\nabla u|^2}} \,dx\Big|=\epsilon^{-1} o(1).
\end{equation}

For $|a|\ge 1$, we have that
$$k_{a}^2\le c_7\epsilon^{|a|-1}  o(1) \to 0\quad{\rm as}\quad \epsilon\to0,$$
then we have $k_{a}=0$ by arbitrary of $\epsilon$ in (\ref{3.2}). Thus, (\ref{3.0}) holds.\hfill$\Box$\medskip

\noindent {\bf Proof of Proposition \ref{cr 2.1-iso}. } From Lemma \ref{pr 2.1-iso},  it implies that  the expression (\ref{S}) reduces to
\begin{equation}\label{3.5}
\cT_{u,R} =\sum_{j=1}^{m_0} k_{p_j}  \delta_{p_j}  \quad {\rm in}\ \cD'(B_R(0)),
\end{equation}
where $\langle \delta_{p_j},\xi\rangle=  \xi(p_j)$. The test function's space  could reduces from
$C^\infty_c(B_R(0))$ to $C^{0,1}_0(B_R(0))$ by the identity  (\ref{3.5}). 
\hfill$\Box$\medskip

 \subsection{Radial light-cone singular solution for $N\geq 3$   }

 When $m=1$, we deal with the fundamental solution of 
 \begin{equation}\label{eq 2.1-N}
 \left\{
\begin{array}{lll}
\displaystyle \  \cM_0u = \alpha  \delta_{0}\quad
  {\rm in}\  \, \R^N, 
\\[2mm]
 \phantom{  }
 \displaystyle \lim_{|x|\to+\infty}u(x)=0,    
 \end{array}
 \right.
 \end{equation}
where $\alpha  >0$ and $N\geq 3$. 
 \begin{proposition}\label{pr 2.10-N}
 Let  
 \begin{equation}\label{fun 1-N}
\Phi_{N, \alpha }(x)=   c_{N} \int_{|x|}^{\infty}   \frac{\alpha }{ \sqrt{s^{2(N-1)}+c_{N}^2 \alpha^2}}\, ds  \quad {\rm for}\ \,  x\in\R^N,
\end{equation}
where $c_{N  }=\frac{1 }{|\partial B_1(0)|}$. 
Then $\Phi_{N,\alpha}$ is a  solution of (\ref{eq 2.1-N}).  
Moreover,  we have 
$$\lim_{|x|\to+\infty} \Phi_{N,\alpha}(x)|x|^{N-2}=c_N\alpha,$$
$$\lim_{\alpha\to +\infty}  \Phi_{N,\alpha }=+\infty,  \;\; \lim_{\alpha\to+\infty} |\nabla \Phi_{N,\alpha}(x)|\, =1
 \;\;
{\rm and}\quad  \lim_{\alpha\to 0^+ } \Phi_{N,\alpha}=0\quad   \text{ uniformly locally in  $\R^N$}. $$
\end{proposition} 

\noindent {\bf Proof. }  
 For the radial function $u(r)=u(x)$ with $r=|x|>0$, 
 $$-\cM_0u(x)= \nabla\cdot\Big(\frac{\nabla u}{\sqrt{1-|\nabla u|^2}} \Big)=\frac1{r^{N-1}} \Big(\frac{r^{N-1} u'(r)}{\sqrt{1-| u'(r)|^2}} \Big)' . $$
If $\cM_0u =0$, then  for some $t\in\R\setminus \{0\}$
 $$\frac{r^{N-1} u'(r)}{\sqrt{1-| u'(r)|^2}}=t\quad {\rm for}\ r>0, $$
 and we get that 
 $$u'(r)^2=\frac{t^2}{r^{2(N-1)}+t^2}\quad {\rm for}\ r>0.  $$

By the decay $\displaystyle \lim_{|x|\to+\infty}u(x)=0$,  one has the solution form    
\begin{equation}\label{fun 1-N}
u_t(x):=t \int_{|x|}^{\infty}   \frac{1}{ \sqrt{s^{2(N-1)}+t^2}}\, ds
\end{equation}
and   for $\phi\in C_c^1(\R^2)$
\begin{align*}
0&=\lim_{\epsilon\to0^+} \int_{\R^N \setminus B_\epsilon(0)} \cM_0 u_t (x) \phi(x) dx
\\[2mm]&=\lim_{\epsilon\to0^+} \int_{\R^N\setminus B_\epsilon(0)} \big(\frac{\nabla u_t}{\sqrt{1-|\nabla u_t|^2}} \big)  \cdot \nabla  \phi dx
 -\lim_{\epsilon\to0^+}\int_{ \partial B_\epsilon(0)} \big( \frac{\nabla u_t}{\sqrt{1-|\nabla u_t|^2}} \big)  \cdot \nu    \phi d\omega(x) 
\\[2mm]&= \int_{\R^N } \big(\frac{\nabla u_t}{\sqrt{1-|\nabla u_t|^2}} \big)  \cdot \nabla  \phi dx
 -|\partial B_1(0)|  t\, \phi(0), 
  \end{align*}
 where $\nu=-\frac{x}{|x|}$. 
That is,    for $\phi\in C_c^1(\R^N)$, 
 \begin{align*}
 \int_{\R^N } \big(\frac{\nabla u_t}{\sqrt{1-|\nabla u_t|^2}} \big)  \cdot \nabla  \phi\, dx=|\partial B_1(0)| t \, \phi(0), 
    \end{align*}
which implies that 
$$t= c_{N,\alpha}:=\frac{\alpha}{|\partial B_1(0)|}=c_N \alpha\quad{\rm and}\quad \Phi_{N, \alpha}= u_{c_{N,\alpha}}. $$  

Note that for any $R$, $c_{N,\alpha}>R>1$ if $\alpha$ is large. 
 For any $x\in B_R(0)$ and $\alpha_1>c_N R$ 
 \begin{align*}
 \Phi_{N, \alpha}(x)&\geq  c_{N,\alpha} \Big( \int_{R}^{c_{N,\alpha} }   \frac{1}{ \sqrt{s^{2(N-1)}+c_{N,\alpha}^2}}\, ds+\int_{c_{N,\alpha}}^{\infty}   \frac{1}{ \sqrt{s^{2(N-1)}+c_{N,\alpha}^2}}\, ds\Big)
 \\[1mm] & \geq  \frac1{\sqrt{2}}    \int_{R}^{c_{N,\alpha}} 1 ds =  \frac1{\sqrt{2}}   (c_{N}\alpha-R) 
    \\[1mm] & \to+\infty\quad{\rm as}\ \,  \alpha\to+\infty
     \end{align*}
     and
      for any $x\in \R^N$  and $ c_{N,\alpha}<1$, letting $r_1=c_{N,\alpha}^{\frac{1}{2(N-2)}}$,
 \begin{align*}
 \Phi_{N, \alpha}(x)&\leq  c_{N,\alpha} \Big( \int_{ r_1 }^{\infty}   s^{1-N} \, ds   + \int_{0}^{r_1}  c_{N,\alpha}^{-1}\, ds\Big)
 \\[1mm] &=  \frac1{N-2}c_{N,\alpha}^{\frac12}+r_1
 \\[1mm] &  \to0^+\quad{\rm as}\ \,  \alpha\to0^+. 
     \end{align*}
Finally, since 
 \begin{align*}
 |\nabla \Phi_{N,\alpha}(x)|=\frac{c_N \alpha}{\sqrt{|x|^{2(N-1)}+c_N^2\alpha^2}}, 
      \end{align*}
we see that for any $|x|$ bounded and $\alpha > \frac{|x|^{N-1}}{c_N }$,
 \begin{align*}
0 \leq  1- |\nabla \Phi_{N,\alpha}(x) | \leq \frac{|x|^{N-1}}{c_N \alpha} \to 0, \quad {\rm as}\ \, \alpha\to+\infty, 
      \end{align*}
that is 
$$\lim_{\alpha\to+\infty} |\nabla \Phi_{N,\alpha}(x)|\, =1\quad \text{ uniformly locally in $\R^N$}.$$
   \hfill$\Box$\medskip

\begin{corollary}\label{pr 2.10-N-1}
 When $N\geq 3$,  fix $\bar \alpha>0$ and for $\alpha\geq \bar \alpha$, let  
 \begin{equation}\label{fun 1-N}
\tilde \Phi_{N, \alpha}(x)=\Phi_{N, \alpha }(x) +\Phi_{N, \bar \alpha }(0) -\Phi_{N, \alpha }(0)  \quad {\rm for}\ \,  x\in\R^N. 
\end{equation}
Then  $\tilde \Phi_{N,\alpha}$ is a  solution of 
 \begin{equation}\label{eq 2.1-N0}
 \left\{
\begin{array}{lll}
\displaystyle \quad \cM_0u = \alpha  \delta_{0}\quad
  {\rm in}\  \, \R^N, 
\\[2mm]
 \phantom{  }
 \displaystyle \lim_{|x|\to+\infty}u(x)=  \Phi_{N, \bar \alpha }(0) -\Phi_{N, \alpha }(0)  
 \end{array}
 \right.
 \end{equation}
and 
$$\tilde \Phi_{N, \alpha }(0)=\Phi_{N, \bar \alpha }(0),\qquad  \tilde \Phi_{N, \alpha }<\Phi_{N,\bar \alpha }\ \ {\rm in}\ \, \R^N\setminus\{0\}.  $$

\end{corollary}
\noindent{\bf Proof. } Since $\tilde \Phi_{N, \alpha }$ and $  \Phi_{N, \alpha }$ are radially symmetric, we use the
notation 
$$   \Phi_{N, \alpha }(r)=  \Phi_{N, \alpha }(x),\qquad \tilde \Phi_{N, \alpha }(r)= \tilde \Phi_{N, \alpha }(x)\quad {\rm for}\ \, r=|x|,\ \,  x\in\R^N.  $$ 
Note that by the assumption $\alpha>\bar \alpha$, 
\begin{align*}
\tilde \Phi_{N, \alpha }'(r)&= - c_{N}    \frac{\alpha }{ \sqrt{|x|^{2(N-1)}+c_{N}^2 \alpha^2}}
\\[1mm]&<- c_{N}  \frac{\bar \alpha }{ \sqrt{|x|^{2(N-1)}+c_{N}^2\bar \alpha^2}} =  \Phi_{N, \bar \alpha }'(r)
\end{align*}
  and 
 $ \tilde \Phi_{N, \alpha }(0)=\Phi_{N, \bar \alpha }(0)$,
 then 
  $$ \tilde \Phi_{N, \alpha }<\Phi_{N,\bar \alpha }\ \ {\rm in}\ \, \R^N\setminus\{0\}.$$
  We complete the proof. \hfill$\Box$\medskip

    \subsection{Radial singular solution for $N=2$   }

We deal with the fundamental solution of
 \begin{equation}\label{eq 2.1}
 \left\{
\begin{array}{lll}
\displaystyle \cM_0u = \alpha \delta_{0}\quad
  {\rm in}\  \, \R^2, 
\\[2mm]
 \phantom{\  }
 \displaystyle u(0)=0,    
 \end{array}
 \right.
 \end{equation}
where $\alpha>0$. 
 \begin{proposition}\label{pr 2.10}
 Let  
 \begin{equation}\label{fun 1}
\Phi_{2,\alpha}(x)=-  \frac{\alpha}{2\pi} 
\Big( \ln \Big(  r+ \sqrt{ \big(\frac{\alpha}{2\pi}\big)^2+r^2}\, \Big)-\ln  \big(\frac{\alpha}{2\pi}\big) \Big)   \quad {\rm for}\ \, r=|x|>0, 
\end{equation}
 then $\Phi_{2,\alpha}$ is a  solution of (\ref{eq 2.1}).
 Furthermore, we have that 
 \begin{equation}\label{fun 1-2-0}
 |\nabla \Phi_{2,\alpha}(x)|\to 1\quad {\rm as}\ |x|\to0^+
 \end{equation}
and 
 \begin{equation}\label{fun 1-2-00}
 \nabla \Phi_{2,\alpha}(x) \cdot \frac{x}{|x|}\to -1\quad {\rm as}\ |x|\to0^+. 
  \end{equation}
\end{proposition} 

\noindent {\bf Proof. }  
 For the radial function, we use the notation: $u(x) =u(r)$ with $r=|x|>0$, 
$\cM_0u =0$ becomes to, for some $c\in\R\setminus \{0\},$
 $$\frac{r u'(r)}{\sqrt{1-| u'(r)|^2}}=c\quad {\rm for}\ r>0, $$
 which is equivalent to
 $$u'(r)^2=\frac{c^2}{r^2+c^2}\quad {\rm for}\ r>0.$$
Under the assumption $u(0)=0$,  we can get
the following:  for some $c\in \R\setminus \{0\}$
\begin{equation}\label{fun 1}
u_c(r)=   c\big( \ln\big(  r+ \sqrt{c^2+r^2}\big)-\ln |c|\big) \quad {\rm for}\ \, r>0
\end{equation}
and 
\begin{equation}\label{fun 2}
 \nabla u_c(x)= c  \frac{1}{\sqrt{|x|^2+c^2}}\,  \frac{x}{|x|},\qquad \frac{\nabla u_c}{\sqrt{1-|\nabla u_c|^2}} = c\frac{x}{|x|^2}. 
 \end{equation}
As for the case of $N\geq 3$, we have 
for $\phi\in C_c^1(\R^2)$
\begin{align*}
0&=\lim_{\epsilon\to0^+} \int_{\R^2\setminus B_\epsilon(0)} \cM_0 u_c (x) \phi(x) dx
\\[2mm]&=  \int_{\R^2 } \big(\frac{\nabla u_c}{\sqrt{1-|\nabla u_c|^2}} \big)  \cdot \nabla  \phi dx
 - \phi(0)\lim_{\epsilon\to0^+} \int_{ \partial B_\epsilon(0)} \big( \frac{\nabla u_c}{\sqrt{1-|\nabla u_c|^2}} \big)  \cdot \nu  \,  d\omega(x) 
\\[2mm]&= \int_{\R^2 } \big(\frac{\nabla u_c}{\sqrt{1-|\nabla u_c|^2}} \big)  \cdot \nabla  \phi dx
 +2\pi c\, \phi(0), 
  \end{align*}
 where $\nu=-\frac{x}{|x|}$. 
That is,    for $\phi\in C_c^1(\R^2)$, 
 \begin{align*}
 \int_{\R^2 } \big(\frac{\nabla u_c}{\sqrt{1-|\nabla u_c|^2}} \big)  \cdot \nabla  \phi\, dx=-2\pi c \, \phi(0), 
    \end{align*}
which implies that 
$$c=-\frac{\alpha }{2\pi}\quad{\rm and}\quad \Phi_{2,\alpha }= u_{-\frac{\alpha}{2\pi}}. $$ 
 
Therefore we obtain  the fundamental solution of $\cM_0$ with a single Dirac mass and the estimates of (\ref{fun 1-2-0}) and (\ref{fun 1-2-00}) follow by (\ref{fun 2}). 
 \hfill$\Box$\medskip

\begin{corollary}\label{pr 2.10-N-1}
 When $N=2$, let $\alpha>\bar \alpha$, then 
 \begin{equation}\label{fun 1-N-2}
 \Phi_{2, \bar \alpha }(x)  <\Phi_{2, \alpha  }(x)  \quad {\rm for}\ \,  x\in\R^N. 
\end{equation}

\end{corollary}

Since $\Phi_{N,\alpha}$ is radially symmetric, we use the notations 
$\Phi_{\alpha,N}(r)=\Phi_{\alpha,N}(x)$ for $x\in\R^N$ and $r=|x|$ in the sequel. 

 \setcounter{equation}{0}
 
 \section{ Multiple Dirac masses in bounded domain}
 
For the multiple Dirac masses,  we first consider  the related problem in bounded problem 
\begin{equation}\label{eq 3.1}
 \left\{
\begin{array}{lll}
\displaystyle  \cM_0 u= \sum^{m_0}_{j=1} \alpha_j \delta_{p_j}\quad
 & {\rm in}\  \   B_R(0), 
\\[4mm]
 \phantom{ -\ \ \,   }
 \displaystyle u=0 \quad
 & {\rm on}\  \, \partial B_R(0),    
 \end{array}
 \right.
 \end{equation}
 where $R>R_0$,  $B_R\subset \R^N$ with $N\geq 2$ and 
 $$\cP_{m_0}\subset B_{\frac12R_0}(0). $$

 \begin{theorem}\label{pr 2.1}
 Let $N\geq 2$, 
 $$\alpha_j>0\ \ {\rm for}\ j=1,\cdots, m_0, \quad{\rm and}\quad  \alpha_0=\sum^{m_0}_{j=1} \alpha_j, $$ 
 then there exist $\theta_0 \geq 1$ such that for $R\geq \theta_0R_0$, problem (\ref{eq 3.1}) has unique weak solution $u_{R}\in C^{2,\gamma}_{\loc}(B_{R}(0)\setminus \cP_{m_0})\cap C^{0,1}(B_{R}(0))\cap C_0(B_{R}(0))$, which is positive in $B_{R}(0)$ and is a classical solution of 
 \begin{equation}\label{eq 3.1-cla-R}
 \left\{
\begin{array}{lll}
\displaystyle  \cM_0 u= 0\ \ 
 & {\rm in}\  \, B_R(0)\setminus\cP_{m_0}, 
\\[2mm]
 \phantom{ -\ \,   }
 \displaystyle u=0 \quad
 & {\rm on}\  \, \partial B_R(0).  
 \end{array}
 \right.
 \end{equation}
  
 Moreover, 
 
 $(i)$   $u_R$ is   the minimizer of the energy functional  
  \begin{equation}\label{eq 3.1-do 2}
\cJ_{N,R}(w)= \int_{B_R(0)}\big(1-\sqrt{1-|\nabla w|^2}\big)\, dx -\sum_{j=1}^{m_0}\alpha_j w(p_j)  \quad {for\ }\, w\in \bX_0(B_R(0)),  
 \end{equation} 
where 
$$\bX_0(B_R(0)):=\big\{v\in C^{0,1}(\overline B_R(0))\!:\, v=0\ {\rm on}\ \partial B_R(0),\ |\nabla v|\leq 1\ \ {\rm a.e.\ in\ }B_R(0)\big\}.  $$

 $(ii)$  $0<|\nabla u_{R}(x)|<1$ for $x\in B_{R}(0)\setminus \cP_{m_0}$ and
 \begin{equation}\label{e 3.1}
u_{R} \geq \max_{j=1,\cdot,m_0}u_{R,j}\ \ {\rm in} \ B_R(0), \qquad   \max_{x\in B_{R}(0)}u_{R}(x)\leq   v_R(0),  
  \end{equation}
  where 
  $$v_R(x)= \Phi_{N,\alpha_0}(x)-\Phi_{N,\alpha_0}(R)\quad {\rm for} \ x\in B_R(0)$$ 
and 
$$u_{R,j} (x)=\Phi_{N,\alpha_j}(x-p_j)-\Phi_{N,\alpha_j}(R_j)\quad {\rm for} \ x\in B_{R_j}(p_j),
$$  which is the radially symmetric weak solution of 
   \begin{equation}\label{eq 3.1-j}
  \left\{
 \begin{array}{lll}
 \displaystyle  \cM_0 u=  \alpha_j\delta_{p_j}\quad
 & {\rm in}\  \   B_{R_j}(p_j), 
 \\[2mm]
  \phantom{ -\ \,   }
  \displaystyle u=0 \quad
 & {\rm on}\  \, \partial B_{R_j}(p_j),    
  \end{array}
  \right.
 \end{equation}
 where $R_j=R-|p_j|>R_0$.

$(iii)$ For $R_2>R_1>\theta_0R_0$,  there holds 
$$u_{R_2}>u_{R_1}\ \ {\rm in}\ B_{R_1}(0).   $$

 \end{theorem}

 \begin{remark}
 The domain $B_R(0)$ in (\ref{eq 3.1}) could be replaced by $\Omega$,
 which satisfies $B_{\theta_0 R_0}(0)\subset \Omega$.  
 
 \end{remark}
 
 \subsection{Approximation}
 
Let $\eta_0:[0,+\infty)$ be an $C^2$, non-increasing function such that 
 $$\eta_0(s)=1\ \ {\rm for}\ s\in[0,1],\quad\  \eta_0(s)>0\ \ {\rm for}\ s\in(1,2), \quad\ \eta_0(s)=0\ \ {\rm for}\ s\in[2,+\infty). $$
Given $n\in \N$, let  
$$\eta_n(x)=\frac{c_N  n^N}{\displaystyle\int_0^2 \eta_0(s)s^{N-1}ds}  \eta_0(n |x|)\quad {\rm for\ any}\ \, x\in\R^N,    $$
where $c_N=\frac{1}{|\partial B_1(0)|}$.

Observe that $\eta_n$ is radially symmetric, non-increasing and $C^2$ function and 
$$\lim_{n\to +\infty} \eta_n\to \delta_0\quad \ \text{ in the sense of distribution.} $$
Let 
\begin{equation}\label{souce-n}
g_n(x) := \sum^{m_0}_{j=1} \alpha_j \eta_n(x-p_j) \quad {\rm for\ any}\ \, x\in\R^N, 
 \end{equation}
then $\{g_n\}_{n\in\N}\in C^2(\R^2)$ is  a sequence of smooth nonnegative functions  such that 
$${\rm supp}(g_n)\subset\, \bigcup_{j=1,\cdots, m_0} B_{\frac2{n}}(p_j)   $$
and
 $$ g_n\to \sum^{m_0}_{j=1} \alpha_j \delta_{p_j}\ \text{ in the sense of distribution}\ {\rm as}\ \, n\to +\infty. $$

 To show the existence of solution of (\ref{eq 3.1}),   we need to consider the  approximation problem
\begin{equation}\label{eq 3.1-n}
 \left\{
\begin{array}{lll}
\displaystyle\cM_0 u = g_n\quad
 & {\rm in}\  \   B_R(0), 
\\[2mm]
 \phantom{-\ \,  }
 \displaystyle u=0 \quad
 & {\rm on}\  \, \R^N\setminus  B_R(0).    
 \end{array}
 \right.
 \end{equation}

\begin{lemma}\label{lm 3.1}

Let $N\geq2$,  $R>R_0$ and $g_n$ be defined in (\ref{souce-n}), then  problem (\ref{eq 3.1-n}) has a unique classical solution $u_{n, R}>0$ in $B_{R}(0)$.  
Moreover, we have 

$(i)$ $|\nabla u_{n, R}|<1$ in $B_{R}(0)$ and 
$$u_{n, R} \geq \max_{j=1,\cdot,m_0}u_{n, R, j}\ \ {\rm in} \ B_R(0), \qquad  \max_{x\in B_R(0)}u_{n, R}(x)\leq v_{n, R}(0),$$
where $v_{n, R}$ is the unique solution of 
 \begin{equation}\label{eq 3.1-n-ral}
 \left\{
\begin{array}{lll}
\displaystyle\cM_0 u = \alpha_0\eta_n\quad
 & {\rm in}\  \   B_R(0), 
\\[2mm]
 \phantom{-\ \,  }
 \displaystyle u=0 \quad
 & {\rm on}\  \, \partial B_R(0)  
 \end{array}
 \right.
 \end{equation}
 and
$u_{n, R, j}$ is the unique solution of 
 \begin{equation}\label{eq 3.1-n-j}
 \left\{
\begin{array}{lll}
\displaystyle\cM_0 u = \alpha_j\eta_n(\cdot-p_j)\quad
 & {\rm in}\  \   B_{R_j}(p_j), 
\\[2mm]
 \phantom{-\ \,  }
 \displaystyle u=0 \quad
 & {\rm on}\  \, \partial B_{R_j}(p_j).
 \end{array}
 \right.
 \end{equation}

$(ii)$ For $R_2>R_1>R_0$ and any $n\in\N$,  there holds 
$$u_{n, R_2}>u_{n, R_1}\ \ {\rm in}\ B_{R_1}(0).   $$

$(iii)$ There exists $\theta=\theta(n,R)\in (0,1)$ such that 
$$|\nabla u_{n, R}|\leq \theta. $$

$(iv)$  There holds 
   \begin{align}\label{eee-1}
 \int_{B_R(0)}   \frac{\nabla u_{n, R}\cdot \nabla  \phi}{\sqrt{1-|\nabla u_{n, R}|^2}}  \, dx=  \int_{B_R(0)}  g_n(x) \phi(x)dx \quad {\rm for\ any}\ \, \phi\in C^{0,1}_0(B_R(0)). 
    \end{align}
\end{lemma}
 \noindent{\bf Proof.  }  {\it 1. Existence: } The existence follows by \cite[Theorem 4.1]{BS82} or \cite[Corollary 4.3]{BS82}. In fact,  the solution $u_{n, R}$ is the maximizer of the energy functional 
 $$\cI_{n, R}(w)= \int_{B_R(0)}\Big( \sqrt{1-|\nabla w|^2} +w g_n\Big)dx \quad {for\ }\, w\in \bX_0(B_R(0)). 
$$ 
 
 Note that  $|\nabla u_{n, Rn}|<1$ in $B_{R}(0)$ follows by Lemma \ref{lm 2.1-reg}
 and $u_{n, R}$ is a classical solution of  (\ref{eq 3.1-n}). \smallskip

Similarly, we can obtain   classical solutions  $v_{n, R}, $ $u_{n, R, j}$ of 
 (\ref{eq 3.1-n-ral}) and (\ref{eq 3.1-n-j}) respectively.\smallskip

 {\it 2. Uniqueness: }   
 The uniqueness follows by  \cite[Proposition 1.1]{BS82}. \smallskip

 {\it 3. Bounds: } $(i)$  Since $g_n=\sum_{i=1}^{m_0}\alpha_i \eta_n(x-p_i)\geq \alpha_j \eta_n(x-p_j)$, then   follows by (the comparison principle) Lemma \ref{lm cp} that 
$$u_{n, R} \geq  u_{n, R, j}\quad {\rm for\ any}\  j=1,\cdots,m_0.$$

 Now we show $\max_{x\in B_R(0)}u_{n, R}(x)\leq v_{n, R}(0)$. In fact, we see that the rearrangement of $u_{n, R}$, denote $u^*_{n, R}$, by Lemma \ref{lm 2.2-rea}, which is a sub-solution of  
 \begin{equation}\label{eq 3.1-n*}
 \left\{
\begin{array}{lll}
\displaystyle\cM_0 u = g_n^*\quad
 & {\rm in}\  \   B_R(0), 
\\[2mm]
 \phantom{-\ \,  }
 \displaystyle u=0 \quad
 & {\rm on}\  \, \partial B_R(0),    
 \end{array}
 \right.
 \end{equation}
where $g_n^*$ is the re-arrangement of $g_n$. 

Since $g_n=\sum_{j=1}^{m_0}\alpha_j \eta_n(x-p_j)$ and $\alpha_0=\sum_{j=0}^{m_0}\alpha_j$,
then $g_n^*=\alpha_0\eta_n$ and  $v_{n, R}$ is the solution of problem (\ref{eq 3.1-n*}), which is 
radially symmetric, decreasing with respect to $|x|$. 

Then by Lemma \ref{lm cp}, we have that 
$$u^*_{n, R}\leq  v_{n, R}\quad {\rm in}\ B_R(0). $$
Particularly,
$$\max_{x\in B_R(0)}u_{n, R}(x)=u^*_{n, R}(0)\leq v_{n, R}(0).  $$

$(ii)$ Note that $u_{n, R_2}$ verifies that 
$$\left\{
\begin{array}{lll}
\displaystyle\cM_0 u_{n, R_2} = g_n \quad
 & {\rm in}\  \   B_{R_1}(0), 
\\[2mm]
 \phantom{- - }
 \displaystyle u_{n, R_2}>0 \quad
 & {\rm on}\  \, \partial B_{R_1}(0),    
 \end{array}
 \right.
$$
then again by Lemma \ref{lm cp}, we have that 
 \begin{equation}\label{com sol-R}
u_{n, R_1}\leq  u_{n, R_2}\quad {\rm in}\ B_{R_1}(0). 
 \end{equation}

$(iii)$ We apply \cite[Theorem 3.6]{BS82} to obtain that there is
$\theta\in(0,1)$ depending on $n, R$ such that 
$$|\nabla u_{n, R}|\leq \theta\quad {\rm in}\ \, B_R(0).$$

$(iv)$ Since $|\nabla u_{n, R}|$ is away from 1 uniformly, then from the equation (\ref{eq 3.1-n}),
we derive (\ref{eee-1}). 
  \hfill$\Box$\medskip

 \begin{lemma}\label{lm 3.2}  
 Let $\alpha>0, \, N \geq 2$ and $v_{\alpha,n,R}$ be the radial unique solution of 
 \begin{equation}\label{eq 3.1-n-rala}
 \left\{
\begin{array}{lll}
\displaystyle\cM_0 u = \alpha \eta_n\quad
 & {\rm in}\  \   B_R(0), 
\\[2mm]
 \phantom{-\ \,  }
 \displaystyle u=0 \quad
 &{\rm on}\  \, \partial B_R(0).  
 \end{array}
 \right.
 \end{equation}

 $(i)$ When $N\geq 3$, we have for any $R>R_0$, 
 $$ -\frac2n- c_N \alpha   R^{2-N} <v_{\alpha,n,R}(0)-\Phi_{N,\alpha }(0)<0 $$ 
 and for $R_0<|x|<R$
 $$ v_{\alpha,n,R}(0) >\int_{|x|}^{R} \frac{  c_N \alpha    }{\sqrt{r^{2(N-1)}+ ( c_N \alpha  )^2}}dr. $$ 
 
 $(ii)$ When $N=2$, 
 then $v_{\alpha,n,R}>0$ in $B_R(0)$ and for any $R>2$ and $x\in B_R(0)$,
  \begin{equation}\label{com sol-R00}
\min\big\{ \Phi_{2,\alpha}(x), \Phi_{2,\alpha}(\frac2n)\big\}  \leq v_{\alpha,n,R}(x) + \Phi_{2,\alpha}(R) \leq \Phi_{2,\alpha}(x)  .
 \end{equation}

\end{lemma}
\noindent{\bf Proof. }  It follows by the directional computation that 
\begin{align*}
v_{\alpha,n,R}(0)=\int_0^R\sqrt{\frac{ h_{\alpha }(r) ^2}{r^{2(N-1)}+ h_{\alpha }(r) ^2}}\, dr,
\end{align*}
where $h_{\alpha }(r)=\int_0^r \alpha  \eta_n(\tau) \tau^{N-1}d\tau$ and we used the fact that  $v_{\alpha,n,R}(x)=0$ for $x\in\partial B_R(0)$. 
Note that for $r>\frac{2}{n}$
$$h_{\alpha}(r)\left\{
\begin{array}{lll}
 =c_N \alpha  \quad
 & {\rm for}\  \   r\geq \frac2n, 
\\[2.5mm]
<c_N  \alpha  \quad
 & {\rm for}\  \,  r\in[0,\frac2n),
 \end{array}
 \right.  
  $$
then we see that 
$$ \sqrt{\frac{ h_{\alpha }(r) ^2}{r^{2(N-1)}+ h_{\alpha }(r) ^2}}= \sqrt{\frac{ ( c_N \alpha  ) ^2}{r^{2(N-1)}+( c_N \alpha  )^2}}\quad {\rm for}\  r\geq \frac2n$$
and
$$\sqrt{\frac{ h_{\alpha }(r) ^2}{r^{2(N-1)}+ h_{\alpha }(r) ^2}}< \sqrt{\frac{ ( c_N \alpha  )  ^2}{r^{2(N-1)}+ ( c_N \alpha  ) ^2}}\quad {\rm for}\   r\in[0, \frac2n).$$ 
 So when $N\geq 3$, 
$$v_{\alpha,n,R}(0)<\int_0^{R} \frac{ ( c_N \alpha )  }{\sqrt{r^{2(N-1)}+ ( c_N \alpha  )^2}}dr<\int_0^{+\infty} \frac{ ( c_N \alpha  )  }{\sqrt{r^{2(N-1)}+ ( c_N \alpha  )^2}}dr=\Phi_{N,\alpha }(0)  $$ 
and
\begin{align*}
v_{\alpha,n,R}(0)&>\int_{\frac 2n}^{R} \frac{ ( c_N \alpha  )  }{\sqrt{r^{2(N-1)}+ ( c_N \alpha  )^2}}dr
\\[1mm]&> \int_0^{+\infty} \frac{ ( c_N \alpha  )  }{\sqrt{r^{2(N-1)}+ ( c_N \alpha  )^2}}dr-\frac2n-( c_N \alpha  ) R^{2-N} \\[1mm]&=\Phi_{N,\alpha }(0)-\frac2n-( c_N \alpha ) R^{2-N}.  
\end{align*}
Furthermore, we have that for $R_0<|x|<R$
\begin{align*}
v_{\alpha,n,R}(x)&>\int_{|x|}^{R} \frac{ ( c_N \alpha  )  }{\sqrt{r^{2(N-1)}+ ( c_N \alpha  )^2}}dr.  
\end{align*}

When $N=2$, since $v_{\alpha,n,R}$ is radial symmetric, then we have the upper bound:  
\begin{align*}
v_{\alpha,n,R}(x)=\int_{|x|}^R \frac{  h_{\alpha}(r)     }{\sqrt{r^{2}+ h_{\alpha}(r)^2}}dr  \leq \int_{|x|}^R \frac{  c_2 \alpha     }{\sqrt{r^{2}+ ( c_2 \alpha  )^2}}dr =
\Phi_{2,\alpha}(x)-\Phi_{2,\alpha}(R).
\end{align*}
Furthermore, we can get the lower bound:   for $|x|>\frac2n$, 
\begin{align*}
v_{\alpha,n,R}(x)&=\int_{|x|}^R \frac{  h_{\alpha}(r)     }{\sqrt{r^{2}+ h_{\alpha}(r)^2}}dr
=\Phi_{2,\alpha}(x)-\Phi_{2,\alpha}(R)
\end{align*}
and for $|x|\leq \frac2n$
\begin{align*}
v_{\alpha,n,R}(x)&> \int_{\frac2n}^R \frac{  h_{\alpha}(r)     }{\sqrt{r^{2}+ h_{\alpha}(r)^2}}dr
=\Phi_{2,\alpha}(\frac2n)-\Phi_{2,\alpha}(R). 
\end{align*}
We complete the proof.   \hfill$\Box$\medskip

 \subsection{Multiple singularities on Balls}

 \noindent{\bf Proof of Theorem \ref{pr 2.1}. }   
  {\bf Existence: }
From  Lemma \ref{lm 3.1}, we have that when $N\geq 3$,
\begin{equation}\label{ine -3.1}
0<u_{n, R}(x)\leq \Phi_{N,\alpha_0}(0)\quad {\rm for}\ \, x\in  B_R(0)
\end{equation}
and for $N=2$, 
\begin{equation}\label{ine -3.1-2}
0<u_{n, R}(x)\leq \frac{\alpha_0}{2\pi} \Big( \ln \Big(  R+ \sqrt{ \big(\frac{\alpha_0}{2\pi}\big)^2+R^2}\, \Big)-\ln  \big(\frac{\alpha_0}{2\pi}\big) \Big)\quad {\rm for}\ \, x\in  B_R(0). 
\end{equation}
 So we can choose $R>R_0$ large  such that 
 $$u_{n, R}(x)<\Phi_{N,\alpha_0}(0)\ \, {\rm for}\ N\geq3\quad {\rm and}\quad u_{n, R}(x)<\frac14 R\ \,  \text{when $N=2$}. $$

 {\bf Claim 1:\ }  {\it  There exist a subsequence, still use the notation  $u_{n, R}$,  and $u_{R}$ such that 
 $$|\nabla u_R|\leq 1\quad{\rm a.e.\ in}\ \bar B_R(0) $$  
and 
$$u_{n, R}\to u_R\quad\text{in $B_R(0)$ and in } \, C^{0,\gamma}(B_R(0))\quad {\rm as}\ \, n\to+\infty.   $$}
In fact, by  (\ref{ine -3.1}) and $|\nabla u_{n, R}|\leq \theta_n<1$, then for any $\gamma\in(0,1)$, the Arzel-Ascoli theorem
there is a subsequence $u_{n_k, R}$ and $u_R$ such that 
$$u_{n_k, R}\to u_R\quad\text{uniformly in $B_R(0)$ and in } \, C^{0,\gamma}(B_R(0))\quad {\rm as}\ \, k\to+\infty.   $$
Fix  $x,y\in\R^N$, $0<|x-y|<1$,  and for any $\epsilon>0$ and
 we have that if $k$ large enough such that $|u_{n_k, R}(x)-u_R(x)|, \,  |u_{n_k, R}(y)-u_R(y)|\leq \epsilon|x-y|$
\begin{align*}
\frac{|u_R(x)-u_R(y)|}{|x-y|}&\leq \frac{|u_{n_k, R}(x)-u_R(x)|}{|x-y|}+\frac{|u_{n_k, R}(x)-u_{n_k, R}(y)|}{|x-y|}+\frac{|u_{R,n_k}(y)-u_R(y)|}{|x-y|}
\\[1mm]&\leq  2\epsilon+\theta_{n_k} <2\epsilon+1 .
\end{align*}
 By the arbitrary of  $\epsilon>0$,  we derive that 
$$\frac{|u_R(x)-u_R(y)|}{|x-y|}\leq 1, $$
which, by Rademacher's theorem, implies that $|\nabla u_R|\leq 1$ a.e. in $B_R(0)$. \medskip

As a result,   we have that 
\begin{equation}\label{ine -3.2} 
0<u_{R}(x)\leq \Phi_{N,\alpha_0}(0)\ \, {\rm for}\ N\geq3\quad {\rm and}\quad 0<u_{R}(x)\leq \frac14 R\ \,  \text{for $N=2$}.  
\end{equation}
We write  
\begin{equation}\label{eq 3.1-ind-1}
\cI_{R}(w):= \int_{B_R(0)}\Big( \sqrt{1-|\nabla w|^2}+\sum^{m_0}_{j=1}\alpha_j w(p_j)\Big)\, dx\;  (=|B_R(0)| - \cJ_{N, R}(w))  \quad {for\ }\, w\in \bX_0(B_{R}(0))
 \end{equation} 
with 
$$\bX_0(B_{R}(0)):=\Big\{v\in C^{0,1}(\bar B_{R}(0)):\,  v= 0\ {\rm on}\ \partial B_{R}(0),\ \, |\nabla v|\leq 1\ \ {\rm a.e.\ in\ } B_{R}(0)\Big\}.  $$

Moreover,  for any $\epsilon\in(0,\frac14\min\{|p_i-p_j|,\ i\not=j\}]$ and $R>R_0$, let $\displaystyle  \cO_{\epsilon}=B_{R}(0)\setminus \cup_{j=1}^{m_0} B_{\epsilon}(p_j)$, 
\begin{equation}\label{eq 3.1-ind-1}
\cI_{R,\epsilon}(w)= \int_{\cO_\epsilon}\sqrt{1-|\nabla w|^2}  \, dx   \quad {for\ }\, w\in \bX_{u_R}(\cO_\epsilon)
 \end{equation} 
with 
$$\bX_{u_R}(\cO_\epsilon):=\Big\{v\in C^{0,1}(\R^N):\,  v= u_R\ {\rm on}\ \partial \cO_\epsilon,\ \, |\nabla v|\leq 1\ \ {\rm a.e.\ in\ }\R^N\Big\}.  $$
Then $u_R$ is weakly spacelike and it follows by \cite[Lemma 1.3]{BS82}  that $u_R$   achieves the maximizer of $\cI_{R,\epsilon}$ over $\cO_\epsilon$.  \medskip

  {\bf Claim 2:\ }      {\it For any $\sigma\in(0,\sigma_0]$, there exists $ \theta_\sigma\in (0,1)$  such that  
$$|\nabla u_R|\leq \theta_\sigma\quad{\rm in}\ \ B_R(0)\setminus \bigcup^{m_0}_{j=1} B_{\sigma}(p_j). $$}

 Let 
\begin{equation}\label{bou spcl-1}
\cK_s=\big\{\overline{xy}\subset  \cO_\epsilon:\, x,y\in\partial\cO_\epsilon, x\not=y, |u_R(x)-u_R(y)|=|x-y|    \big\}.
\end{equation}
 Our aim is to show $\cK_s=\emptyset$.

\smallskip

If not,  we choose $x_1,x_2\in \partial\cO_\epsilon$ such that  $|u_R(x_1)-u_R(x_2)|=|x_1-x_2|$.
$$\bL_{x_1x_2}= \big\{x_t: \text{for $t$ belongs a maximal interval of $\R$ such that } x_t    \in   B_R\setminus \cP_{m_0}\     \big\} , $$
where $x_t=x_1+t(x_2-x_1)$. 
Let $\bar x_,\bar x_2$ be the ends points of $\bL_{x_1x_2}$,   then  either $\bL_{x_1x_2}$ could be extended to cross the boundary $\partial B_R(0)$ twice, i.e. $\bar x_1,\bar x_2\in \partial B_R(0)$  or $\overline{\bL}_{x_1x_2}$ cross the boundary $\partial B_R(0)$ once i.e. 
  $\bar x_1\in \partial B_R(0), \bar x_2\in \cP_{m_0}$  or  $\bL_{x_1x_2}$ stops by two points in $\cP_{m_0}$ i.e. $\bar x_1,\bar x_2\in \cP_{m_0}$. 

We apply \cite[Theorem 3.2]{BS82} to obtain that 
$$
u_R(x_t)=u_R(x_1)+t|x_1-x_2|\quad {\rm for\ all}\  x_t\in \overline{\bL}_{x_1x_2}.
$$
In particular, we have 
 \begin{equation}\label{ind-1-t}
|u_R(\bar x_1)-u_R(\bar x_2)|= |\bar x_1-\bar x_2|.
 \end{equation}

 If $\bar x_1,\bar x_2\in \partial B_R(0)$, then 
 $$|u_R(\bar x_1)-u_R (\bar x_2)|=0<|\bar x_1-\bar x_2|,  $$
 which contradicts (\ref{ind-1-t}). 
 
  If $\bar x_1\in \partial B_R(0), \bar x_2\in \cP_{m_0}$ and we can set 
  $$\bar x_1 \in \bL_{x_1x_2}\cap \partial B_R(0)\quad{\rm and}\quad \bar x_2 \in \cP_{m_0},  $$
 then  $u_R(\bar x_1) =0$ and 
 $|\bar x_1-\bar x_2| \geq R-R_0$. 
So  for $N\geq 3$, 
 $$|u_R(\bar x_1)-u_R (\bar x_2)|=|u_R (\bar x_2)|\leq \Phi_{N,\alpha_0}(0)<|\bar x_1-\bar x_2| $$
 and 
  for $N=2$, 
 $$|u_R(\bar x_1)-u_R (\bar x_2)|=|u_R (\bar x_2)|\leq  \frac14 R<R-R_0\leq |\bar x_1-\bar x_2|$$
 then we get a contradiction with (\ref{ind-1-t}). 
 
 If   $\bar x_1,\bar x_2\in \cP_{m_0}$, we can assume that 
 $$
u_R(\bar x_1)=u_R(\bar x_2)+|\bar x_1-\bar x_2|\quad {\rm for\ all}\;  x_t\in \overline{\bL}_{x_1x_2}.
$$
Let 
$$w_{\alpha}(x)=\Phi_{N, \alpha}(x-\bar x_1)-\Phi_{N, \alpha}(0)+u_R(\bar x_1),\quad x\in B_R(0), $$
then $w_\alpha(\bar x_1)=u_R(\bar x_1)$ and there exist  $\bar \alpha\geq \alpha_j$ such that 
$$w_{\bar \alpha}\leq -1\quad {\rm on}\ \partial B_R(0). $$
Let
$$w_{\bar\alpha,n}(x)=\Phi_{N, \bar\alpha}(x-\bar x_1)-\Phi_{N, \bar\alpha}(0)+u_{n, R}(\bar x_1)\quad{\rm for}\ \, x\in B_R(0), $$
then $w_{\bar\alpha,n}(\bar x_1)=u_{n, R}(\bar x_1)$ and 
$$\lim_{n\to+\infty}w_{\bar\alpha,n}(x)=w_{\bar\alpha}(x)\quad {\rm for}\ \, x\in \overline{B_R(0)} $$
and there exist  $n_0>1$ such that for all $n \geq n_0$, 
$$w_{\bar \alpha,n}<0\quad {\rm on}\ \partial B_R(0). $$

By the  comparison principle, we have that 
$$u_{R,n}\geq  w_{\bar \alpha,n}\quad {\rm in}\ B_R(0), \;\;  \forall n \geq n_0, $$
which implies that 
$$u_{R}\geq  w_{\bar \alpha}\quad {\rm in}\ B_R(0)  $$
and
$$ w_{\bar \alpha}(\bar x_1)-w_{\bar \alpha}(\bar x_2)\geq u_R(\bar x_1)-u_R(\bar x_2)= |\bar x_1-\bar x_2|,  $$
 which contradicts the fact  that    $|\nabla \Phi_{N,\alpha}|<1$ in $\R^N\setminus\{0\}$.  \smallskip

As a consequence, we obtain $\cK_s=\emptyset$ and it follows by \cite[Theorem 4.1, Corollary 4.2]{BS82} 
that $u_R\in C^1(\cO_\epsilon)$ is strictly spacelike in $\cO_\epsilon$ and  
there exists $\theta_\epsilon\in[0,1)$ such that 
 \begin{equation}\label{grandient-1-R}
 |\nabla u_R|\leq \theta_\epsilon\quad {\rm in}\ \, \overline \cO_{2\epsilon}.   
 \end{equation}

Next we show the qualitative properties of $u_R$. 
\smallskip

{\bf Part 1: }   we show that $u_R$ is a weak solution of problem (\ref{eq 3.1}) and a classical solution of (\ref{eq 3.1-cla-R}). \smallskip

Indeed, since $|\nabla u_{n, R}|<1$ and $u_{n, R}=0$ on $\partial B_R(0)$, then $u_{n, R}<R$ in $B_R(0)$. 
 Particularly, we take $\phi=u_{n, R}$ in (\ref{eee-1}) to derive that
 \begin{equation}  
 \int_{B_R(0)} \frac{|\nabla u_{n, R}|^2   }{\sqrt{1-|\nabla u_{n, R}|^2}} dx =\int_{B_R(0)} u_{n, R}  g_n dx  \leq R \int_{B_R(0)}g_n dx=R \sum^{m_0}_{j=1}\alpha_j. \label{bound-1}
    \end{equation}
    

Firstly, we show the uniformly bound that 
 \begin{align*} 
  \int_{B_R(0)} \frac{|\nabla u_{n, R} |   }{\sqrt{1-|\nabla u_{n, R}|^2}} dx &\leq   
2   \int_{B_R(0)\cap \{|\nabla u_{n, R}|\geq \frac12\}} \frac{|\nabla u_{n, R}|^2   }{\sqrt{1-|\nabla u_{n, R}|^2}} dx
 \\&\qquad +   \int_{B_R(0)\cap \{|\nabla u_{n, R}|<\frac12\}} \frac{|\nabla u_{n, R}|   }{\sqrt{1-|\nabla u_{n, R}|^2}} dx
  \\[1mm]& \leq  \Big(2R\sum^{m_0}_{j=1}\alpha_j+\frac{\sqrt{3}}3|B_R(0)|\Big),
  \end{align*}
where we used the bound (\ref{bound-1}) and $\frac{|\nabla u_{n, R}|   }{\sqrt{1-|\nabla u_{n, R}|^2}}\leq \frac{\sqrt{3}}3$ for $|\nabla u_{n, R}|\leq\frac12$. 

For any $\varphi\in C^{0,1}(B_R(0))$ such that $\varphi(x)=\varphi(p_j)$ for $x\in B_\epsilon(p_j)$ for any $j=1,\cdots, m_0$
and $\epsilon>0$ small, then ${\rm supp}(\nabla \varphi)\subset B_R(0)\setminus \bigcup^{m_0}_{j=1} B_\epsilon(p_j)$,  
 \begin{align*} 
 \int_{B_R(0)\setminus\big( \bigcup^{m_0}_{j=1} B_\epsilon(p_j)\big)} \frac{ \nabla u_{n, R} \cdot \nabla \varphi   }{\sqrt{1-|\nabla u_{n, R}|^2}} dx& =  \int_{B_R(0)} \frac{ \nabla u_{n, R} \cdot \nabla \varphi   }{\sqrt{1-|\nabla u_{n, R}|^2}} dx 
 \\[1mm]&=\int_{B_R(0)}g_n \varphi dx
 \\[1mm]& \to \sum^{m_0}_{j=1} \alpha_j \varphi(p_j)
=  \int_{B_R(0)} \frac{ \nabla u_{R} \cdot \nabla \varphi   }{\sqrt{1-|\nabla u_{R}|^2}} dx
 \\[1mm]&= \int_{B_R(0)\setminus\big( \bigcup^{m_0}_{j=1} B_\epsilon(p_j)\big)} \frac{ \nabla u_{R} \cdot \nabla \varphi   }{\sqrt{1-|\nabla u_{R}|^2}} dx, 
  \end{align*}
  that is
 \begin{align} \label{con-1-R}
  \frac{ \nabla u_{n, R}    }{\sqrt{1-|\nabla u_{n, R}|^2}} \to\frac{ \nabla u_{R}    }{\sqrt{1-|\nabla u_{R}|^2}} \quad{\rm weakly\ in\ }L^1\big(B_R\setminus\big( \bigcup^{m_0}_{j=1} B_\epsilon(p_j)\big) \big)^N, 
  \end{align}
then by the upper semicontinuity of the area integral
\begin{align*} 
\int_{B_R(0)\setminus\big( \bigcup^{m_0}_{j=1} B_\epsilon(p_j)\big)} \frac{|\nabla u_{R} |   }{\sqrt{1-|\nabla u_{R}|^2}} dx&\leq   \liminf_{n\to+\infty}  \int_{B_R(0)\setminus\big( \bigcup^{m_0}_{j=1} B_\epsilon(p_j)\big)} \frac{|\nabla u_{n, R} |   }{\sqrt{1-|\nabla u_{n, R}|^2}} dx 
\\[1mm]&  \leq  \Big(2R\sum^{m_0}_{j=1}\alpha_j+\frac{\sqrt{3}}3|B_R(0)|\Big).
  \end{align*}
which,  by the arbitrary of $\epsilon>0$,  implies that 
\begin{align*} 
\int_{B_R(0)} \frac{|\nabla u_{R} |   }{\sqrt{1-|\nabla u_{R}|^2}} dx  \leq  \Big(2R\sum^{m_0}_{j=1}\alpha_j+\frac{\sqrt{3}}3|B_R(0)|\Big).
  \end{align*}

  Thus, we obtain that $u_R\in \bX_R(B_R(0))$,
  where
   $$\bX_R(B_R(0))=\Big\{w\in C^{0,1}_0(B_R(0)):\, |\nabla w|<1\ {\rm\ a.e.\ in}\ B_R(0)\setminus \cP_{m_0},\ \, \frac{|\nabla w|}{\sqrt{1-|\nabla w|^2}}  \in L^1(B_R(0))\Big\}. $$
   Moreover, from (\ref{con-1-R}), we get that for any $\epsilon>0$ small, 
 $$\int_{\cO_\epsilon}  \frac{ \nabla u_{R} \cdot \nabla \varphi   }{\sqrt{1-|\nabla u_{R}|^2}}dx=0\quad \text{  for any $\varphi\in C^{0,1}$ with ${\rm supp}(\varphi)\subset \cO_\epsilon$. } $$
 By (\ref{grandient-1-R}) and  \cite[Theorem 3.6]{BS82},   $u_R\in C^{2,\gamma}(\cO_{2\epsilon})$, by the arbitrary of $\epsilon$, we get that  
$u_R$ verifies the equation (\ref{eq 3.1-cla-R}) in the classical sense.

Now we take  with  ${\rm supp}( \xi)\subset  B_R(0)\setminus\cP_{m_0}$
and 
   $$\int_{B_R(0)}  \frac{\nabla u_R\cdot\nabla \xi}{\sqrt{1-|\nabla u_R|^2}} \,dx =0\quad \text{  for any $\xi\in C^{0,1}_c (B_R\setminus\cP_{m_0})$.} $$ 
   Now we apply Proposition \ref{cr 2.1-iso} to obtain that $u_R$ is a weak solution 
\begin{equation}\label{eq 1.1-R-w-}
 \left\{
\begin{array}{lll}
\displaystyle \cM_0 u= \sum_{j=1}^{m_0} k_{p_j}  \delta_{p_j}   \quad
  {\rm in}\  \, \cD'\big(B_R(0)\big), 
\\[4mm]
 \phantom{    }
 \displaystyle  u(x)=0 \quad {\rm on}\ \partial B_R(0)       
 \end{array}
 \right.
 \end{equation}
for some $k_{p_j}\in \R$ with $j \in \{1,2, \cdots, m_0 \}$.  That is, 
\begin{equation}\label{eq 1.1-R-w-sense}
\int_{B_R(0)}  \frac{\nabla u_R\cdot\nabla \xi}{\sqrt{1-|\nabla u_R|^2}} \,dx =\sum_{j=1}^{m_0}  k_{p_j}  \xi(p_j),\quad\ \ \forall \xi\in C^{0,1}_0 (B_R).
 \end{equation}
 
Now we need to prove $k_{p_j} =\alpha_j$ for any $j=1,\cdots, m_0$. Take $\xi_0\in C^1_0(B_R(0))$
$$\displaystyle \xi_0(x)=\sum_{j=1}^{m_0} b_j 1_{B_{r_0}(p_j)}(x)\quad{\rm for}\ \,  x\in \bigcup_{j=1}^{m_0} B_r(p_j), $$
where $b_j\in\R$ and $r_0=\frac1{16}\min\big\{|p_j-p_{j'}|:\, j\not=j'\big\}$. 

Since $\nabla \xi_0=0$ in $\displaystyle \bigcup_{j=1}^{m_0} B_{r_0}(p_j)$, then for $n$ large, we have that 
supp$\displaystyle  (\xi_0)\subset B_R(0)\setminus \bigcup_{j=1}^{m_0} B_{r_0}(p_j)$, 
\begin{align*} 
\displaystyle \sum^{m_0}_{j=1}k_{p_j} b_j=  \int_{B_R(0)} \frac{\nabla u_{R}\cdot\nabla \xi_0  }{\sqrt{1-|\nabla u_{R}|^2}} dx
&\displaystyle = \int_{B_R(0)\setminus \bigcup_{j=1}^{m_0} B_{r_0}(p_j)} \frac{\nabla u_{R}\cdot\nabla \xi_0  }{\sqrt{1-|\nabla u_{R}|^2}} dx
 \\&\displaystyle =\lim_{n\to+\infty} 
  \int_{B_R(0)\setminus \bigcup_{j=1}^{m_0} B_{r_0}(p_j)} \frac{\nabla u_{n, R}\cdot\nabla \xi_0  }{\sqrt{1-|\nabla u_{n, R}|^2}} dx
  \\&\displaystyle =\lim_{n\to+\infty} 
  \int_{B_R(0)} \frac{\nabla u_{n, R}\cdot\nabla \xi_0  }{\sqrt{1-|\nabla u_{n, R}|^2}} dx
  \\&= \lim_{n\to+\infty}   \int_{B_R(0)} \xi_0  g_n dx 
    \\&= \sum^{m_0}_{j=1}\alpha_jb_j, 
    \end{align*} 
which implies that for any $b_j\in\R$,  $j=1,\cdots,m_0$, 
 $$\sum^{m_0}_{j=1}k_{p_j} b_j= \sum^{m_0}_{j=1}\alpha_jb_j.  $$   
Then $$k_{p_j} =\alpha_j\quad\text{ for  } j=1,\cdots,m_0$$ and 
$u_R$ is a weak solution (\ref{eq 1.1-R-w-}). \smallskip

{\bf Part 2: }   we show that $u_R$ is the unique minimizer of the energy functional (\ref{eq 3.1-do 2}).

Indeed, fix $w\in \bX_R(B_R(0))$ and define
$$\cO_+=\Big\{x\in B_R(0):\, w(x)> u_{R}(x) \Big\},\quad\ \cO_{k,n}=\Big\{x\in B_R(0):\, w(x)-\frac2k> u_{n, R}(x) \Big\}\ \ {\rm with}\ \, k\in\N.  $$
 Since $u_{n, R}\to u_R$  in  $ C_0^{0,1}(B_R(0))$ as $n\to+\infty$, then 
   there exists $n_k\geq k$ such  that    
   $$\sup_{B_R(0)}|u_{n, R}- u_R | \leq \frac1k\quad\text{for $n\geq n_k$,}$$
which implies that $\cO_{k,n}\subset \cO_+$.  Observe that $\bigcup_{k \geq 1}\cO_{k, n_k}= \cO_+$. 

Note that $w(x)-\frac2k=u_{n_k, R}$ on $\partial\cO_{k,n_k}$ and $u_{n_k, R}$ maximizes $\cI_{R,k}$, where 
$$\cI_{R,k}(v):= \int_{\cO_{k,n_k}} \sqrt{1-|\nabla v|^2}\, dx +\sum_{j=1}^{m_0}\alpha_j \big(v(p_j)-\frac{2}{k} \big)1_{\cO_{k,n_k}}(p_j)\quad {for\ }\, w\in \bX_{u_{R,n_k}}(\cO_{k,n_k}),  $$
and
$$\bX_{u_{n_k, R}}(\cO_{k,n_k}):=\Big\{v\in C^{0,1}(\cO_{k,n_k}):\, v=u_{n_k, R} \ {\rm on}\ \partial \cO_{k,n_k},\ |\nabla v|\leq 1\ a.e. \; {\rm in\ }\cO_{k,n_k}\Big\}.  $$ 
 Thus, we have that 
 \begin{align*}
 &\quad \int_{\cO_{k,n_k}} \sqrt{1-|\nabla w|^2}\, dx +\sum_{j=1}^{m_0} \alpha_j w(p_j) 1_{\cO_{k,n_k}}(p_j) 
 \\[1mm]& \leq 
 \int_{\cO_{k,n_k}} \sqrt{1-|\nabla u_{n_k, R}|^2}\, dx +\sum_{j=1}^{m_0}\alpha_j  u_{n_k, R}(p_j)   1_{\cO_{k,n_k}}(p_j) 
  -\frac2k \sum_{j=1}^{m_0}\alpha_j   1_{\cO_{k,n_k}}(p_j) 
  \\[1mm]& \leq 
 \int_{\cO_{k,n_k}} \sqrt{1-|\nabla u_{n_k, R}|^2}\, dx +\sum_{j=1}^{m_0}\alpha_j  u_{n_k, R}(p_j)   1_{\cO_{k,n_k}}(p_j) 
  -\frac2k  \alpha_0.  
  \end{align*}
 Then by using the upper semicontinuity of the area integral, we derive that 
 \begin{align*}
 &\quad\ \int_{\cO_+} \sqrt{1-|\nabla w|^2}\, dx +\sum_{j=1}^{m_0} \alpha_j w(p_j) 1_{\cO_+}(p_j) 
 \\[1mm]& \leq \limsup_{k\to+\infty} \Big(
 \int_{\cO_{k,n_k}} \sqrt{1-|\nabla u_{n_k, R}|^2}\, dx +\sum_{j=1}^{m_0}\alpha_j  u_{n_k, R}(p_j)   1_{\cO_{k,n_k}}(p_j) 
  -\frac2k  \alpha_0 \Big)
  \\[1mm]&\leq  \int_{\cO_+} \sqrt{1-|\nabla u_{R}|^2}\, dx +\sum_{j=1}^{m_0}\alpha_j  u_{R}(p_j)  1_{\cO_+}(p_j).
  \end{align*}
A similar argument applied to $\cO_-=\Big\{x\in B_R(0):\, w(x)< u_{R}(x) \Big\}$ shows that  
 \begin{align*}
  \int_{\cO_-} \sqrt{1-|\nabla w|^2}\, dx +\sum_{j=1}^{m_0} \alpha_j w(p_j) 1_{\cO_-}(p_j)  \leq  \int_{\cO_-} \sqrt{1-|\nabla u_{R}|^2}\, dx +\sum_{j=1}^{m_0}\alpha_j  u_{R}(p_j)  1_{\cO_-}(p_j). 
  \end{align*}
As a consequence, we derive that 
$$\cI_{R}(w)\leq \cI_{R}(u_{R}). $$
Therefore, $u_R$ is the maximizer of $\cI_{R}$. Since $\cJ_{N,R}(u)=-\cI_{R}(u) + |B_R(0)|$,
then $u_R$ is the minimizer of $\cJ_{N,R}$.

\smallskip

{\bf Part 3: }  By Proposition \ref{pr 2.10-N} for $N \geq 3$, and the same arguments, including the uniqueness, show that 
 $u_{R,j}=\Phi_{\alpha_j}(\cdot-p_j)-\Phi_{\alpha_1}(R_j)$ is a weak solution   of (\ref{eq 3.1-j}) 
 and $v_R= =\Phi_{N, \alpha_0}(\cdot)-\Phi_{N, \alpha_0}(R)$ is a weak   solution of   
$$
 \left\{
\begin{array}{lll}
\displaystyle  \cM_0 u=  \alpha_0\delta_{0}\quad
 & {\rm in}\  \   B_{R}(0), 
\\[2mm]
 \phantom{ -\ \,   }
 \displaystyle u=0 \quad
 & {\rm on}\  \, \partial B_{R}(0) .
 \end{array}
 \right.
$$
 
  Moreover,  we have 
$$u_{R,j}\leq u_R,\qquad \max_{x\in B_{R}(0)}  u_R(x)\leq v_R(0)\quad {\rm in}\ B_R(0). $$

Finally, it follows by Lemma \ref{lm 3.1} $(ii)$  that $u_{R_2} \geq  u_{R_1}$ in $B_{R_1}(0)$ by (\ref{com sol-R}).    \hfill$\Box$\medskip

\begin{corollary}\label{cr 3.1-mont-1}
Let $N\geq 3$, 
$$\cP_{m_1}\subset \cP_{m_2}\subset B_{\frac12 R_0}\quad \text{ with $m_2\geq m_1$}$$
 and 
$$0<\alpha_{1,j}\leq\alpha_{2,j}\quad {\rm for}\ \, j=1,\cdots m_1,\quad   \alpha_{2,j}>0\quad {\rm for}\ j>m_1\ \ {\rm if}\ \, m_2>m_1. $$
Let $u_{i}$ with $i=1,2$ be the solutions, respectively, of
\begin{equation}\label{eq 3.1-com}
 \left\{
\begin{array}{lll}
\displaystyle  \cM_0 u= \sum^{m_i}_{j=1} \alpha_{i,j} \delta_{p_j}\quad
 & {\rm in}\  \   B_R(0), 
\\[4mm]
 \phantom{ -\ \,   }
 \displaystyle u=0 \quad
 & {\rm on}\  \, \partial B_R(0),    
 \end{array}
 \right.
 \end{equation}
 where $R>R_0$. 
 Then $u_2\geq u_1$ in $B_R(0)$. 

\end{corollary} 
 \noindent{\bf Proof. } It follows by the construction and uniqueness of solution to (\ref{eq 3.1-com}). \hfill$\Box$\medskip

 \begin{remark}
  By the equality,  $k_{p_j} =\alpha_j$, we can observe that 
   for any $\phi\in C_0^{0,1}(B_R(0))$
 \begin{align*}
 \int_{B_R(0)} \frac{\nabla u_{R} \cdot \nabla \phi  }{\sqrt{1-|\nabla u_{R}|^2}} dx &= \sum^{m_0}_{j=1}\alpha_j\phi(p_j)
=\lim_{n\to+\infty} \int_{B_R(0)} \phi g_n dx
 \\[1mm]&=\lim_{n\to+\infty}\int_{B_R(0)} \frac{\nabla u_{n, R} \cdot \nabla \phi  }{\sqrt{1-|\nabla u_{n, R}|^2}} dx. 
     \end{align*} 
 \end{remark}

     \begin{proposition}\label{max point}
  Under the assumptions of Proposition \ref{pr 2.1},  
 then there exists $p_j \in \cP_{m_0}$ such that 
 $$u_R(p_j)=\max_{z\in B_R(0)} u_R(z).$$
   
     \end{proposition}
 \noindent{\bf Proof. } If not, we assume that $x_0\in B_R(0)$ such that 
 $$u_R(x_0)=\max_{z\in B_R(0)} u_R(z)>\max\big\{ u_R(p_j):\, j=1,\cdots, m_0\big\}.$$
Moreover, 
 we can choose $x_0$ such that  for some $\bar x_0\in B_{R}(0)\setminus \big(\cP_{m_0}\cup\{x_0\}\big)$,   
 $$u_R(x_0)>u(x)\quad {\rm for}\ x\in B_{r}(\bar x_0),$$
  where 
 $B_{2r}(\bar x_0)\cap \cP_{m_0}=\emptyset$ with $r=|x_0-\bar x_0|$.  
In fact, if $\{x\in  B_R(0): u(x) =u_R(x_0)\}^o$ the interior set is not empty, we can choose  
$x_0=\partial \{x\in   B_R(0): u(x) =u_R(x_0)\}$.

Then $u_R$ is $C^2$ at $B_{r}(\bar x_0)$ and 
  \begin{align}\label{eqq1}
  \nabla u_R(x_0)=0.
   \end{align}
     Since $\cM_0 u=0$ in $B_{r}(\bar x_0)$, 
 then $$|\nabla u_R(x)|\leq \theta<1 \text{ for $x\in \bar B_{r}(\bar x_0)$ } 
$$
and  
$$   -\sum_{i,j} a_{i,j} D_{ij}u(x)=-\cM_0u_R(x)=0\ \text{ for $x\in \bar B_{r}(\bar x_0)$,} $$
where 
$$a_{i,j}=\frac{\delta_{ij}}{(1-|\nabla u(x)|^2)^{\frac12}}+\frac{D_iu(x)D_ju(x)}{(1-|\nabla u(x)|^2)^{\frac32}}.  $$
Note that 
  \begin{align*}
   \sum_{i,j}a_{i,j} \xi_i\xi_j&=\frac{|\xi|^2}{(1-|\nabla u(x)|^2)^{\frac12}}+\frac{(\xi\cdot Du(x))^2}{(1-|\nabla u(x)|^2)^{\frac32}}
   \geq \frac{1}{(1-\theta^2)^{\frac12}}|\xi|^2.
     \end{align*}
Then we are able to apply  Hopf's Lemma  in \cite[Chapter 3]{GT83} to obtain that 
 $D_{\nu}u_R(x_0)>0$, where $\nu:=\frac{x_0-\bar x_0}{|x_0-\bar x_0|}$ is the normal vector pointing outside of $B_{r}(\bar x_0)$. 
That contradicts (\ref{eqq1}).\hfill$\Box$ \medskip

   \begin{corollary}\label{max point-n}
  Under the assumptions of Proposition \ref{pr 2.1},  let $u_{n,R}$ be the solution of (\ref{eq 3.1-n})
 then there exists $p \in {\rm supp}(g_n)$ such that 
 $$u_{n,R}(p)=\max_{z\in B_R(0)} u_R(z).$$
   
     \end{corollary}

  \setcounter{equation}{0}
 \section{Solution with multiple light-cone singularites }

 \subsection{ Positive Dirac masses in   $\R^N$ with $N\geq 3$}
 
 For $N \geq 3$, we first consider the approximation problem associated with the equation featuring multiple Dirac mass sources $\sum_{j=1}^{m_0} \alpha_j \delta_{p_j}$:
\begin{equation}\label{eq 3.1-N}
\left\{
\begin{array}{lll}
\displaystyle \cM_0 u = g_n \quad  \text{in } \, \R^N, \\[3mm]
\displaystyle \lim_{|x|\to+\infty} u(x) = 0,
\end{array}
\right.
\end{equation}
where $n \in \mathbb{N} $ and $g_n$ is defined in (\ref{souce-n}). Let 
$u_{n,j}$ be the unique solution of 
 \begin{equation}\label{eq 3.1-N-j}
 \left\{
\begin{array}{lll}
\displaystyle\  \cM_0 u=  g_{n,j}\quad\ 
   {\rm in}\  \   \R^N, 
\\[4mm]
 \phantom{     }
 \displaystyle \lim_{|x|\to+\infty}u(x)=0     
 \end{array}
 \right.
 \end{equation}
which is radially symmetric with respect to $r=|x-p_j|$.
 and  where
 \begin{equation}\label{souce-n-j}
g_{n,j}(x)=  \alpha_j \eta_n(x-p_j) \quad {\rm for\ any}\ \, x\in\R^N.  
 \end{equation}


 \begin{proposition}\label{pr 2.1-n-w}
 Let $N\geq 3$ and $g_n$ be defined in (\ref{souce-n}) with $n\in\N$, then Eq.(\ref{eq 3.1-N})   has a unique classical solutions $u_{n}\in C^{2,\gamma}(\R^N)$ with $\gamma\in(0,1)$. 
 Moreover,  we have 

 $(i)$ There exists $T_0\geq 1$ such that 
\begin{equation}\label{uppp-1}
0<u_{n,j}\leq u_n\leq \min\{\Phi_{N,\alpha_0}(0), \Phi_{N,T_0\alpha_0}(x)\}\quad {\rm for}\ \, x\in \R^N, \; j \in \{1, 2, \cdots, m_0\}.
\end{equation}

 $(ii)$  There exits $\theta_n \in(0,1)$ such that 
 $$|\nabla u_n|,\,\, |\nabla u_{n,j}|\leq \theta_n\quad{\rm in}\ \ \R^N. $$

 $(iii)$  There holds
   \begin{align}\label{eee-2}
 \int_{\R^N}   \frac{\nabla u_{n}\cdot \nabla  \phi}{\sqrt{1-|\nabla u_{n}|^2}}  \, dx=  \int_{\R^N}  g_n(x) \phi(x)dx \quad {\rm for}\ \, \phi\in C^{0,1}_c(\R^N) .
    \end{align}
 

 \end{proposition}

 \noindent{\bf Proof. } {\bf Part 1:} {\it Existence. } Since $u_{n,R}\in C^{0,1}_0(B_R(0))$, we do the zero extension in $\R^N$.  
It follows by Lemma \ref{lm 3.1} and Lemma \ref{lm 3.2} that the mappings  $R\to u_{n,R}$, $R\to u_{n, R, j}$  are increasing and bounded by $\Phi_{N,\alpha_0}(0)$, together with the fact that $|\nabla u_{n,R}|, |\nabla u_{n, R, j}|<1$ in $\R^N$, then there exist $u_n, u_{n,j}\in C^{0,1}_{\loc}(\R^N)$  such that for $\gamma\in(0,1)$
$$u_{n,R}\to u_n,\quad u_{n, R, j}\to u_{n,j}\ \  {\rm in}\ C^{0,\gamma}_{\loc}(\R^N)\ \, {\rm as}\ \, R\to+\infty.  $$

 {\bf Part 2:} {\it  we show $\displaystyle \lim_{|x|\to+\infty}u_n(x)=0$ and $\displaystyle \lim_{|x|\to+\infty}u_{n,j}(x)=0.$}  Since $u_n>0$ in $\R^N$,
 so we only have to   construct a sup solution to control $u_{n,R}$. Since $u_{n,R}<\Phi_{N,\alpha_0}(0)$ in $\R^N$,
 then there exists $T_0>1$ such that 
 $$\Phi_{N,T_0\alpha_0}(x)\geq  \Phi_{N,\alpha_0}(0)\quad {\rm for}\ \, |x|=R_0. $$
 Note that 
 $$\cM_0u_{n,R}=\cM_0 \Phi_{N,T_0\alpha_0}=0\quad {\rm in}\ \R^N\setminus B_{R_0}(0),$$
 then by comparison principle, we have that for any $R>R_0$
 $$v_{\alpha_j,n,R}(\cdot-p_j)\leq u_{n,R}\leq \Phi_{N,T_0\alpha_0}\quad {\rm in}\ \  \R^N\setminus B_{R_0}(0),$$
 which implies that 
 $$u_{n} \leq  \Phi_{N,T_0\alpha_0}\quad {\rm in}\ \  \R^N. $$
and  from Lemma \ref{lm 3.2}, we have that for $R_0<|x|<+\infty$
 $$u_{n} >u_{n,j} \quad {\rm in}\ \, \R^N$$
 and for $|x|>R_0$
 $$u_{n,j}(x)>\int_{|x|}^{\infty} \frac{ ( c_N^{-1}\alpha  )  }{\sqrt{r^{2(N-1)}+ ( c_N^{-1}\alpha  )^2}}dr.   $$

 Thus, $$u_{n,j}\leq u_{n}(x) \leq \min\big\{\Phi_{  \alpha_0,N}(0), \Phi_{N,T_0\alpha_0}(x)\big\}\quad {\rm for}\ \, x\in \R^N$$
 and
 $\displaystyle \limsup_{|x|\to+\infty}u_n(x)|x|^{N-2}\leq c_N T_0\alpha_0$.

  {\bf Part 3:} {\it  $u_n$ is a classical solution and   $|\nabla u_n|\leq \theta_n$ in $\R^N$ for some $\theta_n\in(0,1)$.  }
  For any $\varrho>R_0$, recall  that 
  $u_{n,R}\to u_n$ in  $C^{0,1} (\bar B_\varrho)(0)) \quad{\rm as}\ \, R\to+\infty$, 
  then it follows by \cite[Lemma 1.3]{BS82} that 
  $u_n$ is weakly spacelike and with respect to its own boundary values, solves the variational problem with mean curvature $g_n$, i.e. 
  $u_{n}$ is the maximizer of the energy functional 
 $$\cI_{n,\varrho}(w)= \int_{B_\varrho(0)}\Big(\sqrt{1-|\nabla w|^2} +w(x)g_n\Big)dx \quad {\rm for\ }\, w\in \bX_n(B_\varrho(0)), $$ 
 and it is also the minimizer of the energy functional 
  $$\cJ_{n,\varrho}(w)= \int_{B_\varrho(0)}\Big(\big(1-\sqrt{1-|\nabla w|^2}\big) -w(x)g_n\Big)dx \quad {\rm for\ }\, w\in \bX_n(B_\varrho(0)), $$ 
 where  
 $$\bX_n(B_\varrho(0)):=\big\{v\in C^{0,1}(B_\varrho(0)):\, v=u_n\ {\rm on}\ \partial B_R(0),\ |\nabla v|\leq 1\  {\rm a.e.\ in\ }\, B_\varrho(0)\big\}.  $$

 Set $\cQ_{n,\rho}=\{x\in \R^N:\ u_n(x)>\rho\}$ for $\rho\in(0, u_{n,j}(p_j))$, then $\cQ_{n,\rho}$ is bounded and
$$\bigcup_{\rho>0} \cQ_{n,\rho}=\R^N. $$
 it follows  from \cite[Theorem 3.6]{BS82}, for any $\rho>0$, there exists $\theta_\rho\in(0,1)$, such that 
 $|\nabla u_n|\leq \theta_1$ in $\cQ_{n,\rho}$. 
 Then $u_n$ is a classical solution of Eq.(\ref{eq 3.1-N}) by Lemma \ref{lm 2.1-reg-int}. 
 By the decay of $u_n$ at infinity, it follows  from Proposition \ref{pr grad-ext-2.1} that there exists $\theta_1\in(0,1)$ such that 
 $|\nabla u_n|\leq \theta_1$  in $\R^N\setminus B_{R_0}(0)$.
 As consequence, for some $\theta_n\in(0,1)$ $|\nabla u_n|\leq \theta_n$  in $\R^N$. 
 \smallskip
 
 For any $\varphi\in C_c^1(\R^N)$, there exists $n_0\geq R_0$ such that ${\rm supp}(\varphi)\subset B_{n_0}(0)$
 and for any $R\geq n_0$, there holds by (\ref{eee-1})
  \begin{align*} 
 \int_{\R^N}   \frac{\nabla u_{n,R}\cdot \nabla  \phi}{\sqrt{1-|\nabla u_{n,R}|^2}}  \, dx&= \int_{B_R(0)}   \frac{\nabla u_{n,R}\cdot \nabla  \phi}{\sqrt{1-|\nabla u_{n,R}|^2}}  \, dx
 \\[1mm]&=   \int_{B_R(0)}   g_n(x) \phi(x)dx=
   \int_{\R^N}  g_n(x) \phi(x)dx,
    \end{align*}
then  passing to the limit as $R\to+\infty$, we obtain (\ref{eee-2}). 
   \hfill$\Box$\medskip
 

\vskip2mm
   
 Now we are ready to prove Theorem \ref{teo 1-fund3}. 
 \vskip3mm

\noindent{\bf Proof of Theorem \ref{teo 1-fund3}. } 
{\bf Existence: }
For $N\geq 3$, from Proposition \ref{pr 2.1-n-w}, let $u_n$ be the solutions of (\ref{eq 3.1-N}),
$|\nabla u_n|<1$ in $\R^N$ and 
 \begin{equation} \label{uppp-1--}
0<u_{n,j}\leq u_n\leq \min\{\Phi_{N,\alpha_0}(0), \Phi_{N,T_0\alpha_0}(x)\}\quad {\rm for}\ \, x\in \R^N, \; \forall \; j \in \{1, \cdots, m_0\}.
\end{equation}
then  there is $u_\infty\in   C^{0,1}_{\loc}(\R^N)$ such that for $\gamma\in(0,1)$
\begin{align}\label{conv-13}
 u_{n}\to u_{\infty}\quad{\rm in}\ C^{0,\gamma}_{\loc}(\R^N) \quad {\rm as}\ \, n\to+\infty.   
   \end{align}
   As the proof of {\bf Claim 1} in Theorem   \ref {pr 2.1}, we have that 
$|\nabla u_\infty |\leq 1\;\; {\rm a.e.\ in}\ \, \R^N.$

  On the other hand, by Theorem \ref{pr 2.1}, we have 
  \begin{align}\label{conv-13}
 u_{n, R}\to u_{R}\quad{\rm in}\ C^{0,\gamma}_{\loc}(\R^N) \quad {\rm as}\ \, n\to+\infty.   
   \end{align}
  Since $u_n\geq u_{n,R}$ in $B_R(0)$ for any $R>\theta_0 R_0$,  then 
  $u_\infty\geq u_R$ in $ B_R(0).  $


 By the bound  (\ref{uppp-1--}), we have that  for $j=1,\cdots,m_0$ and any $R>\theta_0 R_0$
 \begin{align}\label{uni-2-N}
u_R(x)  \leq  u_\infty(x)\leq \Phi_{N,\bar\alpha}(x)\quad {\rm for\ all}\  x\in \R^N, 
\end{align}
where $\bar \alpha=T_0\alpha_0$. 
 Therefore, $u_\infty$ is positive and decays to $0$ at infinity. 
  
   For $\rho>0$, denote
  $$\cQ_\rho=\{x\in\R^N:\, u_\infty(x)>\rho\}, $$
then there is some $\rho_0\in \big(0,\frac14\min\{\Phi_{N,\alpha_j}(0),  |p_i-p_j|:\, j=1,\cdots,m_0,\ \,  i\not=j\}\big]$ such that  for $\rho\in(0,\rho_0)$
$$B_{\theta_0R_0}(0)\subset  \cQ_\rho.$$
Observe that we have $\bigcup_{\rho\in(0,\rho_0)} \cQ_\rho=\R^N. $
Let $\displaystyle  \cO_{\rho}=\cQ_\rho\setminus \cup_{j=1}^{m_0} B_{\rho}(p_j)$ and
\begin{equation}\label{eq 3.1-ind-1}
\cI_{\rho}(w)= \int_{\cO_\rho}\sqrt{1-|\nabla w|^2}  \, dx   \quad {for\ }\, w\in \bX_0(\cO_\rho)
 \end{equation} 
with 
$$\bX_0(\cO_\rho):=\Big\{v\in C^{0,1}(\bar \cO_\rho):\,  v= u_\infty\ {\rm on}\ \partial \cO_\rho,\ \, |\nabla v|\leq 1\ \ {\rm a.e.\ in\ }\cO_\rho\Big\}.  $$
Then $u_\infty$ is weakly spacelike and achieves the maximizer of $\cI_{\rho}$. \medskip

 Next    {\it for any $\sigma\in(0,\sigma_0]$, there exists $\theta=\theta(\sigma)\in (0,1)$  such that  
$$|\nabla u_\infty |\leq \theta_\sigma\quad{\rm in}\ \ \cO_{\rho,m} ,$$
where $\cO_{\rho,m}=\cQ_{\rho,m}\setminus \bigcup^{m_0}_{j=1} B_{\sigma}(p_j) $  with  $\cQ_{\rho,m}$ being the component   containing $B_{\theta_0R_0}(0)$.   }
 \smallskip

 Let 
 $$\cK_s=\big\{\overline{xy}\subset  \cO_{\rho,m}:\, x,y\in\partial\cO_{\rho,m}, \; x\not=y,\ \,  |u_\infty(x)-u_\infty(y)|=|x-y|    \big\}.  $$
 Our aim is to show $\cK_s=\emptyset$.

\smallskip

If not,  we choose $x_1,x_2\in \partial\cO_{\rho,m}$ such that  $|u_\infty(x_1)-u_\infty(x_2)|=|x_1-x_2|$.
$$\bL_{x_1x_2}= \big\{x_t: \text{for $t$ belongs a maximal interval of $\R$ such that } x_t    \in   B_R\setminus \cP_{m_0}\,     \big\} , $$
where $x_t=x_1+t(x_2-x_1)$. 
Let $\bar x_,\bar x_2$ be the ends points of $\bL_{x_1x_2}$,   then  either $\bL_{x_1x_2}$ could be extended to cross the boundary $\partial  \cQ_{\rho,m}$ twice, i.e. $\bar x_1,\bar x_2\in \partial \cQ_{\rho,m}$  or $\overline{\bL}_{x_1x_2}$ cross the boundary $\partial \cQ_{\rho,m}$ once i.e. 
  $\bar x_1\in \partial \cQ_{\rho,m}, \bar x_2\in \cP_{m_0}$  or  $\bL_{x_1x_2}$ stops by two point in $\cP_{m_0}$ i.e. $\bar x_,\bar x_2\in \cP_{m_0}$. 

We apply \cite[Theorem 3.2]{BS82} to obtain that 
$$
u_\infty(x_t)=u_\infty(x_1)+t|x_1-x_2|\quad {\rm for\ all}\  x_t\in \overline{\bL}_{x_1x_2}.
$$
 Particularly, we have that 
 \begin{equation}\label{ind-1-t-N}
|u_\infty(\bar x_1)-u_\infty(\bar x_2)|= |\bar x_1-\bar x_2|.
 \end{equation}

 If $\bar x_1,\bar x_2\in \partial \cQ_{\rho,m}$, then 
 $$|u_\infty(\bar x_1)-u_\infty (\bar x_2)|=0<|\bar x_1-\bar x_2|,  $$
 which contradicts (\ref{ind-1-t-N}). 
 
  If $\bar x_1\in \partial \cQ_{\rho,m}, \bar x_2\in \cP_{m_0}$ and we can set 
  $$\bar x_1 \in \bL_{x_1x_2}\cap \partial \cQ_{\rho,m} \quad{\rm and}\quad \bar x_2 \in \cP_{m_0},  $$
 then  $u_\infty(\bar x_1) =\rho<u_\infty(p_j)$ for $p_j\in\cP_{m_0}$ and for $\rho>0$ small 
 $$|\bar x_1-\bar x_2| \geq \max\big\{(\theta_0-\frac12) R_0, \Phi_{N,\alpha_0}(0) \big\}$$ 
and 
 $$|u_\infty(\bar x_1)-u_\infty (\bar x_2)|< u_\infty(\bar x_2)   \leq \Phi_{N,\alpha_0}(0)\leq |\bar x_1-\bar x_2|, $$
 then we get contradictions with (\ref{ind-1-t-N}). 
 
 If   $\bar x_1,\bar x_2\in \cP_{m_0}$, we can assume that 
 $$
u_\infty(\bar x_1)=u_\infty(\bar x_2)+|\bar x_1-\bar x_2|\quad {\rm for\ all}\  x_t\in \overline{\bL}_{x_1x_2}.
$$
Recall that 
$$w_\alpha(x)=\Phi_{N, \alpha}(x-\bar x_1)-\Phi_{N, \alpha}(0)+u_\infty(\bar x_1),\quad x\in \cQ_{\rho,m}, $$
then $w_\alpha(\bar x_1)=u_\infty(\bar x_1)$ and there exist  $\bar \alpha\geq \alpha_j$ such that 
$$w_{\bar \alpha}\leq -1\quad {\rm on}\ \partial \cQ_{\rho,m}. $$
By the same proof in {\bf Claim 2}, we have that 
$u_\infty \geq  w_{\bar \alpha}\; {\rm in}\ \cQ_{\rho,m}, $
which implies that 
$$ w_{\bar \alpha}(\bar x_1)-w_{\bar \alpha}(\bar x_2)\geq u_\infty(\bar x_1)-u_\infty(\bar x_2)= |\bar x_1-\bar x_2|,  $$
 which  contradicts the fact  that    $|\nabla \Phi_{N,\alpha}|<1$ for $\R^N\setminus\{0\}$. \smallskip

Let 
$$w_{\bar\alpha,n}(x)=\Phi_{N, \bar\alpha}(x-\bar x_1)-\Phi_{N, \bar\alpha}(0)+u_{n}(\bar x_1),\quad x\in B_R(0), $$
then $w_{\bar\alpha,n}(\bar x_1)=u_{n}(\bar x_1)$ and    
$$\lim_{n\to+\infty}w_{\bar\alpha,n}(x)=w_{\bar\alpha}(x)\quad {\rm for}\ \, x\in \R^N $$
and there exist  $n_0>1$ and $\bar R>R_0$  such that  
$w_{\bar \alpha,n}<0\; {\rm on}\ \R^N\setminus  B_{\bar R} (0). $

By the comparison principle, we have that $u_{n}\geq  w_{\bar \alpha,n}\ \,  {\rm in}\ B_{\bar R}(0),  $
which implies that 
$$u_{\infty}\geq  w_{\bar \alpha}\quad {\rm in}\ B_{\bar R}(0)  $$
and
$$ w_{\bar \alpha}(\bar x_1)-w_{\bar \alpha}(\bar x_2)\geq u_\infty(\bar x_1)-u_\infty(\bar x_2)= |\bar x_1-\bar x_2|,  $$
 which contradicts the fact  that    $|\nabla \Phi_{N,\alpha}|<1$ for $\R^N\setminus\{0\}$.

As a consequence, we obtain $\cK_s=\emptyset$ and it follows by \cite[Theorem 4.1, Corollary 4.2]{BS82} 
that $u_\infty\in C^1(\cQ_{\rho,m})$ is strictly spacelike in $\cQ_{\rho,m}$ and  
there exists $\theta_\epsilon\in[0,1)$ such that 
 \begin{equation}\label{grandient-1-R-N}
 |\nabla u_\infty|\leq \theta_\epsilon\quad {\rm in}\ \, \overline \cQ_{\rho,m}.   
 \end{equation}

Now we show the qualitative properties of  $u_\infty$.

\vskip2mm
{\bf Part 1: }   we show that $u_\infty$ is a classical solution of 
 \begin{equation}\label{eq 3.1-cla-R-N}
 \left\{
\begin{array}{lll}
\displaystyle\ \,  \cM_0 u= 0\ \ 
  {\rm in}\  \, \R^N \setminus\cP_{m_0}, 
\\[2mm]
 \phantom{   }
 \displaystyle\lim_{|x|\to+\infty} u(x)=0  
 \end{array}
 \right.
 \end{equation}
and $u_\infty$ is a weak solution of the following problem:
 \begin{equation}\label{eq 3.1-ral-N}
 \left\{
\begin{array}{lll}
\displaystyle  \cM_0 u= \sum_{j=1}^{m_0} \alpha_j\delta_{p_j}\ \  & {\rm in}\  \,   \cD'(\R^N), 
\\[3mm]
 \phantom{    }
 \displaystyle\lim_{|x|\to+\infty} u(x)=0.     
 \end{array}
 \right.
 \end{equation}

Fix   $\rho>0$, denote
  $$\cQ_{n,\rho}=\{x\in\R^N:\, u_n(x)>\rho\},$$
then there is some  $\rho_1\in \big(0,\frac14\min\{\Phi_{N,\alpha_j}(0),  |p_i-p_j|:\, j=1,\cdots,m_0,\ \,  i\not=j\}\big]$ such that  for $\rho\in(0,\rho_1)$, there holds
$$B_{\bar R}(0)\subset  \cQ_{n,\rho}.$$
 
 Let $w_{n,\rho}=u_{n}-\rho$, then it is a solution of 
\begin{equation}\label{eq 3.1-N-rho-1}
 \left\{
\begin{array}{lll}
\displaystyle  \cM_0 w_{n,\rho}= g_n\quad
 & {\rm in}\  \   \cQ_{n,\rho}, 
\\[4mm]
 \phantom{ -\ \,   }
 \displaystyle w_{n,\rho}=0 \quad
 & {\rm on}\  \, \R^N\setminus  \cQ_{n,\rho}.    
 \end{array}
 \right.
 \end{equation}
 Taking the test function $\phi=w_{n,\rho}$ in (\ref{eee-1}),  we derive that
 \begin{equation}  
 \int_{\cQ_{R,\rho}} \frac{|\nabla w_{n,\rho}|^2   }{\sqrt{1-|\nabla w_{n,\rho}|^2}} dx= \int_{\R^N}g_n w_{n,\rho}dx  \leq \Phi_{N,\alpha_0}(0)\alpha_0. \label{bound-1-N}
    \end{equation}
    

Firstly, we show the uniformly bound that 
 \begin{align*} 
  \int_{B_{\bar R}(0)} \frac{|\nabla u_{n} |   }{\sqrt{1-|\nabla u_{n}|^2}} dx &\leq   \int_{\cQ_{n,\rho}}  \frac{|\nabla w_{n,\rho} |   }{\sqrt{1-|\nabla w_{n,\rho}|^2}} dx 
   \\[1mm]& =   2   \int_{\cQ_{n,\rho}\cap \{|\nabla w_{n,\rho}|\geq \frac12\}} \frac{|\nabla w_{n,\rho}|^2   }{\sqrt{1-|\nabla w_{n,\rho}|^2}} dx
 \\&\qquad +   \int_{\cQ_{n,\rho} \cap \{|\nabla  w_{n,\rho}|<\frac12\}} \frac{|\nabla w_{n,\rho}|   }{\sqrt{1-|\nabla w_{n,\rho}|^2}} dx
  \\[1mm]& \leq   \Big(2 \Phi_{N,\alpha_0}(0)\alpha_0 +\frac{\sqrt{3}}3 |\cQ_{n,\rho} |\Big). 
  \end{align*}

For any nonnegative $\varphi\in C^{0,1}_c(B_{\bar R}(0))$ such that $\varphi(x)=\varphi(p_j)$ for $x\in B_\epsilon(p_j)$ for any $j=1,\cdots, m_0$
and $\epsilon>0$ small, then ${\rm supp}(\nabla \varphi)\subset B_{\bar R}(0)\setminus \bigcup^{m_0}_{j=1} B_\epsilon(p_j)$ 
 \begin{align*} 
 \lim_{n\to+\infty} 
 \int_{B_{\bar R}(0) \setminus\big( \bigcup^{m_0}_{j=1} B_\epsilon(p_j)\big)} \frac{ \nabla u_{n} \cdot \nabla \varphi   }{\sqrt{1-|\nabla u_{n}|^2}} dx& =  \lim_{n\to+\infty}  \int_{\R^N} \frac{ \nabla u_{n} \cdot \nabla \varphi   }{\sqrt{1-|\nabla u_{n}|^2}} dx 
  \\[1mm]& = \sum^{m_0}_{j=1} \alpha_j \varphi(p_j)
 \\[1mm]&=  \int_{\R^N} \frac{ \nabla u_{\infty} \cdot \nabla \varphi   }{\sqrt{1-|\nabla u_{\infty}|^2}} dx
 \\[1mm]&= \int_{B_{\bar R}(0)\setminus\big( \bigcup^{m_0}_{j=1} B_\epsilon(p_j)\big)} \frac{ \nabla u_{\infty} \cdot \nabla \varphi   }{\sqrt{1-|\nabla u_{\infty}|^2}} dx, 
  \end{align*}
 that is
 \begin{align} \label{con-1-R-N}
  \frac{ \nabla u_{n}    }{\sqrt{1-|\nabla u_{n}|^2}} \to\frac{ \nabla u_{\infty}    }{\sqrt{1-|\nabla u_{\infty}|^2}} \quad{\rm weakly\ in\ }L^1 \big(B_{\bar R}(0)\setminus\big( \bigcup^{m_0}_{j=1} B_\epsilon(p_j)\big) , \; \R^N \big), 
  \end{align}
then by the upper semicontinuity of the area integral
\begin{align*} 
\int_{B_{\bar R}(0)\setminus\big( \bigcup^{m_0}_{j=1} B_\epsilon(p_j)\big)} \frac{|\nabla u_{\infty} |   }{\sqrt{1-|\nabla u_{\infty}|^2}} dx&\leq   \liminf_{n\to+\infty}  \int_{B_{\bar R}(0)\setminus\big( \bigcup^{m_0}_{j=1} B_\epsilon(p_j)\big)} \frac{|\nabla u_{n} |   }{\sqrt{1-|\nabla u_{n}|^2}} dx 
\\[1mm]&  \leq   \Big(2 \Phi_{N,\alpha_0}(0)\alpha_0 +\frac{\sqrt{3}}3 |\cQ_{R,\rho} |\Big),
  \end{align*}
which,  by the arbitrary of $\epsilon>0$,  implies that 
\begin{align*} 
\int_{B_{\bar R}(0)} \frac{|\nabla u_{\infty} |   }{\sqrt{1-|\nabla u_{\infty}|^2}} dx  \leq   \Big(2 \Phi_{N,\alpha_0}(0)\alpha_0 +\frac{\sqrt{3}}3 |\cQ_{R,\rho} |\Big).
  \end{align*}
As a consequence, by the arbitrary of $\bar R>\theta_0R_0$, we have that  
$$\frac{|\nabla u_{\infty} |   }{\sqrt{1-|\nabla u_{\infty}|^2}}\in L^1_{\loc} (\R^N)$$ 
and we obtain that $u_\infty \in \tilde\bX_\infty(\R^N)$,
  where
   $$\tilde\bX_\infty(\R^N) =\Big\{w\in C^{0,1}_c(\R^N):\, |\nabla w|<1\ {\rm\ in}\ \R^N \setminus \cP_{m_0},\ \, \frac{|\nabla w|}{\sqrt{1-|\nabla w|^2}}  \in L^1_{\loc}(\R^N)\Big\}. $$
   Moreover, from (\ref{con-1-R-N}), we get that for any $\epsilon>0$ small, 
 $$\int_{\cO_\epsilon}  \frac{ \nabla u_{\infty} \cdot \nabla \varphi   }{\sqrt{1-|\nabla u_{\infty}|^2}}dx=0\quad \text{  for any $\varphi\in C^{0,1}_c(\R^N)$ with ${\rm supp}(\varphi)\subset \cO_\epsilon:=\R^N\setminus \bigcup_{j=1}^{m_0}B_\epsilon(p_j)$. } $$
 By (\ref{grandient-1-R-N}) and  \cite[Theorem 3.6]{BS82},   $u_\infty\in C^{2,\gamma}_{\loc}(\cO_{2\epsilon})$, by the arbitrary of $\epsilon$, we get that  
$u_\infty$ verifies the equation (\ref{eq 3.1-cla-R-N}) in the classical sense.

\vskip2mm
Now we take $\xi\in C^{0,1}_c(\R^N)$  with  ${\rm supp}( \xi)\subset   \R^N \setminus\cP_{m_0}$
and 
   $$\int_{\R^N}  \frac{\nabla u_\infty\cdot\nabla \xi}{\sqrt{1-|\nabla u_\infty|^2}} \,dx =0\quad \text{  for any $\xi\in C^{0,1}_c (\R^N\setminus\cP_{m_0})$.} $$ 
   Now we apply Proposition \ref{cr 2.1-iso} to obtain that $u_\infty$ is a weak solution 
\begin{equation}\label{eq 1.1-R-w-N}
  \cM_0 u= \sum_{j=1}^{m_0} k_{p_j}  \delta_{p_j}   \quad
  {\rm in}\  \, \cD'\big(\R^N\big)
  \end{equation}
for some $k_{p_j}\in \R$.  That is, 
\begin{equation}
\int_{ \R^N }  \frac{\nabla u_\infty\cdot\nabla \xi}{\sqrt{1-|\nabla u_\infty|^2}} \,dx =\sum_{j=1}^{m_0}  k_{p_j}  \xi(p_j),\quad\, \forall \xi\in C^{0,1}_c (\R^N).
 \end{equation}
 
Finally we need to prove $k_{p_j} =\alpha_j$ for any $j=1,\cdots,m_0$. This is very similar to the part in the proof of 
Theorem \ref{pr 2.1} for $u_R$. We omit it. 
Therefore we conclud that $u_\infty$ is a weak solution (\ref{eq 1.1-R-w-N}). \smallskip


{\bf Asymptotic behavior at poles: }  At the Dirac poles with positive multiplicities, it follows by \cite[Theorem 1.5]{E86} (also see \cite[Theorem 1.4]{K12} and \cite[Theorem 1.6]{BAP}) that 
$  u_\infty$ is light-cone singular at $\cP_{m_0}$ with  the behavior
$$\lim_{|x-p_j|\to0^+}| \nabla  u_\infty(x)|=1.  $$
Moreover, the  vertex of the cone is upwards, i.e.  $p_j$ isn't a local minimal point of $  u_\infty$.  \smallskip

{\bf Asymptotic behavior at infinity:  } {\it Lower bound:}  Note that for any $R>R_0$, 
$$u_\infty(x)\geq u_{R}\geq u_{R,j}  $$
and 
$$u_{R,j}\to \Phi_{N,\alpha_j}(\cdot-p_j)\quad{\rm in}\ \R^N, $$
thus, for any $j=1,\cdots,m_0$, 
\begin{align}\label{bh 1-3} 
u_\infty(x)\geq \Phi_{N,\alpha_j}(\cdot-p_j)\quad{\rm in}\ \R^N. 
 \end{align} 
 
  From (\ref{bb1-2}), we have that 
the solution $ u_{\infty}$  has the behavior 
\begin{align}\label{bh 1-13} 
 u_{\infty}(x)= c - \frac1{|\partial B_1(0)|} (1-|\vec{a}|) {\rm Res}[ u_{\infty}]  |x|^{2-N} +O(|x|^{1-N})\quad {\rm as}\ \, |x|\to+\infty, 
 \end{align} 
where   by (\ref{bh 1-3}),  $\vec{a}=0$ and  $c\geq 0$. 

Recall that  
\begin{align*}
0&=\int_{B_R(0)\setminus \big(\cup^{m_0}_{j=1} B_\epsilon(p_j)\big)} \cM_0  u_{\infty}(x) dx
\\&=\int_{\partial B_R(0)} \frac{\nabla  u_\infty (x)}{\sqrt{1-|\nabla  u_\infty(x)|^2}}\cdot \frac{x}{R}dH_{N-1}(x)+
\sum^{m_0}_{j=1}\int_{\partial B_\epsilon(p_j)} \frac{\nabla u_\infty (x)}{\sqrt{1-|\nabla  u_\infty(x)|^2}}\cdot \frac{x-p_j}{|x-p_j|}dH_{N-1}(x)
\\&\to {\rm Res}(  u_{\infty})+\alpha_0\quad {\rm as}\ \epsilon\to0^+,
 \end{align*}
 that is 
 \begin{align*}
 {\rm Res}[  u_{\infty}]= - \alpha_0 =  - \sum^{m_0}_{j=1}\alpha_j.
 \end{align*} 
Thus, it follows by (\ref{bh 1-13}) that 
\begin{align}\label{bh 1-23} 
  u_{\infty}(x)= c+ \frac{\alpha_0}{|\partial B_1(0)|}|x|^{2-N}  +O\big(|x|^{1-N} \big)\quad {\rm as}\ \, |x|\to+\infty, 
 \end{align} 
where  $c\geq0$.  Note that  $u_{\infty}-c$ decays as $|x|^{2-N}$ at infinity, and  it   is  a  solution of (\ref{eq 1.1-fund3}). 
 By the maximum principle, we can obtain that $u_{\infty}-c$ is positive and 
 $u_{\infty}-c\geq  u_R$ in $B_R(0)$ for any $R>0$, which together 
 with $u_R\to u_{\infty}$ in $C^{0,1}_{\loc}(\R^N)\cap C^2_{\loc}(\R^N\setminus\cP_{m_0})$ as $R\to+\infty$, 
  implies that $c=0$ and (\ref{bh 1-23}) reduces to 
  \begin{align}\label{bh 1-23++} 
  u_{\infty}(x)=  \frac{\alpha_0}{|\partial B_1(0)|}|x|^{2-N}  +O\big(|x|^{1-N} \big)\quad {\rm as}\ \, |x|\to+\infty. 
 \end{align} \smallskip

 {\bf Uniqueness: } Let $\bar u$ be another solution satisfying
 $\bar u(x)\to 0$ as $|x|\to +\infty$. Then we can show $\bar u+\epsilon, \bar u-\epsilon$ will be a super and sub solutions    respectively, of
 \begin{equation} 
 \left\{
\begin{array}{lll}
\displaystyle  \cM_0 u= \sum^{m_0}_{j=1} \alpha_j \delta_{p_j}\quad
 & {\rm in}\  \   B_R(0), 
\\[4mm]
 \phantom{ -\ \,   }
 \displaystyle u=\bar u\pm \epsilon \quad
 & {\rm on}\  \, \partial B_R(0).    
 \end{array}
 \right.
 \end{equation}
 Note that $\bar u\pm\epsilon$ maximizes 
$$\cI_{B_R(0)}(u) :=\int_{B_R(0)}\sqrt{1-|\nabla u|^2}dx+\sum^{m_0}_{j=1}\alpha_ju(p_j)$$
with respect to the boundary value $ (\bar u\pm\epsilon)\big|_{\partial B_R(0)}$. 
 Since $\bar u+\epsilon>u_\infty$, $\bar u-\epsilon<u_\infty$ on $\R^N\setminus B_R(0)$ for $R>R_0$ large enough, 
  then by Lemma \ref{lm cp}, we derive that 
 $$\bar u-\epsilon\leq u_\infty=\bar u+\epsilon\quad {\rm in}\  \R^N. $$
 By the arbitrary of $\epsilon>0$, we derive that $u_\infty=\bar u$ and the uniqueness follows.  \smallskip

 {\bf Maximizer of $\cJ_\infty$: } Since $\sum^{m_0}_{j=1}\alpha_j\delta_{p_j}\in (C^{0,1}(\R^N))^*$,  then it follows by \cite[Theorem 1.3]{BAP} that 
 the energy functional  
$$
\cJ_{\infty}(w):= \int_{\R^N}\big(1- \sqrt{1-|\nabla w|^2}\big)\, dx -  \sum^{m_0}_{j=1}\alpha_jw(p_j)  \quad {for\ }\, w\in \bX_\infty (\R^N),   
$$
has a unique minimizer, where recall that 
$$\bX_\infty (\R^N)=\{v\in C^{0,1}(\R^N): |\nabla v|\leq 1\ \, {\rm a.e.\ in}\ \R^N, \ \,  \int_{\R^N}\big(1- \sqrt{1-|\nabla w|^2}\big)dx<+\infty \big\}. $$
 
Note that $u_\infty$ is approximating by $u_{n}$ in $C^{0,\gamma}$, and 
by \cite[Lemma 1.3]{BS82},  $u_n$ is the critical point of 
$$\cJ_{N,n}(w):= \int_{\R^N}\big(1- \sqrt{1-|\nabla w|^2}\big)dx -\int_{\R^N} g_nwdx  \quad {for\ }\, w\in \bX_0 (\R^N). $$  
Since $u_{n}\to u_\infty$ in $C^{0,\gamma}(\R^N)$  as $n\to +\infty$,
 then  $u_\infty$ is the unique maximizer of  $\cJ_{\infty}$. Therefore, Theorem \ref{teo 1-fund3} is proved by setting $u_\infty (x)= 
u_{N, \alpha_0}(x), \; \forall x \in \R^N$. 
  \hfill$\Box$\medskip

\begin{corollary}\label{cr 3.1-mont}
Let $N\geq 3$ and
$$\cP_{m_1}\subset \cP_{m_2} \quad \text{ with $m_2\geq m_1$}$$
 and 
$$0<\alpha_{1,j}\leq\alpha_{2,j}\quad {\rm for}\ \, j=1,\cdots m_1,\quad  \alpha_{2,j}>0\quad {\rm for}\ j>m_1\ \ {\rm if}\ \, m_2>m_1. $$
Let $u_{i}$ with $i=1,2$ be the solutions, respectively, of
\begin{equation}\label{eq 3.1-com-1}
 \left\{
\begin{array}{lll}
\displaystyle\  \cM_0 u= \sum^{m_i}_{j=1} \alpha_{i,j} \delta_{p_j}\quad
 & {\rm in}\  \   \R^N, 
\\[4mm]
 \phantom{ ,   }
 \displaystyle \lim_{|x|\to+\infty}u(x)=0.    
 \end{array}
 \right.
 \end{equation}
 Then $u_2\geq u_1$ in $\R^N$. 

\end{corollary} 
 \noindent{\bf Proof. } It follows by the construction and uniqueness of solution to (\ref{eq 3.1-com-1}). \hfill$\Box$\medskip

  \subsection{Positive Dirac masses in   $\R^2$}

When $N=2$, we first consider  the approximation problem 
\begin{equation}\label{eq 3.1-N-2}
\left\{
\begin{array}{lll}
\displaystyle \cM_0 u = g_n \quad  \text{in } \, \R^2, \\[3mm]
\displaystyle \max_{x\in\R^2} u(x) = 0,
\end{array}
\right.
\end{equation}
where $n \in \mathbb{N}$ and $g_n$ is defined in (\ref{souce-n}).


 \begin{proposition}\label{pr 2.1-n-w-2}
 Let    $g_n$ be defined in (\ref{souce-n})  with $n\in\N$, then Eq.(\ref{eq 3.1-N-2})  has a unique classical solution  $\tilde u_{n}\in C^{2,\gamma}(\R^2 )$ with $\gamma\in(0,1)$. 
 Moreover,  we have

 $(i)$ There hols
\begin{equation}\label{uppp-1-2}
\Phi_{2, \alpha_0}(x)-R_0  \leq \tilde u_n(x)\leq  \Phi_{2, \alpha_0}(x)+R_0 \quad {\rm for}\ \, x\in \R^2. 
\end{equation}
 
 $(ii)$  There exits $\theta_n \in(0,1)$ such that 
 $$|\nabla \tilde u_n| \leq \theta_n\quad{\rm in}\ \ \R^2.  $$

 $(iii)$  We have
   \begin{align}\label{eee-2-2}
 \int_{\R^2}   \frac{\nabla \tilde u_{n}\cdot \nabla  \phi}{\sqrt{1-|\nabla \tilde u_{n}|^2}}  \, dx=  \int_{\R^2}  g_n(x) \phi(x)dx \quad {\rm for}\ \, \phi\in C^{0,1}_c(\R^2).  
    \end{align}
 \end{proposition}
 
  \noindent{\bf Proof. } Recall that $u_{n,R}$ is the unique solution of (\ref{eq 3.1-n}). 
 Let  
$$\tilde u_{n,R}=u_{n,R}-\max_{x\in B_R(0)}u_{n,R}(x)\quad {\rm in}\ \, \R^2  $$
and we claim that   
\begin{align}\label{uni-1-2}
\Phi_{2,\alpha_0}(x)-R_0 \leq  \tilde u_{n,R}(x) \leq \Phi_{2,\alpha_j}(x)+R_0\quad {\rm for\ all}\  x\in B_R(0)
\end{align}
for any $j=1,\cdots,m_0$. 
 
 \smallskip
 
 In fact, it follows by the comparison principle that 
    $$ v_{\alpha_j,n, R}(x-p_j)\leq u_{n,R}(x)\leq  v_{\alpha_0,n, R}(x)\quad{\rm for}\ x\in B_R(0)$$ 
    for $j=1,\cdots,m_0$,
    where  by  Lemma \ref{lm 3.2} for $x\in B_R(0)$,
   $$  \min \big\{\Phi_{2,\alpha}(x), \; \Phi_{2,\alpha}(\frac2n)\big\}  \leq v_{\alpha,n,R}(x) +\Phi_{2,\alpha}(R)\leq \Phi_{2,\alpha}(x). $$  
For any $R > R_0$, by Corollary \ref{max point-n}, there exits $p\in {\rm supp}(g_n)\subset B_{R_0}(0)$    such that 
$$ u_{n,R}(p ) =\max_{x\in B_R(0)} u_{n,R}(x). $$
Let 
$$\tilde u_{n,R}(x)=u_{n,R}(x)-u_{n,R}(p) \quad {\rm for}\ \, x\in B_R(0). $$

 For the upper bound, since $|\nabla u_{n,R}|< 1$ in $B_R(0) $ and ${\rm supp}(g_n)\subset B_{\frac12R_0}(0)$,
 then  for $j=1,\cdots,m_0$
 $$u_{n,R}(x)-u_{n,R}(p)\leq \Phi_{2,\alpha_j}(x-p)+R_0\quad {\rm for}\ \, x\in \bar B_{R_0}(0)$$
 and for $x\in \partial B_{R}(0)$, 
 $$\tilde u_{n,R}(x)  \leq -u_{n,R}(p ) \leq  -\Phi_{2,\alpha_j}(R_j) +|p|\leq -\Phi_{2,\alpha_j}(R) +R_0. $$
 Since 
 $$\cM_0\tilde u_{n,R}=0=\cM_0\big(\Phi_{2,\alpha_j}(x-p_j)+R_0\big) \quad{\rm in}\ \, B_R(0)\setminus\bar B_{R_0}(0), $$
 then the comparison principle implies that 
\begin{align}\label{uni-3-2}
\tilde u_{n,R}\leq \Phi_{2,\alpha_j}(x-p_j)+R_0\quad {\rm in}\ \,  B_R(0). 
\end{align}
 
 For the lower bound, since $|\nabla u_{n,R}|< 1$ in $B_R(0) $,
 then 
 $$u_{n,R}(x)-u_{n,R}(0)\geq -R_0\geq \Phi_{2,\alpha_0}(x)-R_0\quad {\rm for}\ \, x\in \bar B_{R_0}(0)$$
 and for $x\in \partial B_{R}(0)$, 
 $$\tilde u_{n,R}(x)=-u_{n,R}(p_{\bar j}) \geq   \Phi_{2,\alpha_0}(R)  \geq  \Phi_{2,\alpha_0}(R) -R_0.$$
 Since 
 $$\cM_0\tilde u_{n,R}=0=\cM_0\big(\Phi_{2,\alpha_0}(x)-R_0\big) \quad{\rm in}\ \, B_R(0)\setminus\bar B_{R_0}(0), $$
again by the comparison principle, we obtain 
\begin{align}\label{uni-4-2}
\tilde u_{n,R}\geq \Phi_{2,\alpha_0}(x)-R_0\quad {\rm in}\ \,  B_R(0). 
\end{align}
 The  bound in (\ref{uni-1-2}) follows by (\ref{uni-3-2}) and (\ref{uni-4-2}) directly.  \smallskip

 {\bf Part 1:} {\it Existence. }  By (\ref{uni-1-2}) and the fact that $|\nabla  \tilde u_{n,R}| <1$ in $\R^2$, then there exist $ \tilde u_n \in C^{0,1}_{\loc}(\R^2)$  such that for $\gamma\in(0,1)$
$$ \tilde u_{n,R}\to \tilde  u_n \,  \  {\rm in}\ C^{0,\gamma}_{\loc}(\R^2)\ \, {\rm as}\ \, R\to+\infty.  $$
 and 
 \begin{align}\label{uni-1-2-}
\Phi_{2,\alpha_0}(x)-R_0 \leq  \tilde u_{n}(x) \leq \Phi_{2,\alpha_j}(x)+R_0\quad {\rm for\ all}\  x\in B_R(0), 
 \end{align}
 which means $\tilde u_n(x)\to-\infty$ as $|x|\to+\infty$. \smallskip

  {\bf Part 2:} {\it  $\tilde u_n$ is a classical solution and   $|\nabla  \tilde u_n|\leq \theta_n$ in $\R^2$ for some $\theta_n\in(0,1)$.  }
  For any $\varrho>R_0$, recall  that 
  $ \tilde u_{n,R}\to  \tilde u_n$ in  $C^{0,1} (\bar B_\varrho)(0)) \; {\rm as}\ \, R\to+\infty$, 
  then it follows by \cite[Lemma 1.3]{BS82} that 
  $ \tilde u_n$ is weakly spacelike and with respect to its own boundary values, solves the variational problem with mean curvature $g_n$, i.e. 
  $\tilde u_n$ is the maximizer of the energy functional 
 $$\cI_{n,\varrho}(w)= \int_{B_\varrho(0)}\Big(\sqrt{1-|\nabla w|^2} +w(x)g_n\Big)dx \quad {\rm for\ }\, w\in \bX_{\tilde u_n}(B_\varrho(0)), $$ 
 where  
 $$\bX_{\tilde u_n}(B_\varrho(0)):=\big\{v\in C^{0,1}(\bar B_\varrho(0)):\, v=\tilde u_n\ {\rm on}\ \partial B_R(0),\ |\nabla v|\leq 1\  {\rm a.e.\ in\ }\, B_\varrho(0)\big\}.  $$

 Set $\tilde \cQ_{n,\rho}=\{x\in \R^2:\ u_n(x)>\rho\}$ for $\rho<\Phi_{2,\alpha_0}(R_0)$,  then $\cQ_{n,\rho}$ is bounded and
$$\bigcup_{\rho>-\infty} \cQ_{n,\rho}=\R^2. $$
 It follows  from \cite[Theorem 3.6]{BS82}, for any $\rho>0$, there exists $\theta_\rho\in(0,1)$, such that 
 $|\nabla u_n|\leq \theta_1$ in $\tilde\cQ_{n,\rho}$. 
 Then $u_n$ is a classical solution of Eq.(\ref{eq 3.1-N}) by Lemma \ref{lm 2.1-reg-int}. 
 By the decay of $\tilde u_n(x)\to-\infty$ as $|x|\to \infty$, it follows  from Proposition \ref{pr grad-ext-2.1} that there exists $\theta_1\in(0,1)$ such that 
 $|\nabla \tilde u_n|\leq \theta_1$  in $\R^N\setminus B_{R_0}(0)$.
 As consequence, for some $\theta_n\in(0,1)$ $|\nabla \tilde u_n|\leq \theta_n$  in $\R^2$
    \smallskip
 
 For any $\varphi\in C_c^1(\R^2)$, there exists $n_0\geq R_0$ such that ${\rm supp}(\varphi)\subset B_{n_0}(0)$
 and for any $R\geq n_0$, there holds by (\ref{eee-1})
  \begin{align*} 
 \int_{\R^2}   \frac{\nabla \tilde u_{n,R}\cdot \nabla  \phi}{\sqrt{1-|\nabla \tilde u_{n,R}|^2}}  \, dx&= \int_{B_R(0)}   \frac{\nabla\tilde u_{n,R}\cdot \nabla  \phi}{\sqrt{1-|\nabla u_{n,R}|^2}}  \, dx
 \\[1mm]&=   \int_{B_R(0)}   g_n(x) \phi(x)dx=
   \int_{\R^2}  g_n(x) \phi(x)dx,
    \end{align*}
then  passing to the limit as $R\to+\infty$, we obtain (\ref{eee-2-2}).
   \hfill$\Box$\medskip
 

 \vskip2mm
Now we prove Theorem \ref{teo 1-fund}.
\vskip2mm

\noindent{\bf Proof of Theorem \ref{teo 1-fund}. } 
{\bf Existence: }  From Proposition \ref{pr 2.1-n-w-2}, let $\tilde u_n$ be the solutions of (\ref{eq 3.1-N-2}), 
$|\nabla \tilde u_n|<1$ in $\R^2$ and 
 \begin{equation} \label{uppp-1--2}
\Phi_{2, \alpha_0}(x)-R_0  \leq \tilde u_n(x)\leq  \Phi_{2, \alpha_0}(x)+R_0  \quad {\rm for}\ \, x\in \R^2, 
\end{equation}
then  there is $\tilde u_\infty\in   C^{0,\gamma}_{\loc}(\R^2)$ such that for $\gamma\in(0,1)$
\begin{align}\label{conv-13--2}
\tilde  u_{n}\to \tilde  u_{\infty}\quad{\rm in}\ C^{0,\gamma}_{\loc}(\R^2) \quad {\rm as}\ \, n\to+\infty.   
   \end{align}
 So we derive that 
 \begin{equation} \label{uppp-1--3}
\Phi_{2, \alpha_0}(x)-R_0  \leq \tilde u_\infty(x)\leq  \Phi_{2, \alpha_0}(x)+R_0  \quad {\rm for}\ \, x\in \R^2.
\end{equation}

   As the proof of {\bf Claim 1} in Theorem   \ref {pr 2.1}, we have that $\tilde u_\infty\in   C^{0,1}_{\loc}(\R^2)$,  and 
$$|\nabla \tilde  u_\infty |\leq 1\quad{\rm a.e.\ in}\ \, \R^2.$$

By Lemma \ref{max point}, there exists $p_{\bar j}$ for some $\bar j\in\{1,\cdots,m_0\}$ and a sequence $R_n$ such that $R_n\to+\infty$
and $u_{R_n}$ has maximum point at $p_{\bar j}$, i.e.
$$ u_{R_n}(p_{\bar j})=\max_{x\in B_{R}(0)}  u_R(x).  $$

Let
$$\tilde u_R(x)=u_R(x)-u_{R_n}(p_{\bar j}). $$
Since $\tilde u_n\geq \tilde u_{n,R}$ in $B_R(0)$ for any $R>\theta_0 R_0$,  then 
  $\tilde u_\infty\geq \tilde u_R$ in $ B_R(0).  $

Next we claim that   
   $$|\nabla \tilde u_{\infty}|< 1 \quad {\rm  in}\ \R^2\setminus\cP_{m_0}.  $$ 
 For $\rho<-1$, $-\rho$ large enough, denote
  $$\cQ_\rho=\{x\in\R^2:\, \tilde u_\infty(x)>\rho\}, $$
then by the bound (\ref{uppp-1--3}), $\bigcup_{\rho\in(-\infty,\rho_1)} \cQ_\rho=\R^2 $
and for any $R>\theta_0R_0$,  there is $\rho_1<-u_{R}(p_{\bar j})$ such that  for $\rho\in(-\infty,\rho_1)$
$$B_{\theta_0R_0}(0)\subset  \cQ_\rho.$$

Reset that  $\displaystyle  \cO_{\rho}=\cQ_\rho\setminus \cup_{j=1}^{m_0} B_{\frac1{-\rho}}(p_j)$ and 
\begin{equation}\label{eq 3.1-ind-1}
\cI_{\rho}(w)= \int_{\cO_\rho}\sqrt{1-|\nabla w|^2}  \, dx   \quad {for\ }\, w\in \bX_0(\cO_\rho)
 \end{equation} 
with 
$$\bX_0(\cO_\rho):=\Big\{v\in C^{0,1}(\bar \cO_\rho):\,  v=\tilde u_\infty\ {\rm on}\ \partial \cO_\rho,\ \, |\nabla v|\leq 1\ \ {\rm a.e.\ in\ }\cO_\rho\Big\}.  $$
Then $\tilde u_\infty$ is weakly spacelike and achieves the maximizer of $\cI_{\rho}$. \medskip

 Next    {\it for any $\sigma\in(-\infty,\sigma_0]$, there exists $\theta=\theta(\sigma)\in (0,1)$  such that  
$$|\nabla \tilde u_\infty |\leq \theta_\sigma\quad{\rm in}\ \ \cO_{\rho,m} ,$$
where $\cO_{\rho,m}=\cQ_{\rho,m}\setminus \bigcup^{m_0}_{j=1} B_{\sigma}(p_j) $  with  $\cQ_{\rho,m}$ being the component   containing $B_{\theta_0R_0}(0)$.   }
 \smallskip

 Let 
 $$\cK_s=\big\{\overline{xy}\subset  \cO_{\rho,m}:\, x,y\in\partial\cO_{\rho,m},\ \, x\not=y,\ \,  |\tilde u_\infty(x)-\tilde u_\infty(y)|=|x-y|    \big\}.  $$
 Our aim is to show $\cK_s=\emptyset$.

\smallskip

If not,  we choose $x_1,x_2\in \partial\cO_{\rho,m}$ such that  $|\tilde u_R(x_1)-\tilde u_R(x_2)|=|x_1-x_2|$.
$$\bL_{x_1x_2}= \big\{x_t: \text{for $t$ belongs a maximal interval of $\R$ such that } x_t    \in   B_R\setminus \cP_{m_0}   \big\} , $$
where $x_t=x_1+t(x_2-x_1)$. 
Let $\bar x_1,\bar x_2$ be the ends points of $\bL_{x_1x_2}$,   then  either $\bL_{x_1x_2}$ could be extended to cross the boundary $\partial  \cQ_{\rho,m}$ twice, i.e. $\bar x_1,\bar x_2\in \partial \cQ_{\rho,m}$  or $\overline{\bL}_{x_1x_2}$ cross the boundary $\partial \cQ_{\rho,m}$ once i.e. 
  $\bar x_1\in \partial \cQ_{\rho,m}, \bar x_2\in \cP_{m_0}$  or  $\bL_{x_1x_2}$ stops by two point in $\cP_{m_0}$ i.e. $\bar x_1,\bar x_2\in \cP_{m_0}$. 

We apply \cite[Theorem 3.2]{BS82} to obtain that 
$$
\tilde u_\infty(x_t)=\tilde u_\infty(x_1)+t|x_1-x_2|\quad {\rm for\ all}\  x_t\in \overline{\bL}_{x_1x_2},
$$
 Particularly, we have that 
 \begin{equation}\label{ind-1-t-2-2}
|\tilde u_\infty(\bar x_1)-\tilde u_\infty(\bar x_2)|= |\bar x_1-\bar x_2|.
 \end{equation}

 If $\bar x_1,\bar x_2\in \partial \cQ_{\rho,m}$, then 
 $$|\tilde u_\infty(\bar x_1)-\tilde u_\infty (\bar x_2)|=0<|\bar x_1-\bar x_2|,  $$
 which contradicts (\ref{ind-1-t-2-2}). 
 
  If $\bar x_1\in \partial \cQ_{\rho,m}, \bar x_2\in \cP_{m_0}$ and we can set 
  $$\bar x_1 \in \bL_{x_1x_2}\cap \partial \cQ_{\rho,m} \quad{\rm and}\quad \bar x_2 \in \cP_{m_0},  $$
 then  $\tilde u_\infty(\bar x_1) =\rho \ll \tilde  u_\infty(p_{\bar j})$,
  and by (\ref{uppp-1--3}), 
  $$\rho\geq -\frac{\alpha_0}{2\pi}\ln |\bar x_1|-C_0, $$
  where $C_0\geq 0$ is independent of $\rho$.  
Thus, we obtain that 
\begin{align*} 
|\bar x_1|-R_0 \leq |\bar x_1-\bar x_2|  &= |\tilde u_\infty(\bar x_1)-\tilde u_\infty (\bar x_2)|
  \leq  \frac{\alpha_0}{2\pi}\ln |\bar x_1| + C_0,   
\end{align*}
 then we get contradictions if $|\bar x_1|$ is large enough, which is equivalent $- \rho >1$ large enough. 
 
 If   $\bar x_1,\, \bar x_2\in \cP_{m_0}$, we have that 
 $$
\tilde u_\infty(\bar x_1)=\tilde u_\infty(\bar x_2)+|\bar x_1-\bar x_2|\quad {\rm for\ all}\  x_t\in \overline{\bL}_{x_1x_2}.
$$
Let
$$\tilde w_\alpha(x)=\Phi_{2, \alpha}(x-\bar x_1) +\tilde u_\infty(\bar x_1),\quad x\in \cQ_{\rho,m}, $$
then $w_\alpha(\bar x_1)=\tilde u_\infty(\bar x_1)$ and there exist  $\bar \alpha\geq \alpha_0$ such that 
$$w_{\bar \alpha}\leq 2\rho\quad {\rm on}\ \partial \cQ_{\rho,m}. $$
By the comparison principle, we have that 
$$\tilde u_\infty \geq  w_{\bar \alpha}\quad {\rm in}\ \cQ_{\rho,m}, $$
which implies that 
$$ w_{\bar \alpha}(\bar x_1)-w_{\bar \alpha}(\bar x_2)\geq \tilde u_\infty(\bar x_1)-\tilde u_\infty(\bar x_2)= |\bar x_1-\bar x_2|,  $$
 which is impossible. \smallskip

As a consequence, we obtain that  $\cK_s=\emptyset$ and it follows by \cite[Corollary 4.2]{BS82} 
that $\tilde u_\infty$ is strictly spacelike in $\cQ_{\rho,m}$ and  
there exists $\theta_\epsilon\in[0,1)$ such that 
 \begin{equation}\label{grandient-1-R-2}
 |\nabla \tilde u_\infty|\leq \theta_\epsilon\quad {\rm in}\ \, \overline \cQ_{\rho,m}.   
 \end{equation}
 and then 
 $\tilde u_\infty$ is a classical solution of 
 \begin{equation}\label{eq 3.1-cla-R-2}
 \left\{
\begin{array}{lll}
\displaystyle\ \,  \cM_0 u= 0\ \ 
  {\rm in}\  \, \R^2 \setminus\cP_{m_0}, 
\\[2mm]
 \phantom{   }
 \displaystyle\lim_{|x|\to+\infty} u(x)=-\infty.  
 \end{array}
 \right.
 \end{equation}

Now we start to prove some qualitative properties of $\tilde u_\infty$. The nmain steps are the same as for case $N \geq 3$. 

{\bf Part 1: }   We show that   $\tilde u_\infty$ is a weak solution of the following problem:
 \begin{equation}\label{eq 3.1-ral-2}
 \left\{
\begin{array}{lll}
\displaystyle  \cM_0 u= \sum_{j=1}^{m_0} \alpha_j\delta_{p_j}\quad \    {\rm in}\  \,   \cD'(\R^2), 
\\[3mm]
 \phantom{    }
 \displaystyle\lim_{|x|\to+\infty} u(x)=-\infty.     
 \end{array}
 \right.
 \end{equation}

Fix $\bar R>\theta_0R_0$ and for  $\rho<0$, denote
  $$\cQ_{n,\rho}=\{x\in\R^2:\, \tilde u_n(x)>\rho\}, $$
then for $R>\bar R$,  there is $\rho_2\leq -u_{R_n}(p_{\bar j})$ such that  for $\rho\in(0,\rho_2)$
$$B_{\bar R}(0)\subset  \cQ_{n,\rho}.$$
 
 Let $w_{n,\rho}=\tilde u_{n}-\rho$, then it is the solution of 
\begin{equation}\label{eq 3.1-N-rho-1}
 \left\{
\begin{array}{lll}
\displaystyle  \cM_0 w_{n,\rho}= \sum^{m_i}_{j=1} \alpha_{i,j} \delta_{p_j}\quad
 & {\rm in}\  \   \cQ_{n,\rho}, 
\\[4mm]
 \phantom{ -\ \,   }
 \displaystyle w_{R,\rho}=0 \quad
 & {\rm on}\  \, \partial \cQ_{n,\rho}.   
 \end{array}
 \right.
 \end{equation}
 Taking the test function $\phi=w_{n,\rho}$ in (\ref{eee-1}) to derive that
 \begin{equation}  
 \int_{\cQ_{n,\rho}} \frac{|\nabla w_{n,\rho}|^2   }{\sqrt{1-|\nabla w_{n,\rho}|^2}} dx=  \sum^{m_0}_{j=1}\alpha_jw_{n,\rho}(p_j)\leq -\Phi_{2,\alpha_0}(\bar r)\alpha_0, \label{bound-1-2}
    \end{equation}
where $\bar r>\bar R$ such that  $\cQ_{n,\rho}\subset B_{\bar r}(0)$.

Firstly, we show the uniformly bound that 
 \begin{align*} 
  \int_{B_{\bar R}(0)} \frac{|\nabla \tilde u_{n} |   }{\sqrt{1-|\nabla \tilde u_{n}|^2}} dx &\leq   \int_{\cQ_{n,\rho}}  \frac{|\nabla w_{n,\rho} |   }{\sqrt{1-|\nabla w_{n,\rho}|^2}} dx 
   \\[1mm]& =   2   \int_{\cQ_{n,\rho}\cap \{|\nabla w_{n,\rho}|\geq \frac12\}} \frac{|\nabla w_{n,\rho}|^2   }{\sqrt{1-|\nabla w_{n,\rho}|^2}} dx
 \\&\qquad +   \int_{\cQ_{n,\rho} \cap \{|\nabla  w_{n,\rho}|<\frac12\}} \frac{|\nabla w_{n,\rho}|   }{\sqrt{1-|\nabla w_{n,\rho}|^2}} dx
  \\[1mm]& \leq   \Big(-2 \Phi_{2,\alpha_0}(\bar r)\alpha_0 +\frac{\sqrt{3}}3 |\cQ_{R,\rho} |\Big). 
  \end{align*}

For any $\varphi\in C^{0,1}_c(B_{\bar R}(0))$ such that $\varphi(x)=\varphi(p_j)$ for $x\in B_\epsilon(p_j)$ for any $j=1,\cdots, m_0$
and $\epsilon>0$ small, then ${\rm supp}(\nabla \varphi)\subset B_{\bar R}(0)\setminus \bigcup^{m_0}_{j=1} B_\epsilon(p_j)$ 
 \begin{align*} 
 \int_{B_{\bar R}(0) \setminus\big( \bigcup^{m_0}_{j=1} B_\epsilon(p_j)\big)} \frac{ \nabla \tilde u_{n} \cdot \nabla \varphi   }{\sqrt{1-|\nabla \tilde u_{n}|^2}} dx& =  \int_{B_R(0)} \frac{ \nabla \tilde u_{n} \cdot \nabla \varphi   }{\sqrt{1-|\nabla \tilde u_{n}|^2}} dx 
 = \sum^{m_0}_{j=1} \alpha_j \varphi(p_j)
 \\[1mm]&=  \int_{\R^2} \frac{ \nabla \tilde u_{\infty} \cdot \nabla \varphi   }{\sqrt{1-|\nabla \tilde  u_{\infty}|^2}} dx
 \\[1mm]&= \int_{B_{\bar R}(0)\setminus\big( \bigcup^{m_0}_{j=1} B_\epsilon(p_j)\big)} \frac{ \nabla \tilde u_{\infty} \cdot \nabla \varphi   }{\sqrt{1-|\nabla\tilde u_{\infty}|^2}} dx
  \end{align*}
and we have that as $n\to+\infty$
 \begin{align} \label{con-1-R-N}
  \frac{ \nabla \tilde u_{n}    }{\sqrt{1-|\nabla \tilde u_{n}|^2}} \to\frac{ \nabla \tilde u_{\infty}    }{\sqrt{1-|\nabla \tilde u_{\infty}|^2}} \quad{\rm weakly\ in\ }\big( L^1 \big(B_{\bar R}(0)\setminus\big( \bigcup^{m_0}_{j=1} B_\epsilon(p_j)\big), \R^2\big), 
  \end{align}
then by the upper semicontinuity of the area integral
\begin{align*} 
\int_{B_{\bar R}(0)\setminus\big( \bigcup^{m_0}_{j=1} B_\epsilon(p_j)\big)} \frac{|\nabla \tilde u_{\infty} |   }{\sqrt{1-|\nabla \tilde u_{\infty}|^2}} dx&\leq   \liminf_{n\to+\infty}  \int_{B_{\bar R}(0)\setminus\big( \bigcup^{m_0}_{j=1} B_\epsilon(p_j)\big)} \frac{|\nabla \tilde u_{n} |   }{\sqrt{1-|\nabla \tilde u_{n}|^2}} dx 
\\[1mm]&  \leq    \Big(-2 \Phi_{2,\alpha_0}(\bar r)\alpha_0 +\frac{\sqrt{3}}3 |\cQ_{R,\rho} |\Big),
  \end{align*}
which,  by the arbitrary of $\epsilon>0$,  implies that 
\begin{align*} 
\int_{B_{\bar R}(0)} \frac{|\nabla\tilde  u_{\infty} |   }{\sqrt{1-|\nabla\tilde  u_{\infty}|^2}} dx  \leq    \Big(-2 \Phi_{2,\alpha_0}(\bar r)\alpha_0 +\frac{\sqrt{3}}3 |\cQ_{R,\rho} |\Big).
  \end{align*}
As a consequence, by the arbitrary of $\bar R>\theta_0R_0$, we have that  
$$\frac{|\nabla \tilde u_{\infty} |   }{\sqrt{1-|\nabla\tilde u_{\infty}|^2}}\in L^1_{\loc} (\R^2)$$ 
and we obtain that $\tilde  u_\infty \in \bX_\infty(\R^2)$,
  where
   $$\bX_\infty(\R^2) =\Big\{w\in C^{0,1}(\R^2):\, |\nabla w|<1\ {\rm\ in}\ \R^2 \setminus \cP_{m_0},\ \, \frac{|\nabla w|}{\sqrt{1-|\nabla w|^2}}  \in L^1_{\loc}(\R^2)\Big\}. $$
   Moreover, from (\ref{con-1-R-N}), we get that for any $\epsilon>0$ small, 
 $$\int_{\cO_\epsilon}  \frac{ \nabla  \tilde u_{\infty} \cdot \nabla \varphi   }{\sqrt{1-|\nabla  \tilde u_{\infty}|^2}}dx=0\quad \text{  for any $\varphi\in C^{0,1}_c(\R^2)$ with ${\rm supp}(\varphi)\subset \cO_\epsilon:=\R^2\setminus \bigcup_{j=1}^{m_0}B_\epsilon(p_j)$. } $$
 By (\ref{grandient-1-R-2}) and  \cite[Theorem 3.6]{BS82},   $\tilde u_\infty\in C^{2,\gamma}_{\loc}(\cO_{2\epsilon})$, by the arbitrary of $\epsilon$, we get that  
$\tilde u_\infty$ verifies the equation (\ref{eq 3.1-cla-R-2}) in the classical sense.

Now we take $\xi\in C^{0,1}_c(\R^2)$  with  ${\rm supp}( \xi)\subset   \R^2 \setminus\cP_{m_0}$
and 
   $$\int_{\R^2}  \frac{\nabla \tilde u_\infty\cdot\nabla \xi}{\sqrt{1-|\nabla \tilde u_\infty|^2}} \,dx =0\quad \text{  for any $\xi\in C^{0,1}_c (\R^2\setminus\cP_{m_0})$.} $$ 
   Now we apply Proposition \ref{cr 2.1-iso} to obtain that $\tilde u_\infty$ is a weak solution 
\begin{equation}\label{eq 1.1-R-w-2}
  \cM_0 u= \sum_{j=1}^{m_0} k_{p_j}  \delta_{p_j}   \quad
  {\rm in}\  \, \cD'\big(B_R(0)\big)
  \end{equation}
for some $k_{p_j}\in \R$.  That is, 
\begin{equation}
\int_{ \R^2 }  \frac{\nabla \tilde u_\infty\cdot\nabla \xi}{\sqrt{1-|\nabla \tilde u_\infty|^2}} \,dx =\sum_{j=1}^{m_0}  k_{p_j}  \xi(p_j),\quad\, \forall \xi\in C^{0,1}_c (\R^2).
 \end{equation}
 

Again we can prove that $$k_{p_j} =\alpha_j\quad\text{ for  } j=1,\cdots,m_0$$ and 
$u_\infty$ is a weak solution (\ref{eq 1.1-R-w-N}). \smallskip

{\bf Asymptotic behavior at poles: }  At the Dirac poles with positive multiplicities, we can obtain that 
$\tilde u_\infty$ is light-cone singular at $\cP_{m_0}$ with  the behavior
$$\lim_{|x-p_j|\to0^+}|\nabla \tilde u_\infty(x)|=1.  $$
Moreover, the  vertex of the cone is upwards, i.e.  $p_j$ isn't a local minimal point of $\tilde u_\infty$.  \smallskip

{\bf Asymptotic behavior at infinity: }  Recall  that 
 $${\rm Res}[\tilde u_{\infty}]=\int_{\partial B_R} \frac{\nabla \tilde u_\infty (x)}{\sqrt{1-|\nabla \tilde u_\infty(x)|^2}}\cdot \frac{x}{R}dH_{1}(x). $$
 
Since $\tilde u_\infty\leq 0$ in $\R^2$, then  we have that $\vec{a}=0$. 
Next we compute the residue:  for $R>R_0$ and $\epsilon\in(0,\frac12\min\{|p_j-p_{j'}|,\ j\not=j'\})$ 
\begin{align*}
0&=\int_{B_R(0)\setminus \big(\cup^{m_0}_{j=1} B_\epsilon(p_j)\big)} \cM_0 \tilde u_{\infty}(x) dx
\\&=\int_{\partial B_R(0)} \frac{\nabla \tilde u_\infty (x)}{\sqrt{1-|\nabla \tilde u_\infty(x)|^2}}\cdot \frac{x}{R}dH_{1}(x)-
\sum^{m_0}_{j=1}\int_{\partial B_\epsilon(p_j)} \frac{\nabla \tilde u_\infty (x)}{\sqrt{1-|\nabla \tilde u_\infty(x)|^2}}\cdot \frac{x-p_j}{|x-p_j|}dH_{1}(x)
\\&\to   {\rm Res}[\tilde u_{\infty}]+\alpha_0\quad {\rm as}\ \, \epsilon\to0^+, 
 \end{align*} 
which implies that 
$${\rm Res}[\tilde u_{\infty}]=- \alpha_0.$$

Thus, it follows by  (\ref{bb1-1})    that 
\begin{align}\label{bh 1-2} 
 \tilde u_{\infty}(x)=  -\frac{\alpha_0}{2\pi} \ln |x|   +c+o(1)\quad {\rm as}\ \, |x|\to+\infty. 
 \end{align}

 {\bf Uniqueness: } If Eq.(\ref{eq 1.1-fund}) has  two solution 
 such that 
  $$w_i(x)=  -\frac{\alpha_0}{2\pi} \ln |x|   +c+o(1)\quad {\rm as}\ |x|\to+\infty,$$
 then for any $\epsilon>0$
 $w_1\pm \epsilon$ is a solution Eq.(\ref{eq 1.1-fund}) and by comparison principle, 
$ w_1-\epsilon\leq w_2\leq w_1+\epsilon$ in $\R^2$.  The arbitrary of $\epsilon>0$ implies that  
$w_1\leq w_2\leq w_1$ in $\R^2$. The uniqueness follows and this finishes the proof of Theorem \ref{teo 1-fund}.    \hfill$\Box$\medskip

\section{Extension models}

\subsection{Model with positive and negative Dirac masses}

Reset
$$g_n=g_{n,+}-g_{n,-}$$
where
$$
g_{n,+}(x)= \sum^{m_1}_{j=1} \alpha_j \eta_n(x-p_j), \quad g_{n,-}(x)= \sum^{m_0}_{j=m_1+1} \beta_j \eta_n(x-p_j) \quad {\rm for}\ \, x\in\R^N, 
$$
then $\{g_n\}_{n\in\N}, \{g_{n,\pm}\}_{n\in\N}\in C^2(\R^2)$ are   sequences of smooth functions  such that 
$${\rm supp}(g_n)\subset\, \bigcup_{j=1,\cdots, m_0} B_{\frac2{n}}(p_j),   $$
 $$\lim_{n\to +\infty} g_{n,+}= \sum^{m_1}_{j=1} \alpha_j \delta_{p_j},\quad  \lim_{n\to +\infty} g_{n,-}= \sum^{m_0}_{j=m_1+1} \beta_j \delta_{p_j}$$
and then 
 $$\lim_{n\to +\infty} g_{n}=\sum^{m_1}_{j=1} \alpha_j \delta_{p_j}-\sum^{m_0}_{j=m_1+1} \beta_j \delta_{p_j}\  \text{ in the distributional sense.}   $$

To prove Theorem \ref{teo 1-fund+-},  we need  to consider  the proximation problem 
\begin{equation}\label{eq 3.1+-}
 \left\{
\begin{array}{lll}
\displaystyle  \cM_0 u= g_n \quad
  {\rm in}\  \   \R^N, 
\\[3mm]
 \phantom{    }
 u(x)=0 \quad  {\rm on}\  \   \partial \R^N,
 \end{array}
 \right.
 \end{equation}
with $N\geq 3$.

 \begin{proposition}\label{pr 2.1+-}

If $N\geq 3$,    then problem (\ref{eq 3.1+-}) has unique classical solution $u_{n}\in C^{2,\gamma}( \R^N)$ with $\gamma\in(0,1)$.  
 Moreover, 
 
 $(i)$   $u_n$ is   the minimizer  of the energy functional  
  \begin{equation}\label{eq 3.1-do 2+-}
\cJ_{\infty}(w)= \int_{\R^N}\big(1- \sqrt{1-|\nabla w|^2}\big)\, dx - \int_{\R^N}g_n wdx \quad {for\ }\, w\in \bX_\infty (\R^N),   
 \end{equation} 
 where 
 $$\bX_\infty (\R^N)=\{v\in C^{0,1}(\R^N): |\nabla v|\leq 1\ {\rm a.e.\ in}\ \R^N, \ \,  \int_{\R^N}\big(1- \sqrt{1-|\nabla w|^2}\big)dx<+\infty \big\}. $$
 
 $(ii)$ There holds $-u_{n,-}\leq u_n\leq u_{n,+}$ in $ \R^N$, where $u_{n,\pm}$ are the positive solutions of 
\begin{equation}\label{eq 3.1+-iso+}
 \left\{
\begin{array}{lll}
\displaystyle \  \cM_0 u=\pm g_{n,\pm} \quad
 & {\rm in}\  \   \R^N, 
\\[4mm]
 \phantom{    }
 \displaystyle \lim_{|x|\to+\infty}u(x)=0.
 \end{array}
 \right.
 \end{equation}

 \end{proposition}
\noindent{\bf Proof. }  
It follows by \cite[Theorem 4.1]{BS82} or \cite[Corollary 4.3]{BS82} that 
$$
 \left\{
\begin{array}{lll}
\displaystyle  \cM_0 u=g_n\ \  (g_{n,\pm}\ {\rm resp.}) \quad
 & {\rm in}\  \   B_R(0), 
\\[2mm]
 \phantom{ -\ \,   }
 \displaystyle u=0 \quad
 & {\rm on}\  \, \partial B_R(0)  
 \end{array}
 \right.
$$
has a unique classical solution $w_{n,R}$ ($w_{n,R,\pm}$ respectively). Note that $w_{n,R}$ is the maximizer  of 
$$\cI_{R}(w)= \int_{B_R(0)} \Big(\sqrt{1-|\nabla w|^2}\, + g_n w \Big) dx \quad {for\ }\, w\in \bX_R(B_R(0)). $$
Moreover,  by the comparison principle, we have that 
\begin{equation}\label{bound+-iso-n+-0}
-w_{n,R,-}\leq w_{n,R}\leq w_{n,R,+}\quad\text{ in $B_R(0)$.}
 \end{equation} 
Note that when $N\geq 3$, 
\begin{equation}\label{bound+-iso-n+-1}
w_{n,R,+}\leq \Phi_{N,\alpha_0}(0)\quad {\rm and}\quad w_{n,R,-}\leq \Phi_{N,\beta_0}(0). 
 \end{equation} 
 As the proof of Proposition \ref{pr 2.1-n-w},   there is $u_{n}\in C^{0,1}(\R^N)$ such that $|\nabla u_n|\leq 1$ and for some $\gamma\in(0,1)$ 
$$w_{n,R}\to u_n\quad {\rm in}\ C^{0,\gamma}_{\loc}(\R^N) \quad {\rm as}\ \, R\to+\infty. $$
It follows by \cite[Lemma 1.3]{BS82} that $u_n$ is a weak and classical solution of (\ref{eq 3.1+-}) and 
by (\ref{bound+-iso-n+-0})
 \begin{equation}\label{bound+-iso-n+-}
-u_{n,-}\leq u_{n}\leq u_{n,+}\quad\text{ in } \R^N. 
 \end{equation}
We complete the proof.    \hfill$\Box$\medskip


\noindent{\bf Proof of Theorem \ref{teo 1-fund+-}. }  
It follows by Proposition \ref{pr 2.1+-} that (\ref{eq 3.1+-}) has a unique weak solution $u_n$, 
$$-u_{n,-}\leq  u_n\leq u_{n,+} \quad {\rm in}\  \R^N$$
and from the proof of Theorem \ref{teo 1-fund3}, there hold
$$0\leq u_{n,+}\leq \min\{\Phi_{N,\alpha_0}(0), \Phi_{N,T_0\alpha_0}(x)\},\quad  0\leq u_{n,-}\leq \min\{\Phi_{N,\beta_0}(0), \Phi_{N,T_0\beta_0}(x)\}, $$
and the limits as $n\to+\infty$, denoting  $u_{\infty,\alpha_0},\, u_{\infty, \beta_0}$ respectively, 
which  are the positive solutions of 
\begin{equation}\label{eq 3.1+-iso+N}
 \left\{
\begin{array}{lll}
\displaystyle  \cM_0 u=\sum^{m_1}_{j=1}\alpha_j\delta_{p_j} \quad
 & {\rm in}\  \   \R^N, 
\\[4.5mm]
 \phantom{     }
 \displaystyle\lim_{|x|\to+\infty} u(x)=0    
 \end{array}
 \right.
 \end{equation} 
and 
\begin{equation}\label{eq 3.1+-iso-N}
 \left\{
\begin{array}{lll}
\displaystyle  \cM_0 u=  \sum^{m_0}_{j=m_1+1}\beta_j\delta_{p_j} \quad
 & {\rm in}\  \   \R^N, 
\\[4.5mm]
 \phantom{   }
 \displaystyle\lim_{|x|\to+\infty} u(x)=0.       
 \end{array}
 \right.
 \end{equation} 
 Together with the fact that $|\nabla u_R|\leq 1$ a.e., then   there  is $ u_{\infty}\in C^{0,1}_{\loc}(\R^N)$ such that
 $$  |\nabla u_{\infty}|\leq 1\quad {\rm a.e.\ in}\ \ \R^N$$
 and     for $\gamma\in(0,1)$
 $$ u_{n}\to   u_{\infty}\quad{\rm in}\ C^{0,\gamma}_{\loc}(\R^N)\quad {\rm as}\ \, n\to+\infty.   $$
Observe that 
$$ -u_{\infty,\beta_0}\leq u_\infty\leq u_{\infty,\alpha_0}\quad{\rm in}\ \, \R^N$$
and $u_{\infty,\beta_0}$ and $u_{\infty,\alpha_0}$ decays with the rate of $|x|^{2-N}$ at infinity.

 \vskip2mm
Next we show that for any $\rho>0$ small, there exists $\theta_\rho\in(0,1)$ such that 
$$|\nabla u_\infty|<\theta_\rho\quad{\rm  in}\ \, \cO_\rho\setminus \bigcup_{j=1}^{m_0} B_{\rho}(p_j), $$
where $ \cO_\rho=\{x\in\R^N:\, u_\infty(x)>\rho \}\cup \{x\in\R^N:\, u_\infty(x)<-\rho \}$, which is bounded.   
Let 
$$\cK_{0}=\big\{\overline{p_ip_j}: i\not=j,|\tilde u_\infty(p_i)-\tilde u_\infty(p_j)|=|p_i-p_j|\big\}, $$
then we show $\cK_0=\emptyset$. If not, there is a contradiction with  (\ref{con-classic-1}). Now we can show 
$$\cK_{s}:=\big\{\overline{x_1x_2}: x_1,x_2\in \partial \cO_\rho,  x_1\not=x_2,|\tilde u_\infty(x_1)-\tilde u_\infty(x_2)|=|x_1-x_2|\big\}=\emptyset, $$
which leads to our argument by our previous proof.

Next we show $\cup_{\rho>0}\cO_\rho=\R^N\setminus \{x\in\R^N: u_\infty(x)=0\}^o$, where $\{x\in\R^N: u_\infty(x)=0\}^o$
is the set of  the inner points in  $\{x\in\R^N: u_\infty(x)=0\}$.  Obviously, $\cM_0 u_\infty=0$ in $\{x\in\R^N: u_\infty(x)=0\}^o$ in 
the classical sense. This means, $\cP_{m_0}\cap \{x\in\R^N: u_\infty(x)=0\}^o=\emptyset$.

As a consequence,  we can show that $u_\infty$ is a weak solution of 
$$
 \left\{
\begin{array}{lll}
\displaystyle   \cM_0 u = \sum^{m_1}_{j=1}\alpha_j\delta_{p_j} - \sum^{m_0}_{j=m_1+1}\beta_j\delta_{p_j}  \quad
  {\rm in}\  \, \cD'(\R^N), 
\\[5mm]
 \phantom{    }
 \displaystyle \lim_{ |x|\to+\infty}u(x)=0       
 \end{array}
 \right.
$$
and a classical solution of 
$$  \cM_0 u = 0\quad  {\rm in}\  \,  \R^N\setminus \cP_{m_0}. $$
We omit the detailed proof.  
 
 Furthermore, direct computation shows that 
$$\int_{\partial B_R(0)}\frac{\nabla u_R}{\sqrt{1-|\nabla u|^2} } \cdot  \frac{x}{|x|}dH_{1} =\alpha_0-\beta_0,   $$
then 
$$u_{\infty}(x)=\frac1{|\partial B_1(0)|} \big(\alpha_0-\beta_0 \big) |x|^{2-N}   +O(|x|^{1-N})\quad {\rm as}\ |x|\to+\infty.  $$
 The remainder arguments are standard. \hfill$\Box$\medskip

\begin{remark}
The condition (\ref{con-classic-1}) is used to rule out the case that 
$$u(x_1)=u(x_2)+|x_1-x_2|\quad {\rm for}\ \,  x_1\in \cP_{m_1,+}, x_2\in \cP_{m_2,-}$$
which guarantee the regularity $|\nabla u|<1$ in $\R^N\setminus \cP_{m_0}$.  From the proof,  it could be replaced by a shaper condition 
$$ u_{N, \alpha_0}(p)+u_{N, \beta_0}(q) >|p-p|\quad {\rm for}\ \, p\in \cP_{m_1,+},\ q\in \cP_{m_2,-}, $$
where $u_{N, \alpha_0}, u_{N, \beta_0}$ are positive solutions of (\ref{eq 3.1+-iso+N}) and (\ref{eq 3.1+-iso-N}) respectively.

\end{remark}

 \subsection{Model with infinitely many light-cones }
 
  Denote  $\cP_{\infty}$  the set of  the light-cone  singularities 
 \begin{equation}\label{sing set-inf1}
 \cP_{\infty}=\Big\{p_j\in\R^N\!: j\in\N,\ |p_{j}-p_{j'}|>0\ {\rm for}\ j\not=j'\   \Big\}. 
  \end{equation}
 We construct  the Hypersufaces having infinitely many light-cones with vertices  $ \cP_{\infty}$ by considering the equation
 \begin{equation}\label{eq 1.1-fund-inf}
 \left\{
\begin{array}{lll}
\displaystyle\ \, \cM_0 u= \sum^{\infty}_{j=1} \alpha_j \delta_{p_j}\quad
 & {\rm in}\  \,   \R^N, 
\\[4mm]
 \phantom{   }
 \displaystyle \lim_{|x|\to+\infty}u(x)=0,
 \end{array}
 \right.
  \end{equation}
 where $\alpha_j>0$ and $p_j\in \cP_{\infty}$.

 {\it Here a function $u$ is said to be a weak solution of (\ref{eq 1.1-fund-inf}) if 
$u\in C^{0,1}_{\loc}(\R^N)\cap C^2_{\loc}(\R^N\setminus \cP_{\infty})$ such that 
$\frac{|\nabla u|}{\sqrt{1-|\nabla u|^2}}\in L^1_{\loc}(\R^N)$, $\displaystyle \lim_{ |x|\to+\infty}u(x)=0$ and 
 \begin{equation}\label{fun 1 weak form-inf}
\int_{\R^N} \frac{\nabla u(x) \cdot \nabla \phi(x)}{\sqrt{1-|\nabla u(x)|^2}} dx=\sum^{\infty}_{j=1} \alpha_j\phi(p_j)
\quad{\rm for\ any}\ \, \phi\in C^{0,1}_c(\R^N).  
\end{equation}}

 Let us first consider the case in which $\cP_{\infty}$ is bounded  and  has only one cluster point.

 \begin{theorem}\label{teo 1-fund-inf1}
 Let   $N\geq 3$, $\cP_{\infty}\subset B_{\frac{1}{2}R_0}(0)$ be  given in (\ref{sing set-inf1}) satisfy
  $\displaystyle \lim_{j\to+\infty}p_j= {\bf p}$   and 
 $$    \alpha_j>0 \quad{\rm and}\quad   \alpha_\infty=\sum^{\infty}_{j=1} \alpha_j<+\infty, $$
 then  Eq.(\ref{eq 1.1-fund-inf}) has a minimal positive solution $u_{b,\infty}\in C^2(\R^N\setminus \cP_{\infty})\cap C^{0,1}(\R^N)$  satisfying that 
 $\cP_{\infty}$ is the set of  light-cone singularities of $u_{b,\infty}$ and 
 $$  u_{b,\infty}(x)  = c_N \alpha_\infty   |x|^{2-N}+O( |x|^{1-N})\quad {\rm as}\ \, |x|\to+\infty$$ 
where 
$c_N=\frac{\Gamma(\frac{N}{2})}{2\pi^{\frac N2}}=\frac1{|\partial B_1(0)|}$.  Furthermore,  we have

 \smallskip

$(a)$ there exist $\lambda_j\in\R$ with $j=1,\cdots, m_0$ such that 
$$\lim_{|x-p_j|\to0^+}u_{b,\infty}(x)=\lambda_j$$
and
$$ |\lambda_j-\lambda_{j'}|< |  p_j-p_{j'}| \quad{\rm for}\ j\not=j'. $$
 
 
 $(b)$ The function $u_{b,\infty}$ verifies 
 \begin{equation}\label{eq 1.1-fund-cl-inf1}
 \left\{
\begin{array}{lll}
\displaystyle  \cM_0 u =0  \quad
  {\rm in}\  \, \R^N\setminus \overline\cP_{\infty}, 
\\[3mm]
 \phantom{    }
 \displaystyle \lim_{|x|\to+\infty}u(x)=0.  
 \end{array}
 \right.
 \end{equation}
 
 $(c)$  $u_{b,\infty}$ is the minimizer of the functional  
 $$
\cJ_\infty(w)= \int_{\R^N}\big(1- \sqrt{1-|\nabla w|^2}\big)\, dx -\sum_{j=1}^{\infty}\alpha_j w(p_j)  \quad {for\ }\, w\in \bX_\infty(\R^N).
$$
\end{theorem}

In order to prove Theorem \ref{teo 1-fund-inf1}, we need the following auxilary  lemma. 

  \begin{lemma}\label{lm uniform b-N-inf}
Assume that  $N\geq 3$ and  $n\in\N$
and 
$$\tilde g_n(x)= \sum^n_{j=1} \alpha_j \eta_n(x-p_j)\quad {\rm for}\ \, x\in\R^N,  $$
where $p_j\in\cP_\infty\subset B_{\frac12 R_0}(0)$.

    Let $u_{n}$ be the unique classical solution of
\begin{equation}\label{eq 3.1-inf-n}
 \left\{
\begin{array}{lll}
\displaystyle  \cM_0 u=\tilde  g_n\quad\ 
   {\rm in}\  \,   \R^N, 
\\[3mm]
 \phantom{    }
 \displaystyle \lim_{|x|\to+\infty}u(x)=0.   
 \end{array}
 \right.
 \end{equation}
 If 
 $$\alpha_\infty:=\displaystyle \sum_{j=1}^\infty \alpha_j<+\infty, $$ 
 then 
  there exists $\bar \alpha\geq \alpha_0$ independent of $n, R$ such that 
\begin{align}\label{uni-0-N}
u_{n}(x)\leq    \min\{\Phi_{N,\alpha_\infty}(0), \Phi_{N,T_0\alpha_\infty}(x)\}\quad {\rm for\ all}\  x\in \R^N
\end{align}
and for any $j=1,\cdots,n $, there holds
\begin{align}\label{uni-1-N-inf}
u_{n}(x)\geq c\alpha_j (1+|x|)^{2-N}\quad {\rm for\ all}\  x\in \R^N, 
\end{align}
where
 $R_j=R-|p_j|$ and $c>0$.
 \end{lemma}

\noindent{\bf Proof. }  
 From the proof of Theorem \ref{teo 1-fund3}, there hold
$$0\leq u_{n,R}(x)\leq \min\{\Phi_{N,\alpha_\infty}(0), \Phi_{N,T_0\alpha_\infty}(x)\},\quad x\in B_R(0), $$
which implies that 
$$0\leq u_{n}\leq \min\{\Phi_{N,\alpha_\infty}(0), \Phi_{N,T_0\alpha_\infty}(x)\}, \quad x\in\R^N.$$
On the other hand, we have that 
$$
 u_{n,R, j}(x) =\int_{|x|}^{R} \frac{ ( c_N^{-1}\alpha_j  )  }{\sqrt{r^{2(N-1)}+ ( c_N^{-1}\alpha_j  )^2}}dr  \quad {\rm for\ all}\  x\in B_R(0), 
$$
where
 $R_j=R-|p_j|$.
 So we derive that 
 $$ u_{n}(x)\geq u_{n,j}\geq c\alpha_j (1+|x|)^{2-N}.  $$
We omit the remainder proof.  \hfill$\Box$\bigskip

\noindent{\bf Proof of Theorem \ref{teo 1-fund-inf1}. }  From  the proof of Theorem \ref{teo 1-fund3}, for the integer $n\geq 1$,   problem (\ref{eq 3.1-inf-n})
has a unique solution  $u_{n}$ 
satisfying 
$$c\alpha_j (1+|x|)^{2-N}\leq u_{n}(x) \leq \min\{\Phi_{N,\alpha_\infty}(0), \Phi_{N,T_0\alpha_\infty}(x)\}.  $$
Then  we can obtain that 
the mapping $n\mapsto u_{n} $  is  bounded by $ \Phi_{N,\bar \alpha }$ in $\R^N$ and $|\nabla u|<1$ in $\R^N$
then  for some $\gamma\in(0,1)$ 
$$
 u_{n}\to u_{\infty}\quad{\rm in}\ C^{0,\gamma}_{\loc}(\R^N)\quad {\rm as}\ \, n\to+\infty .  
$$

 As we shown before,     $u_{\infty}$ is the solution of 
\begin{equation}\label{eq 3.1-inf-infty}
 \left\{
\begin{array}{lll}
\displaystyle\ \, \cM_0 u= \sum^{\infty}_{j=1} \alpha_j \delta_{p_j}\quad
 & {\rm in}\  \,   \R^N, 
\\[4mm]
 \phantom{   }
 \displaystyle \lim_{|x|\to+\infty}u(x)=0. 
 \end{array}
 \right.
 \end{equation}
 The regularity could be shown in $\R^N\setminus\overline{ \cP_{\infty}}$, so $u_\infty$ is a classical solution of (\ref{eq 1.1-fund-cl-inf1}). 
 
  The remainder proof follows the same steps as before and we omit it. \hfill$\Box$

\begin{remark}
From our proof, it is natural to extend our results to the equations with  the settings either where the Dirac points possess finitely many cluster points 
i.e. $\overline\cP_\infty\setminus \cP_\infty$ is finite,  or where the coefficients  of  the Dirac masses change signs, i.e.
$$
 \left\{
\begin{array}{lll}
\displaystyle\ \, \cM_0 u= \sum^{\infty}_{j=1} \alpha_j \delta_{p_j}- \sum^{\infty}_{j=1} \beta_j \delta_{q_j}\quad
 & {\rm in}\  \,   \R^N, 
\\[4mm]
 \phantom{   }
 \displaystyle \lim_{|x|\to+\infty}u(x)=0, 
 \end{array}
 \right.
$$
where $ \alpha_j, \beta_j >0$.

\end{remark}

  \bigskip
{ \footnotesize
\noindent {\bf  Conflicts of interest:} The authors declare that they have no conflicts of interest regarding this work.
\medskip

\noindent {\bf  Data availability:}  This paper has no associated data.\medskip

\noindent{\bf Acknowledgements:}  
 H. Chen is supported by  the National Natural Science Foundation of China, No.  12361043.\smallskip
 
Y. Wang is supported by the National Natural Science Foundation of China, No. 12461041,
 by the Natural Science Foundation of Jiangxi Province,
 Nos. 20232ACB211001 and 20252BAC240158.\smallskip
 
F. Zhou is supported by Science and Technology Commission of Shanghai Municipality, No. 22DZ2229014
and also he National Natural Science Foundation of China, No. 12071189. }

 \end{document}